\documentclass[11pt,a4paper]{amsart} %ARXIV

\usepackage[utf8]{inputenc}
\usepackage[T1]{fontenc} 
\usepackage{graphicx}
\usepackage{hyperref} \usepackage{caption}
\usepackage{csquotes}
\usepackage{xspace}
\usepackage{amsmath,amsfonts}
\usepackage{bbm}
\usepackage[bb=boondox]{mathalfa}

\usepackage{tikz}
\usepackage{pgf,pgfplots}
\usetikzlibrary{calc}
\pgfplotsset{compat=1.13}
\usepackage{mathrsfs}
\usetikzlibrary{arrows}
\definecolor{sqsqsq}{rgb}{0.12549019607843137,0.12549019607843137,0.12549019607843137}
\definecolor{aqaqaq}{rgb}{0.6274509803921569,0.6274509803921569,0.6274509803921569}
\definecolor{uuuuuu}{rgb}{0.26666666666666666,0.26666666666666666,0.26666666666666666}

\usepackage{amsthm}

\theoremstyle{plain}

\newtheorem{theorem}{Theorem} % Numbering respecting springer style in the arxiv version
\newtheorem{lemma}{Lemma}

\newtheorem{corollary}{Corollary}
\newtheorem{proposition}{Proposition}

\theoremstyle{definition}

\newtheorem{definition}{Definition} % vs springer-style in arxiv
\newtheorem{example}{Example}
\newtheorem{observation}[theorem]{Observation}
\newtheorem{remark}{Remark}
\newtheorem{question}{Question}

\newcommand{\ones}{\mathbb{1}}
\newcommand{\R}{\mathbb{R}}
\newcommand{\torus}[1]{\R^{#1}/\R\ones} % d-torus: \torus{d+1}
\newcommand{\TP}[1]{\mathbb{TP}^{#1}} % tropical projective space
\newcommand{\dist}{\operatorname{dist}}
\newcommand{\tmap}{\mathcal C}
\newcommand{\ball}[1]{\mathbb{B}^{#1}}
\newcommand{\balld}{\mathbb{B}^d}

\DeclareMathOperator{\conv}{conv}
\DeclareMathOperator{\relint}{relint}
\DeclareMathOperator{\cone}{cone}
\DeclareMathOperator{\Vor}{Vor}
\DeclareMathOperator{\interior}{int}
\newcommand{\Fan}{\mathcal{F}}

\newcommand{\midv}{\mu}
\newcommand{\ccF}{\mathcal{F}_\text{bis}^d}
\newcommand{\bisector}{\operatorname{bis}}
\newcommand{\bop}{\Fan_\text{bop}}
\newcommand\smallSetOf[2]{\{#1\,|\,#2\}}
\newcommand\SetOf[2]{\left\{#1\,\vphantom{#2}\right|\left.\vphantom{#1}\,#2\right\}}
\newcommand\tropseg[2]{[#1,#2]} % tropical line segment

\newcommand\bisectorpoly{\phi}
\newcommand\lineardist{\lambda}
\DeclareMathOperator\proj{pr}

\newcommand\doi[1]{\href{http://dx.doi.org/#1}{\texttt{doi:#1}}}

\newcommand\polymake{\texttt{polymake}\xspace}
\renewcommand\epsilon{\varepsilon}

\title{Tropical bisectors and Voronoi Diagrams}
\author{Francisco Criado \and Michael Joswig \and Francisco Santos}

\address[Francisco Criado, Michael Joswig]{
  Institut f{\"u}r Mathematik,
  TU Berlin,
  Str.\ des 17. Juni 136, 10623 Berlin, Germany
}
\email{\{criado,joswig\}@math.tu-berlin.de}

\address[Francisco Santos]{
  Departamento de Matem\'{a}ticas, Estad\'{i}stica y Computaci\'{o}n,
  Universidad de Cantabria,
  Av.~de Los Castros 48, 39005 Santander, Spain
  }
\email{francisco.santos@unican.es}

\thanks{%
  F.~Criado has been supported by Berlin Mathematical School and Einstein Foundation Berlin (EVF-2015-230).
  M.~Joswig has been supported by Deutsche Forschungsgemeinschaft (EXC 2046: \enquote{MATH$^+$}, SFB-TRR 109: \enquote{Discretization in Geometry and Dynamics}, SFB-TRR 195: \enquote{Symbolic Tools in Mathematics and their Application}, and GRK 2434: \enquote{Facets of Complexity}).
  F.~Santos has been supported by the Einstein Foundation Berlin (EVF-2015-230) and by Grants
  MTM2017-83750-P/AEI/10.13039/ 501100011033 and PID2019-106188GB-I00/AEI/10.13039/501100011033 of the Spanish State Research Agency}
\date{}
\begin{document}

% SPRINGER
%\title{Tropical bisectors and Voronoi diagrams%
% \thanks{%
%   F.~Criado has been supported by Berlin Mathematical School and Einstein Foundation Berlin (EVF-2015-230).
%   M.~Joswig has been supported by Deutsche Forschungsgemeinschaft (DFG, German Research Foundation) under Germany's Excellence Strategy - The Berlin Mathematics Research Center MATH$^+$ (EXC-2046/1, project ID 390685689);  \enquote{Symbolic Tools in Mathematics and their Application} (TRR 195/2, project-ID 286237555); \enquote{Facets of Complexity} (GRK 2434).
%   F.~Santos has been supported by the Einstein Foundation Berlin (EVF-2015-230) and by grants  MTM2017-83750-P/AEI/10.13039/501100011033 and PID2019-106188GB-I00/AEI/10.13039/501100011033 of the Spanish State Research Agency.}
% }
%\author{Francisco Criado \and Michael Joswig \and Francisco Santos}
%
%\institute{Francisco Criado \at
% TU Berlin, Chair of Discrete Mathematics/Geometry.
% \email{criado@math.tu-berlin.de}
% \and
% Michael Joswig \at
% TU Berlin, Chair of Discrete Mathematics/Geometry;
% Max-Planck Institute for Mathematics in the Sciences, Leipzig, Germany.
% \email{joswig@math.tu-berlin.de}
% \and
% Francisco Santos \at
% Departamento de Matem\'{a}ticas, Estad\'{i}stica y Computaci\'{o}n,
% Univ.~de Cantabria,
% % Av.~de Los Castros 48, 39005
% Santander, Spain.
% \email{francisco.santos@unican.es}
%}

%\date{Received: date / Accepted: date}
%% The correct dates will be entered by the editor
%\SPRINGER

\maketitle

% svjour3.cls does not give a clue where the following belongs:
%\noindent
%Communicated by Evelyne Hubert.
%\bigskip

\begin{abstract}
  In this paper we initiate the study of tropical Voronoi diagrams.
  We start out with investigating bisectors of finitely many points with respect to arbitrary polyhedral norms.
  For this more general scenario we show that bisectors of three points are homeomorphic to a non-empty open subset of Euclidean space, provided that certain degenerate cases are excluded.
  Specializing our results to tropical bisectors then yields structural results and algorithms for tropical Voronoi diagrams.
  % SPRINGER
  %\keywords{polyhedral norms \and polytropes \and randomized algorithms \and Voronoi diagrams}
  %\subclass{14T15 % Combinatorial aspects of tropical varieties
  %  \and 46B20    % Geometry and structure of normed linear spaces
  %  \and 52B55}   % Computational aspects related to convexity 
  % \SPRINGER
\end{abstract}

\section{Introduction}

One early route to the success of tropical geometry is based on the tropicalization of classical algebraic varieties defined over some valued field.
Key examples include Mikhalkin's correspondence principle, which relates tropical plane curves with classical complex algebraic curves \cite{Mikhalkin:2005}, or the tropical Grassmannians of Speyer and Sturmfels \cite{SpeyerSturmfels04}.
In all of this the focus lies on the combinatorial properties of tropical varieties, which are ordinary polyhedral complexes.

More recently, however, tropical semi-algebraic sets and their intrinsic geometry came into the picture; see\ \cite{Alessandrini:2013}, \cite{JellScheidererYu:1810.05132}.
For instance, their metric properties appear in \cite{ABGJ:2018} as a tool to show that standard versions of the interior point method of linear programming exhibit an exponential complexity in the unit cost model.
The proof of that result is based on translating metric data on a family of tropical linear programs into curvature information about the central paths of their associated ordinary linear programs.
Similarly, tropical analogs of isoperimetric (or isodiametric) inequalities have been studied in \cite{DepersinGaubertJoswig:2017}, where a tropical volume is defined that corresponds to an \enquote{energy gap} in mathematical physics \cite{KosowskyYuille:1994}.
Another example is the statistical analysis of phylogenetic trees by Lin, Monod and Yoshida \cite{LinMonodYoshida:1805.12400}.

We feel that all this calls for a more systematic investigation of metric properties of tropical varieties.
Starting from first principles, this naturally leads to tropical Voronoi diagrams.
The \emph{tropical distance} between two points $a,b\in\R^{d+1}$ is
\begin{equation}\label{eq:dist}
  \dist(a,b) \ = \ \max_{i\in [d+1]} \left(a_i-b_i\right) - \min_{j\in [d+1]} \left(a_j-b_j\right) \ = \ \max_{i,j\in [d+1]}(a_i-b_i-a_j+b_j) \enspace .
\end{equation}
It does not depend on choosing $\min$ or $\max$ as the tropical addition.
The map $\dist:\R^{d+1}\times\R^{d+1}\to\R$ is non-negative, symmetric, and it satisfies the triangle inequality.
Moreover, it is homogeneous, so it induces a norm on the \emph{tropical $d$-torus} $\torus{d+1}\cong\R^d$, where $\ones=(1,\dots,1)$ denotes the all ones vector.
The \emph{tropical Voronoi region} of a site $s\in S$ with respect to a set $S$ comprises those points in $\torus{d+1}$ to which $s$ is the nearest among all sites in $S$, with respect to $\dist$.
The \emph{tropical Voronoi diagram} $\Vor(S)$ is the cell decomposition of $\torus{d+1}$ into Voronoi regions.
Tropical Voronoi diagrams are a special case of Voronoi diagrams for polyhedral norms, a classical topic in convexity and computational geometry; cf.\ \cite[Sect.~7.2]{AurenhammerKleinLee:2013} or \cite[Sect.~4]{MartiniSwanepoel:2004}.

The intersection of two or more Voronoi regions is part of a \emph{bisector}, i.e., the locus of points which are equidistant to a given set.
For instance, in the Euclidean case the bisector of two points is a degenerate quadric which agrees with an affine hyperplane as a set.
In the tropical setting, the bisector of two points can also be described as part of a tropical hypersurface, but this is now of degree $d+1$; see Proposition~\ref{prop:hypersurface}.
Further, in the tropical case two  points may already produce degenerate bisectors (which may contain, e.g., full-dimensional pieces), whereas the first degenerate case in the Euclidean metric arises for three points.
So tropical Voronoi diagrams behave quite differently from Euclidean Voronoi diagrams.

Yet there are also similarities.
A key structural result is that the tropical Voronoi regions are star convex and can be described as unions of finitely many ordinary polyhedra; see Proposition~\ref{prop:polyhedra-bisectors} and Theorem~\ref{thm:star}.
We prove a second main result, Theorem~\ref{thm:3points}, for the more general case of an arbitrary polyhedral norm in~$\R^d$: the bisector of any three points in weak general position is homeomorphic to an open subset of $\R^{d-2}$.
Our proof generalizes the arguments from \cite{IKLM95}, \cite{trisectors}, where a similar result was proved for smooth norms in $d=2,3$.
However, the global topology of tropical bisectors of three or more points can be radically different from the topology of the classical bisectors.
For instance, tropical bisectors are sometimes disconnected and, more strongly, $d+1$ points can have more than one circumcenter.
This may happen even in general position; see Examples~\ref{exm:circumcenters} and~\ref{exm:not_connected}.
We do not know if bisectors may have nontrivial higher Betti numbers, but we suspect they can; see Theorem~\ref{thm:3points_topo}.

Another contribution in our paper is a randomized incremental algorithm for computing the tropical Voronoi diagram of $n$ points in general position in $\torus {d+1}$ with an expected running time of order $O(n^d\log n)$, for fixed dimension~$d$; see Theorem~\ref{thm:tree}.
Euclidean Voronoi diagrams of finite point sets can be explained fully in terms of ordinary convex polyhedra and convex hull algorithms; see \cite{Four+Marks}, \cite{AurenhammerKleinLee:2013}.
We do not know if there is a tropical analog.

Amini and Manjunath \cite{AminiManjunath:2010} study the Voronoi diagram of a lattice with respect to the following asymmetric version of the tropical distance:
\begin{equation*}
\dist(a,b) \ = \ \max_{i\in [d+1]} \left(a_i-b_i\right).
\end{equation*}
As they show in \cite[Lemma 4.7]{AminiManjunath:2010}, this is the polyhedral distance obtained taking as unit ball the standard simplex.
Their motivation comes from work of Baker and Norine \cite{BakerNorine:2007} on a Riemann--Roch theorem for graphs, which implies a Riemann--Roch theorem for tropical curves.

Our paper is organized as follows.
The short Section~\ref{sec:tropical}, in which we verify that the tropical distance is induced by a polyhedral norm and discuss the combinatorics of the tropical unit ball, sets the stage.
In Section~\ref{sec:polyhedral_norms} we collect our general structural results on bisectors and Voronoi diagrams. The results in this section are proved for general polyhedral norms, but all our examples address the tropical case.
A subtle point is the right concept of \enquote{general position}.
In fact, we distinguish between \emph{weak general position} which prevents bisectors to contain full-dimensional parts (see Proposition~\ref{prop:full-dimensional}), and a stronger \emph{general position} which is defined via stability of facets in the bisector under perturbation of the sites.
For instance, the bisector of any number $k$ of points in general position in $\R^d$ is either empty or a polyhedral complex of \emph{pure} dimension $d+1-k$; see Corollary~\ref{coro:pure}.
As a special case, the bisector of $d+1$ points in $\R^d$ in general position is finite.
Section~\ref{sec:2points} returns to the tropical case.
We specialize our results on bisectors in general polyhedral norms, and we show that the combinatorial types of  tropical bisectors of two points  are classified in terms of a certain polyhedral fan related to the tropical unit ball and the braid arrangement; see Theorems~\ref{thm:ccfan} and~\ref{thm:star}.
This is related to work of Develin~\cite{Develin:2005} on the moduli of tropically collinear points.
Finally, in Section~\ref{sec:computing} we discuss algorithms.
This includes a tropical variant of Fortune's beach line algorithm \cite{Fortune:1987} for planar Voronoi diagrams as well as the aforementioned algorithm in arbitrary dimension.

\paragraph*{Acknowledgment.}
We thank G\"unter Rote and three anonymous referees for useful comments on a first version of this paper.
Especially, one referee brought the reference \cite{AminiManjunath:2010} to our attention.

\section{The tropical unit ball}
\label{sec:tropical}

The unit ball with respect to the tropical distance function defined in \eqref{eq:dist} is
\begin{equation}\label{eq:ball}
  \begin{split}
    \ball{d} \, &= \, \SetOf{x\in\torus{d+1}}{\dist(x,0)\leq 1} \, = \, \bigcap_{i\neq j} \SetOf{x\in\torus{d+1}}{x_i-x_j \leq 1} \\
    &= \, \frac{1}{2} \conv\left(\{\pm 1\}^{d+1}\setminus\{\pm\ones\}\right) + \R\ones \enspace .
  \end{split}
\end{equation}
In this way, $\ball{d}$ is a polytope in the tropical torus $\torus{d+1}$.
We also write
$
  \ball{d}(a,r) % \ = \ \SetOf{x\in\torus{d+1}}{\dist(x,a)\leq r}
$
for the tropical ball with center $a$ and radius $r$.
All tropical balls result from scaling and translating~$\ball{d}$.
In fact, the tropical norm agrees with polyhedral norm with respect to the tropical unit ball, in the sense of Section~\ref{sec:polyhedral_norms}.
Such distances are called \emph{convex distance functions} in \cite[Sect.~7.2]{AurenhammerKleinLee:2013}; see also \cite{MR3135681,trisectors,IKLM95}.

Both the inequality and the vertex descriptions of $\balld$ in Eq.~\eqref{eq:ball} are non-redundant:
\begin{itemize}
\item $\ball{d}$ has $d(d+1)$ facets.
  Each facet corresponds to a choice of coordinates achieving the maximum and the minimum.
\item $\ball{d}$ has $2^{d+1}-2$ vertices.
  Each vertex corresponds to a (nontrivial) partition of the coordinates into maxima and minima.
  For example, $\ball{2}$ is a hexagon and $\ball{3}$ is a rhombic dodecahedron.
\end{itemize}

%\begin{observation}\label{obs:vertices-facets}
%  The inequality decription of $\ball{d}$ given in \eqref{eq:ball} is irredundant; in particular, the polytope $\ball{d}$ has $d(d+1)$ facets.
%  Each facet corresponds to a choice of coordinates achieving the maximum and the minimum.
%  Similarly, the description as a convex hull is irredundant, too, and $\ball{d}$ has $2^{d+1}-2$ vertices.
%  Each vertex corresponds to a (nontrivial) partition of the coordinates into maxima and minima.
%  For example, $\ball{2}$ is a hexagon and $\ball{3}$ is a rhombic dodecahedron.
%\end{observation}

The vertex description also shows that  $\ball{d}$ equals the projection of the $(d{+}1)$-dimensional regular cube $[-1,1]^{d+1}$ in $\R^{d+1}$ along  the direction~$\ones$.
That is, $\ball{d}$ is a zonotope with $d+1$ generators in general position, and all its faces are parallelepipeds.
These generators correspond to the $d+1$ coordinate directions in $\torus{d+1}$. % ... coordinates, which we denote $x_1,\dots, x_{d+1}$.
This suggests a combinatorial way to specify the faces of $\ball{d}$: Each face $F$ can be written as a Minkowski sum
\[
F \ = \ \sum_{i=1}^{d+1} s_i \enspace,
\]
where each $s_i$ is one of $\{-e_i\}$, $[-e_i,e_i]$ or $\{+e_i\}$.
We say that $F$ is of \emph{type} $(F_-,F_*,F_+)$ if
\begin{equation}\label{eq:type}
  \begin{aligned}
    F_-\ =&\ \left\{i \in [d+1]: s_i = \{-e_i\}\right\}\,, \\
    F_*\ =&\ \left\{i \in [d+1]: s_i = [-e_i,e_i]\right\}\,, \\
    F_+\ =&\ \left\{i \in [d+1]: s_i = \{e_i\}\right\} \enspace.
  \end{aligned}
\end{equation}
Conversely, a partition of $[d+1]$ into three parts $F_-$, $F_*$, $F_+$ corresponds to a face of $\ball{d}$ if and only if neither $F_-$ nor $F_+$ is empty.
Moreover, the dimension of $F$ equals the cardinality of $F_*$.
In particular, the vertices of $\ball{d}$ correspond to the $2^{d+1} -2$ ways of partitioning $[d+1]$ into two non-empty subsets.
The facets of $\ball{d}$ correspond to the $d(d+1)$ ways of choosing an ordered pair from $[d+1]$, without repetition.

%In the rest of the paper we write $\Fan(\ball{d})$ for the \emph{face fan} of $\ball{d}$; its cones are the conical hulls of the faces; see Figure~\ref{fig:facefan}.
%The notation from \eqref{eq:type} also applies to the cones of the face fan.

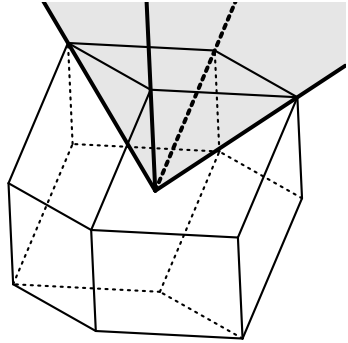
\begin{figure}[th]
  %\sidecaption
  % Picture of the two fans showing that the bisector is piecewise lienear
  \begin{tikzpicture}[line cap=round,line join=round,>=triangle 45,x=1.0cm,y=1.0cm]
\clip(-2.5,-2.5) rectangle (2.5,2.5);
\fill[line width=0.8pt,dotted,fill=black,fill opacity=0.10000000149011612] (0.,0.) -- (5.407220724754539,3.5870906936134554) -- (-2.6286239350960168,4.487232206390993) -- cycle;
\draw [line width=0.8pt] (1.1480299914398842,-1.9597619433925049)-- (1.930048539262598,-0.09985284614025479);
\draw [line width=0.8pt] (1.1480299914398842,-1.9597619433925049)-- (-0.7820185478227137,-1.85990909725225);
\draw [line width=0.8pt] (-0.8429685379040314,-0.5201168529438285)-- (-0.7820185478227137,-1.85990909725225);
\draw [line width=0.8pt] (-0.7820185478227137,-1.85990909725225)-- (-1.8690985491812802,-1.2399393981681666);
\draw [line width=0.8pt,dotted] (-1.8690985491812802,-1.2399393981681666)-- (-1.0870800013585664,0.6199696990840834);
\draw [line width=0.8pt,dotted] (-1.0870800013585664,0.6199696990840834)-- (-1.1480299914398842,1.9597619433925049);
\draw [line width=0.8pt] (-1.1480299914398842,1.9597619433925049)-- (-0.06094999008131763,1.3397922443084216);
\draw [line width=0.8pt] (-0.06094999008131763,1.3397922443084216)-- (1.8690985491812802,1.2399393981681668);
\draw [line width=0.8pt] (1.930048539262598,-0.09985284614025479)-- (1.8690985491812802,1.2399393981681668);
\draw [line width=0.8pt,dotted] (1.930048539262598,-0.09985284614025479)-- (0.8429685379040315,0.5201168529438286);
\draw [line width=0.8pt] (-1.8690985491812802,-1.2399393981681666)-- (-1.930048539262598,0.09985284614025502);
\draw [line width=0.8pt] (-1.930048539262598,0.09985284614025502)-- (-1.1480299914398842,1.9597619433925049);
\draw [line width=0.8pt] (-1.1480299914398842,1.9597619433925049)-- (0.7820185478227137,1.85990909725225);
\draw [line width=0.8pt] (0.7820185478227137,1.85990909725225)-- (1.8690985491812802,1.2399393981681668);
\draw [line width=0.8pt] (1.8690985491812802,1.2399393981681668)-- (1.0870800013585664,-0.6199696990840833);
\draw [line width=0.8pt] (1.0870800013585664,-0.6199696990840833)-- (1.1480299914398842,-1.9597619433925049);
\draw [line width=0.8pt] (1.0870800013585664,-0.6199696990840833)-- (-0.8429685379040314,-0.5201168529438285);
\draw [line width=0.8pt] (-0.8429685379040314,-0.5201168529438285)-- (-1.930048539262598,0.09985284614025502);
\draw [line width=0.8pt] (-0.06094999008131763,1.3397922443084216)-- (-0.8429685379040314,-0.5201168529438285);
\draw [line width=0.8pt,dotted] (0.8429685379040315,0.5201168529438286)-- (0.7820185478227137,1.85990909725225);
\draw [line width=0.8pt,dotted] (0.8429685379040315,0.5201168529438286)-- (-1.0870800013585664,0.6199696990840834);
\draw [line width=0.8pt,dotted] (0.06094999008131774,-1.3397922443084216)-- (0.8429685379040315,0.5201168529438286);
\draw [line width=0.8pt,dotted] (1.1480299914398842,-1.9597619433925049)-- (0.06094999008131774,-1.3397922443084216);
\draw [line width=0.8pt,dotted] (0.06094999008131774,-1.3397922443084216)-- (-1.8690985491812802,-1.2399393981681666);
\draw [line width=1.5pt,domain=-2.5:0.0] plot(\x,{(-0.--1.9597619433925049*\x)/-1.1480299914398842});
\draw [line width=1.5pt,domain=-2.5:0.0] plot(\x,{(-0.--1.3397922443084216*\x)/-0.06094999008131763});
\draw [line width=1.5pt,domain=0.0:2.5, dotted] plot(\x,{(-0.--1.85990909725225*\x)/0.7820185478227137});
\draw [line width=1.5pt,domain=0.0:2.5] plot(\x,{(-0.--1.2399393981681668*\x)/1.8690985491812802});
\draw [line width=0.8pt,dotted] (0.,0.)-- (5.407220724754539,3.5870906936134554);
\draw [line width=0.8pt,dotted] (5.407220724754539,3.5870906936134554)-- (-2.6286239350960168,4.487232206390993);
\draw [line width=0.8pt,dotted] (-2.6286239350960168,4.487232206390993)-- (0.,0.);
\begin{scriptsize}
\draw [fill=black] (-2.6286239350960168,4.487232206390993) circle (2.5pt);
\draw[color=black] (-1.5041762306351854,2.6866909836855406) node {$H$};
\draw [fill=black] (5.407220724754539,3.5870906936134554) circle (2.5pt);
\draw[color=black] (2.665314585760604,1.3677704193154492) node {$I$};
\draw[color=black] (1.4350311560927989,4.221886156729223) node {$a_2$};
\end{scriptsize}
\end{tikzpicture}
\caption{The tropical 3-ball $\ball{3}$, with the conical hull of a facet highlighted}
\label{fig:facefan}
\end{figure}

\begin{remark}
The zonotope $\ball{d}$ is dual to an arrangement of $d+1$ linear hyperplanes in general position in $\R^d$, oriented so that the intersection of all positive half-spaces is empty. In particular, its face lattice is the same as the lattice of non-zero covectors of the unique totally cyclic oriented matroid of rank $d$ with $d+1$ elements. Covectors of an oriented matroid are usually written as $(V_-,V_0,V_+)$ but in our context we prefer to use $*$ instead of zero meaning that the corresponding coordinate is not fixed.
\end{remark}

\begin{remark}
  Another general description of $\ball{d}$ is that it equals the (ordinary) Voronoi cell of the lattice of type $A_d$ (i.e., the triangular lattice for $d=2$ and the face centered cubic lattice (FCC) for $d=3$).
  Similarly, $\ball{d}$ is the polytope polar to the difference body $T-T$ of a regular $d$-simplex $T$.
  This description shows that $\ball{d}$  is the same as the polytope $U_d$ that appears in Makeev's conjecture.
  See, e.g., \cite[Conjecture 21.3.2]{HDCG:3:topological+methods}.
\end{remark}

\section{Bisectors in polyhedral norms}
\label{sec:polyhedral_norms}

Throughout this section we work in the general framework of Minkowski norms; see \cite[Sect.~7.2]{AurenhammerKleinLee:2013}, \cite{MR3135681}, \cite{MartiniSwanepoel:2004}.
Consider a polytope $K \subset \R^d$ with the origin in its interior.
Let $\dist(a,b)$ be the unique scaling factor $\alpha>0$ such that $b-a \in \partial (\alpha K)$. % MJ: I find this better to read than "\alpha \partial K"
Then, $\dist$ satisfies the triangle inequality, is invariant under translation, and homogeneous under scaling.
If $K=-K$ then $\dist(a,b)= \dist(b,a)$ and $\dist(0,\cdot)$ is a norm in $\R^d$ in the usual sense.
We allow $K\neq -K$, whence $\dist(a,b)\neq \dist(b,a)$ in general, but we still call it a norm. 
Bisectors and Voronoi diagrams for these norms have been studied in computational geometry~\cite[Sect.~7.2]{AurenhammerKleinLee:2013}, \cite[Sect.~4]{MartiniSwanepoel:2004}.

For any finite point set $S$ we define its \emph{bisector}:
\[
  \bisector(S) \ := \ \SetOf{x\in \R^d}{\dist(a,x) = \dist(b,x) \text{ for } a,b\in S} \enspace.
\]
Following the computational geometry tradition we will often call the elements of $S$ the \emph{sites}.
Although some of our results also hold for general convex bodies, for simplicity we assume $K$ to be a polytope.
We denote by $\Fan(K)$ the face fan of~$K$.
The norm $\dist(0,\cdot)$ is linear in each of these cones, so we write
\[
  \bisector_{(F_1, \dots, F_k)}(\{a_1,\dots,a_k\}) \ = \ \bisector(\{a_1,\dots,a_k\}) \cap a_1+F_1 \cap \dots \cap a_k+F_k \enspace,
\]
for the intersection of the bisector with a choice of cones $F_i\in\Fan(K)$.
Each cell of the form $\bisector_{(F_1,\dots,F_k)}(a_1,\dots,a_k)$ is the intersection of the polyhedron $(a_1+F_1) \cap\dots \cap (a_k+F_k)$ with an affine subspace, which implies it is itself a polyhedron. As a consequence:

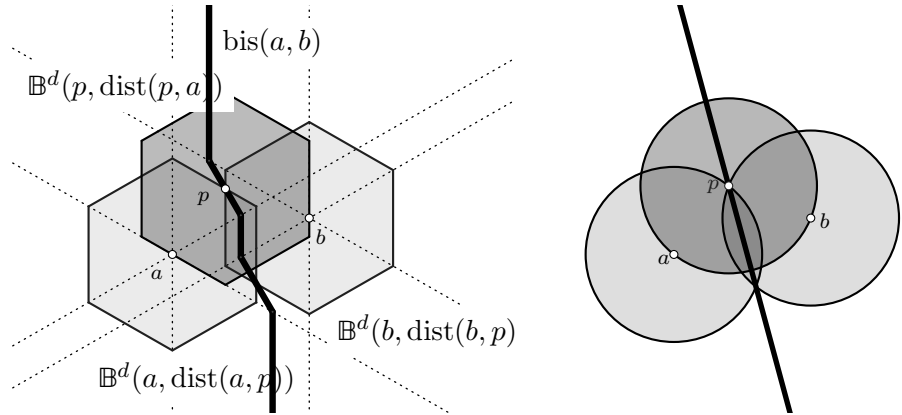
\begin{figure}[th]
  \begin{tikzpicture}[scale=0.6, line cap=round,line join=round,>=triangle 45,x=1.0cm,y=1.0cm]
\clip(-3.5,-3.5) rectangle (7.5,5.5);
\fill[line width=2.pt,color=sqsqsq,fill=sqsqsq,fill opacity=0.09000000357627869] (1.8374788999639984,1.060868937524472) -- (0.,2.1217378750489435) -- (-1.8374788999639986,1.0608689375244715) -- (-1.8374788999639984,-1.0608689375244726) -- (0.,-2.121737875048945) -- (1.837478899963998,-1.0608689375244738) -- cycle;
\fill[line width=2.pt,color=sqsqsq,fill=sqsqsq,fill opacity=0.09000000357627869] (1.1625211000360016,1.8647429436676326) -- (1.1625211000360016,-0.25699493138131113) -- (3.,-1.3178638689057833) -- (4.837478899963999,-0.2569949313813113) -- (4.837478899964,1.864742943667632) -- (3.,2.9256118811921046) -- cycle;
\fill[line width=2.pt,fill=black,fill opacity=0.27000001072883606] (3.,0.3896870674800665) -- (3.,2.5114249425290103) -- (1.1625211000360012,3.5722938800534823) -- (-0.674957799927998,2.5114249425290103) -- (-0.6749577999279983,0.38968706748006665) -- (1.1625211000359998,-0.6711818700444061) -- cycle;
\draw [line width=0.5pt,dotted] (3.,-4.) -- (3.,6.);
\draw [line width=0.5pt,dotted,domain=-3.5:7.5] plot(\x,{(-0.--1.5*\x)/-2.598076211353316});
\draw [line width=0.5pt,dotted] (0.,-4.) -- (0.,6.);
\draw [line width=0.5pt,dotted,domain=-3.5:7.5] plot(\x,{(-2.4114740677141646--1.5*\x)/2.598076211353316});
\draw [line width=0.5pt,dotted,domain=-3.5:7.5] plot(\x,{(-6.588525932285835--1.5*\x)/-2.598076211353316});
\draw [line width=0.5pt,dotted,domain=-3.5:7.5] plot(\x,{(-0.-0.5*\x)/-0.8660254037844385});
\draw [line width=0.8pt,color=sqsqsq] (1.8374788999639984,1.060868937524472)-- (0.,2.1217378750489435);
\draw [line width=0.8pt,color=sqsqsq] (0.,2.1217378750489435)-- (-1.8374788999639986,1.0608689375244715);
\draw [line width=0.8pt,color=sqsqsq] (-1.8374788999639986,1.0608689375244715)-- (-1.8374788999639984,-1.0608689375244726);
\draw [line width=0.8pt,color=sqsqsq] (-1.8374788999639984,-1.0608689375244726)-- (0.,-2.121737875048945);
\draw [line width=0.8pt,color=sqsqsq] (0.,-2.121737875048945)-- (1.837478899963998,-1.0608689375244738);
\draw [line width=0.8pt,color=sqsqsq] (1.837478899963998,-1.0608689375244738)-- (1.8374788999639984,1.060868937524472);
\draw [line width=0.8pt,color=sqsqsq] (1.1625211000360016,1.8647429436676326)-- (1.1625211000360016,-0.25699493138131113);
\draw [line width=0.8pt,color=sqsqsq] (1.1625211000360016,-0.25699493138131113)-- (3.,-1.3178638689057833);
\draw [line width=0.8pt,color=sqsqsq] (3.,-1.3178638689057833)-- (4.837478899963999,-0.2569949313813113);
\draw [line width=0.8pt,color=sqsqsq] (4.837478899963999,-0.2569949313813113)-- (4.837478899964,1.864742943667632);
\draw [line width=0.8pt,color=sqsqsq] (4.837478899964,1.864742943667632)-- (3.,2.9256118811921046);
\draw [line width=0.8pt,color=sqsqsq] (3.,2.9256118811921046)-- (1.1625211000360016,1.8647429436676326);
\draw [line width=0.8pt] (3.,0.3896870674800665)-- (3.,2.5114249425290103);
\draw [line width=0.8pt] (3.,2.5114249425290103)-- (1.1625211000360012,3.5722938800534823);
\draw [line width=0.8pt] (1.1625211000360012,3.5722938800534823)-- (-0.674957799927998,2.5114249425290103);
\draw [line width=0.8pt] (-0.674957799927998,2.5114249425290103)-- (-0.6749577999279983,0.38968706748006665);
\draw [line width=0.8pt] (-0.6749577999279983,0.38968706748006665)-- (1.1625211000359998,-0.6711818700444061);
\draw [line width=0.8pt] (1.1625211000359998,-0.6711818700444061)-- (3.,0.3896870674800665);
\draw (0.8761648351067844,5.227635760021221) node[anchor=north west,fill=white] {$\bisector(a,b)$};
\draw (-1.8550187330721633,-2.15) node[anchor=north west,fill=white] {$\ball{d}(a,\dist(a,p))$};
\draw (3.3866467008066254,-1.1) node[anchor=north west,fill=white] {$\ball{d}(b,\dist(b,p))$};
\draw (-3.4,4.3) node[anchor=north west,fill=white] {$\ball{d}(p,\dist(p,a))$};
\draw [line width=2.4pt] (1.5,-0.06215139764127786)-- (1.5,0.8660254037844386);
\draw [line width=2.4pt] (1.5,0.8660254037844386)-- (0.8038246892380552,2.0718364129991795);
\draw [line width=2.4pt] (1.5,-0.06215139764127786)-- (2.1961753107619453,-1.2679624068560191);
\draw [line width=2.4pt] (0.8038246892380552,2.0718364129991795) -- (0.8038246892380552,6.);
\draw [line width=2.4pt] (2.1961753107619453,-1.2679624068560191) -- (2.1961753107619453,-4.);
\begin{scriptsize}
\draw [fill=white,draw=black] (0.,0.) circle (2.5pt);
\draw[color=black] (-0.33769452852830345,-0.4416392224108269) node {$a$};
\draw [fill=white,draw=black] (3.,0.8038740061431607) circle (2.5pt);
\draw[color=black] (3.2487081367571835,0.3859921618858225) node {$b$};
\draw [fill=uuuuuu] (1.5,0.8660254037844386) circle (1.0pt);
\draw [fill=uuuuuu] (1.5,-0.06215139764127786) circle (1.0pt);
\draw [fill=uuuuuu] (0.8038246892380552,2.0718364129991795) circle (1.0pt);
\draw [fill=uuuuuu] (2.1961753107619453,-1.2679624068560191) circle (1.0pt);
\draw [fill=white,draw=black] (1.1625211000360016,1.4505560050045383) circle (2.5pt);
\draw[color=black] (0.683050845437566,1.2412112589923603) node {$p$};
\end{scriptsize}
\end{tikzpicture}
\qquad
\begin{tikzpicture}[scale=0.6, line cap=round,line join=round,>=triangle 45,x=1.0cm,y=1.0cm]
\clip(-2,-3.5) rectangle (5,5.5);
\draw [line width=0.8pt,fill=black,fill opacity=0.12999999523162842] (3.,0.8038740061431607) circle (1.9344809080638885cm);
\draw [line width=0.8pt,fill=black,fill opacity=0.12999999523162842] (0.,0.) circle (1.9344809080638883cm);
\draw [line width=0.8pt,fill=black,fill opacity=0.2800000011920929] (1.2014317669821444,1.5161721844664768) circle (1.9344809080638885cm);
\draw (-3,4.461877167674134) node[anchor=north west] {}; % $\bisector(a,b)$};
\draw (-2.7295748965261732,-1.8148880008685546) node[anchor=north west] {}; %$\mathbb{B}^d(a,\dist(a,p))$};
\draw (3.650369973417217,-0.8346808375618883) node[anchor=north west] {}; %$\mathbb{B}^d(b,\dist(b,p))$};
\draw (0.8989112693985043,4.5306636352746015) node[anchor=north west] {}; %$\mathbb{B}^d(p,\dist(p,a))$};
\draw [line width=2.pt,domain=-3.5:7.5] plot(\x,{(-4.8231067088763275--3.*\x)/-0.8038740061431607});
\begin{scriptsize}
\draw [fill=white,draw=black] (0.,0.) circle (2.5pt);
\draw[color=black] (-0.21886882910909777,-0.10382461930691764) node {$a$};
\draw [fill=white,draw=black] (3.,0.8038740061431607) circle (2.5pt);
\draw[color=black] (3.306437635414878,0.7388096087988131) node {$b$};
\draw [fill=white,draw=sqsqsq] (1.2014317669821444,1.5161721844664768) circle (2.5pt);
\draw[color=sqsqsq] (0.8645180355982703,1.4954607524039591) node {$p$};
\end{scriptsize}
\end{tikzpicture}

\caption{Left: A point, $p$, in the tropical bisector of $a$ and $b$. Right: The analogous picture, in classical geometry.}
\label{tikz:bisector+tropical+classical}
\end{figure}

\begin{proposition} \label{prop:polyhedra-bisectors}
  Let $K$ be a polytope with the origin in its interior, and let $\dist$ be the corresponding Minkowski norm.
  Let $S=\{a_1,\dots,a_k\}\subset\R^d$ be a finite point set.
  Then the bisector $\bisector(\{a_1,\dots,a_k\})$ is a polyhedral complex whose cells are the polyhedra
  \[
    \bisector_{(F_1,\dots,F_k)}(a_1,\dots,a_k)
  \]
for all choices of   $F_1,\dots,F_k\in \Fan(K)$.%, together with all their faces.
\end{proposition}

\begin{proof}
The family of polyhedra
\[
  \bisector_{(F_1,\dots,F_k)}(a_1,\dots,a_k) \,, \text{ with } F_1,\dots,F_k\in \Fan(K) \,,
\]
forms a polyhedral complex since
\[
  \bisector_{(F_1,\dots,F_k)}(S) \cap \bisector_{(F'_1,\dots,F'_k)}(S) \ = \ \bisector_{(F_1\cap F'_1,\dots,F_k\cap F'_k)}(S) \enspace.
\]
That polyhedral complex covers the entire bisector since for each point $p\in\bisector(S)$ and for each $i$, the point $a_i$ must lie in some face $F_i$ of $p-\dist(a_i,p)K$.
%\qed
\end{proof}

Our primary example is the case where $K=\balld$ is the tropical ball.
In Figure \ref{tikz:bisector+tropical+classical} the point $p$, which is generic within the bisector of $a$ and $b$, lies in the facet $\bisector_{(-*+),(+-*)}(a,b)$.
For the purpose of drawing pictures, notice that any three vectors $v_1,v_2,v_3\in\R^2$ with $v_1+v_2+v_3=0$ define a map from $\torus{3}$ to $\R^2$ via $e_i\mapsto v_i$.
While $v_1=(1,0)$, $v_2=(0,1)$, $v_3=(-1,-1)$ is a common base for diagrams in tropical geometry, for our pictures we settle for the more symmetric \emph{isometric view} where our base is:
\[
  v_1=\left(-\sin\frac{2\pi}{3},\,\cos\frac{2\pi}{3}\right)\,,\ v_2=\left(\sin\frac{2\pi}{3},\,\cos\frac{2\pi}{3}\right)\,,\ v_3=\bigl(0,1\bigr) \enspace.
\]

\begin{remark}
  \label{rem:emptiness_is_open}
  If $\bisector(a_1,\dots,a_k)=\emptyset$ then $\bisector(a'_1,\dots,a'_k)=\emptyset$ for every choice of $a'_1,\dots, a'_k$ sufficiently close to $a_1,\dots, a_k$, by continuity of $\dist(\cdot,\cdot)$.
\end{remark}

% \begin{proof}
% For each choice of $(F_1,\dots,F_k)$, non-emptiness of $\bisector_{(F_1,\dots,F_k)}(a_1,\dots,a_k)$ can be written as feasibility of a certain system of (closed) linear inequalities on the coordinates of $a_1,\dots,a_k$. This gives a description of the $k$-tuples with empty bisector as a finite intersection of open sets in $(\R^{d})^k$.
% \qed\end{proof}

\subsection{Weak general position and general position}
\label{sec:gen_pos}

\begin{definition}[General position]
A finite point set $S \subset \R^{d}$ is in \emph{general position} with respect to $K$, if for every subset $\{a_1,\dots,a_k\} \subset S$ there are neighborhoods $U_i$ of each $a_i$ such that for every choice of $\{a'_1,\dots,a'_k\}$ with $a'_i \in U_i$ and for every choice of maximal cones $F_1,\dots,F_k\in \Fan(K)$ we have
\[
\bisector_{(F_1,\dots,F_k)}(a_1,\dots,a_k) = \emptyset
\ \iff \
\bisector_{(F_1,\dots,F_k)}(a'_1,\dots,a'_k) = \emptyset \enspace.
\]

Moreover, the set $S$ is in \emph{weak general position} if no pair of points $a,b\in S$ lies in a hyperplane parallel to a facet of~$K$.
\end{definition}

\begin{remark}
As the name suggests, \enquote{weak general position} is implied by \enquote{general position} (see Corollary~\ref{cor:general_position}).
A yet  stronger notion would arise requiring stability not only of facets but also of lower-dimensional cells in bisectors; that is, allowing lower-dimensional cones from $\Fan(K)$ for the $F_i$ in the definition of general position.
But this intermediate notion of \enquote{general position} is the most appropriate for our purposes and the algorithms in Section \ref{sec:computing}.
As a first indication, Theorem~\ref{thm:genpos} provides a local characterization.
\end{remark}

By Proposition~\ref{prop:polyhedra-bisectors} bisectors are polyhedral complexes, and thus each cell has a dimension.

\begin{proposition}
\label{prop:full-dimensional}
Two points $a,b$ are in weak general position if and only if $\bisector(a,b)$ does not contain full-dimensional cells.
\end{proposition}

\begin{proof}
Since $\bisector_{(F,F')}(a,b) \subset (a+F) \cap (b+F')$, for it to be $d$-dimensional we need $F$ and $F'$ to be cones of facets. We also need $F=F'$, so that $\dist(a,\cdot)$ and $\dist(b,\cdot)$ have the same gradient on $(a+F) \cap (b+F)$, and we need $b-a$ to be parallel to the facet, so that $\dist(a,\cdot) = \dist(b,\cdot)$  on $(a+F) \cap (b+F)$. 

Conversely, if $b-a$ is parallel to a
facet of $K$ with cone $F$ then
\[
\bisector_{(F,F)}(a,b) \ = \  (a+F) \cap (b+F) \enspace,
\]
and this is $d$-dimensional.
%\qed
\end{proof}

\begin{corollary}\label{cor:general_position}
General position implies weak general position.
\end{corollary}

\begin{proof}
  Suppose $a-b$ is parallel to a facet of~$K$, so that $\bisector_{(F,F)}(a,b)$ is full-dimensional, where $F$ is the cone of that facet.
  Taking $b'$ close to $b$ but away from the hyperplane parallel to the facet makes $\bisector_{(F,F)}(a,b')$ empty.
  Now the claim follows from Proposition~\ref{prop:full-dimensional}.
%\qed
\end{proof}

\begin{theorem}
\label{thm:genpos}
Let $S= \{a_1,\dots,a_k\} \subseteq\R^d$ and for each $a_i$ choose a maximal cone $F_k\in \Fan(K)$. Let $Q:=(a_1+F_1) \cap \dots \cap (a_{k}+F_{k})$, let $\lineardist_{F_i}(x)$ be the linear function that restricts to $\dist(0,x)$ on $F_i$, and let $H$ be the affine subspace defined by $\lineardist_{F_1}(x-a_1)=\dots=\lineardist_{F_k}(x-a_k)$.

Then, the following conditions are equivalent:
\begin{enumerate}
\item\label{it:1}
  There are neighborhoods $U_i$ of each $a_i$ such that for any choice of $a'_i \in U_i$ the polyhedron $\bisector_{(F_1,\dots,F_k)}(a'_1,\dots,a'_k)$ is not empty.
\item\label{it:2}
\begin{enumerate}
\item $Q$ is full-dimensional and $H$ intersects its interior; and
\item the $k-1$ functions $\lineardist_{F_i}-\lineardist_{F_1}$ for $i=2,\dots,k$ are linearly independent.
\end{enumerate}
\end{enumerate}
\end{theorem}
Since
\begin{equation}\label{eq:QH}
  \bisector_{(F_1,\dots,F_{k})}(a_1,\dots, a_{k}) \ = \ Q\cap H
\end{equation}
condition (a) is equivalent to \enquote{$\bisector_{(F_1,\dots,F_{k})}(a_1,\dots, a_{k})$ meets the interior of $Q$}.
\begin{proof}
  For the implication from  \enquote{\ref{it:1}} to \enquote{\ref{it:2}}, let us first show that \enquote{\ref{it:1}} forces $Q$ to be full-dimensional.
  Aiming for a contradiction, we assume that $Q$ is contained in the boundary of one of the cones $a_i +F_i$.
  That is, the polyhedron $Q_i := \bigcap_{j\neq i} (a_j + F_j)$ does not meet the interior of $a_i +F_i$, for some~$i$.
  Then any $a'_i$ in the interior of $a_i +F_i$ yields $(a'_i+F_i) \cap Q_i =\emptyset$.
  Hence $\bisector_{(F_1,\dots,F_k)}(a'_1,\dots,a'_k)=\emptyset$, where $a'_j=a_j$ for $j\neq i$.
  This contradicts \enquote{\ref{it:1}} and shows that $Q$ is full-dimensional.
  
  To see that $\bisector_{(F_1,\dots,F_k)}(a_1,\dots,a_k)$ must intersect the interior of $Q$, suppose we are given neighborhoods $U_i$ as in \enquote{\ref{it:1}}; recall \eqref{eq:QH}.
  For each $i$, choose $v_i$ in the interior of $F_i$ and such that $\lambda_i(v_i)=1$.
  Let $a'_i := a_i +\epsilon v_i$, where $\epsilon>0$ is taken small enough so that $a'_i \in U_i$.
  Observe that our choice of $v_i$ makes the affine subspace defined by $\lineardist_{F_1}(x-a'_1)=\dots=\lineardist_{F_k}(x-a'_k)$ agree with $H$. 
  Let $Q':= (a'_1+F_1) \cap \dots \cap (a'_{k}+F_{k})$, which lies in the interior of $Q$.
  We have
  \[
    \bisector_{(F_1,\dots,F_{d+1})}(a'_1,\dots, a'_{d+1})  \ = \ Q' \cap H \enspace.
  \]
  By \enquote{\ref{it:1}} this is not empty, and thus $H$ intersects the interior of $Q$.

  For condition (b) observe that $H$ can equivalently be defined by the $k-1$ affine equalities
  \begin{equation}\label{eq:H}
    (\lineardist_{F_i}- \lineardist_{F_1})(x) = \lineardist_{F_i}(a_i)- \lineardist_{F_1}(a_1)  \qquad \text{for } i=2,\dots,k \enspace.
  \end{equation}
  If the left-hand sides are linearly dependent, then one of the $k-1$ equations, say the $i$th one, is redundant.
  But then choosing a point $a'_i \in U_i$ with $\lineardist_{F_i}(a'_i) \neq \lineardist_{F_i}(a_i)$ (and letting $a'_j=a_j$ for the rest) renders the system of equations infeasible.
  Hence $\bisector_{(F_1,\dots,F_{d+1})}(a'_1,\dots, a'_{d+1}) = \emptyset$, contradicting \enquote{\ref{it:1}}.

  We now show that  \enquote{\ref{it:2}} implies  \enquote{\ref{it:1}}.
  Consider arbitrary points $a'_i$, and let $Q'_i =\bigcap_i (a'_i+F_i)$.
  Further let $H'$ be the affine subspace defined by 
  \[
    \lineardist_{F_1}(x-a'_1) \ = \ \cdots \ = \ \lineardist_{F_k}(x-a'_k) \enspace,
  \]
  so that 
  \[
    \bisector_{(F_1,\dots,F_{k})}(a'_1,\dots, a'_{k}) \ = \ Q' \cap H' \enspace.
  \]
  We want to show that if each $a'_i$ is sufficiently close to the corresponding $a_i$ for all $i$ then $Q'\cap H'$ is not empty.

  Condition (b) says that $H'$ is $(d{+}1{-}k)$-dimensional (and parallel to $H$) for every choice of $a'_i$s and that it varies continuously with the choice.
  Condition (a) says that $Q'$ stays full-dimensional if the $a'_i$s are sufficiently close to the original $a_i$s, and that it also varies continuously with the choice, in the following strong sense: consider a description of each $a_i+F_i$ by a finite system of linear inequalities.
  Then $a'_i+F_i$ is defined by a system with the same linear functions and with right-hand sides varying continuously with the $a'_i$s. 

  Thus, if each $a'_i$ is close to $a_i$, then $Q'$ and $H'$ are a full-dimensional polyhedron and a $(d{+}1{-}k)$-dimensional affine subspace close to $Q$ and $H$ respectively.
  Since, by (a), $H$ intersects the interior of $Q$, continuity implies that $H'$ still intersects the interior of $Q'$ when $a'_i$ is close enough to $a_i$ for each $i$.
  In particular, $\bisector_{(F_1,\dots,F_{k})}(a'_1,\dots, a'_{k})$ is not empty.
%\qed
\end{proof}

\begin{corollary}
\label{coro:pure}
The bisector of $k$ points in general position is either empty or pure of dimension $d+1-k$.
In particular, the bisector of $d+1$ points in general position is finite, and this is empty for more than $d+1$ points.
\end{corollary}

\begin{proof}
Every maximal non-empty cell $\bisector_{(F_1,\dots,F_{k})}(a_1,\dots, a_{k})$
is the intersection of the polyhedron $Q$ with the affine subspace $H$ of Theorem~\ref{thm:genpos}.
That result implies that $H$ has dimension $d+1-k$ and meets the interior of the full-dimensional polyhedron $Q$.
Thus, the cell has dimension $d+1-k$.
%\qed
\end{proof}

\begin{corollary}
\label{coro:d+2}
If every subset of at most $d+2$ points in $S$ is in general position then so is $S$.
\end{corollary}

\begin{corollary}\label{coro:genericity}
  For any $n\geq 1$, the sequences of $n$ points in $\R^{d}$ in general position form an open dense subset of  $(\R^{d})^n$.
\end{corollary}

\begin{proof}
  Let $S\subset\R^d$ be a set of cardinality~$n$.
  Then $S$ is in general position if and only if, for each subset $\{a_1,\dots,a_k\}\subset S$ and maximal cones  $F_1,\dots,F_k\in \Fan(K)$, the polyhedron $\bisector_{(F_1,\dots,F_k)}(a_1,\dots,a_k)$  either is  empty or satisfies condition (1) of Theorem~\ref{thm:genpos}.
  Since there are finitely many choices of $\{a_1,\dots,a_k\}$ and $\{F_1,\dots,F_k\}$, it suffices to prove the statement for one such choice.

  For openness: Emptyness is an open condition because of the continuity of $\dist(\cdot,\cdot)$.
  On the other hand, Theorem~\ref{thm:genpos} says that condition (1) is equivalent to (2), which is open by the arguments in the proof (namely, the fact that both $Q$ and $H$ depend continuously on the sites).   	

  For density, we are going to show that if $Q\cap H$ is not empty but fails to satisfy condition (a) or (b)
then the set $\{a_1,\dots,a_k\}$ lies in one of finitely many linear hyperplanes in $(\R^{d})^k$. 
To emphasize that $Q$ and $H$ depend on the choice of sites we denote them $Q(a_1,\dots,a_k)$ and $H(a_1,\dots,a_k)$.
 
 If (b) fails for $\{a_1,\dots,a_k\}$ then the linear system \eqref{eq:H} defining $H(a_1,\dots,a_k)$ is feasible but overdetermined, which implies a linear relation, depending solely on $F_1,\dots,F_k$, among the 
 right-hand sides $\lineardist_{F_i}(a_i)- \lineardist_{F_1}(a_1)$. The relation is not tautological on the  $a_i$s since, as shown in the proof of Theorem~\ref{thm:genpos}, it is easy to construct a point set $\{a'_1,\dots,a'_k\}$ with $H(a'_1,\dots,a'_k)=\emptyset$.
 
 For (a), consider the inequality descriptions of the cones $F_i$ and translate them to obtain an inequality description of $Q(a_1,\dots,a_k)$ as the feasibility region of a system $\mathcal S(a_1,\dots,a_k)$ of affine inequalities with fixed gradients and with right-hand sides parameterized linearly by the $a_i$s. 
 If (a) fails for $\{a_1,\dots,a_k\}$ then one of two things happen: 
 \begin{itemize}
\item  \emph{$Q(a_1,\dots,a_k)$ is non-empty but not full-dimensional}. Consider a minimal subsystem of $\mathcal S(a_1,\dots,a_k)$ that already defines a non-full-dimensional feasibility region. Minimality implies that this feasibility region is an affine subspace of $\R^d$ and that turning the inequalities to equalities produces an over-determined subsystem. This implies, as in the previous case, a linear relation among the $a_i$s. 

\item \emph{$Q(a_1,\dots,a_k)$ is full-dimensional but $Q(a_1,\dots,a_k)\cap H(a_1,\dots,a_k)$ is contained in its boundary}. This implies $Q(a_1,\dots,a_k)\cap H(a_1,\dots,a_k)$ to be contained in a facet of $Q(a_1,\dots,a_k)$. Let $H_0(a_i)$ be the hyperplane containing that facet. Note that the hyperplane $H_0$ depends only on one $a_i$ since it comes from the description of one of the cones $a_i + F_i$. Then, adding to the the linear system \eqref{eq:H} the equation defining $H_0(a_i)$ produces again an over-determined system, hence there is a linear relation among the $a_i$s. 

%  need to prove that a random small perturbation makes conditions $(a)$ and $(b)$ hold almost surely for any non empty cell of the perturbed bisector. Condition (a) holds almost surely because a random perturbation will almost surely put the points in weak general position, so the polyhedron $Q$ has non empty interior. The cell is the intersection of a linear space with the polyhedron $Q$. This intersection ocurs in the boundary if a non-trivial linear set of equations for the $a_i$ hold, which happens with probability zero.

\end{itemize}
The relations for condition (a) are not tautological on the $a_i$s since we can easily make $Q(a'_1,\dots,a'_k)$ full-dimensional and $H(a'_1,\dots,a'_k)$ intersect its interior as follows: choose each $a'_i$ in the interior of $-F_i$ and such that $\lineardist_{F_i}(a_i) = -1$, so that both $H$ and the interior of $Q$ contain the origin. The relations are finitely many since we have at most one for each subsystem (respectively equation) of the system $\mathcal S(a_1,\dots,a_k)$, which depends only on the choice of $F_i$s.
%\qed
\end{proof}

For the case of the tropical norm, condition (b) admits a nice combinatorial characterization.
Observe that a choice of facets $F_1,\dots,F_k \in \Fan(\ball{d})$ can be encoded as a directed graph on the vertex set $[d+1]$ and with an arc $a_i$ going from the coordinate that is minimized at $F_i$ to the coordinate that is maximized at $F_i$, for $i=1,\dots,k$.
We denote this graph $G(F_1,\dots,F_k)$.

\begin{proposition}
\label{prop:digraph}
For the case of the tropical norm, condition 2.(b) of Theorem~\ref{thm:genpos} holds if and only if
the graph $G(F_1,\dots,F_k)$ either has no (undirected) cycle or it has a unique cycle and it is unbalanced; that is, the number of arcs in one direction is different from the other direction.
\end{proposition}

\begin{proof}
A cycle in $G(F_1,\dots,F_k)$ is equivalent to a linear dependence among the corresponding linear functions $\lineardist_{F_i}$s, by simply adding them with signs corresponding to the direction of the arcs along the cycle.
If the cycle is balanced then $\lineardist_{F_1}$ can be subtracted from each $\lineardist_{F_i}$ so that the corresponding functions $(\lineardist_{F_i}- \lineardist_{F_1})$ are also dependent.
The same thing can be done if $G(F_1,\dots,F_k)$ has two different (unbalanced) cycles, since a linear combination of the two corresponding dependences can be made balanced.

Conversely, any linear dependence among the functions $(\lineardist_{F_i}- \lineardist_{F_1})$ corresponds to a balanced dependence among the corresponding $\lineardist_{F_i}$.
The latter either corresponds to a balanced circuit in the graph or decomposes into two (or more) linear dependences with distinct supports.
%\qed
\end{proof}

As a consequence of Corollary~\ref{coro:pure} the bisector of a set $S$ of $d+1$ points in general position is a finite set of points, which we call \emph{circumcenters} of~$S$.
In dimension two, three points in (weak) general position have a most one circumcenter, as we show in Corollary~\ref{coro:3points} below.
In higher dimension the same is known to be false for other polytopal norms \cite{IKLM95}, and here is an  example for the tropical norm:

\begin{example}[Non-uniqueness of circumcenters]
\label{exm:circumcenters}
Let us consider the four points $a_1= (0,2,3,3)$, $ a_2=(0,4,2,2)$, $ a_3=(2,4,1,1)$ and $a_4=(4,0,2,2)$.
Their bisector contains the points $x = (0, 0, 1, -1)$ and  $y = (0, 0, -1, 1)$.
Indeed, both $x$ and $y$ are at distance $4$ from all the $a_i$'s since we have
\begin{align*}
 a_1 - x = (0,2,2,4), & \qquad a_1 - y = (0,2,4,2), \\
 a_2 - x = (0,4,1,3), & \qquad a_2 - y = (0,4,3,1), \\
 a_3 - x = (2,4,0,2), & \qquad a_3 - y = (2,4,2,0), \\
 a_4 - x = (4,0,1,3), & \qquad a_4 - y = (4,0,3,1).
\end{align*}

The points $a_1,\dots,a_4$ are not in weak general position, as they lie in the plane $x_3-x_4 =0$.
However, they satisfy conditions (a) and (b) of Theorem~\ref{thm:genpos} for the polytopes $Q_x$ and $Q_y$ containing the circumcenters $x$ and $y$:
For condition (b) observe that the digraphs corresponding to $x$ and $y$ are, respectively, $\{14, 12, 32, 21\}$ and $\{13, 12, 42, 21\}$. They both have a single cycle, $\{12,21\}$, which is unbalanced.
Condition (a) follows from the fact that  $x$ (resp.\ $y$) is in the interior of all the cones whose intersection defines $Q_x$ (resp.\ $Q_y)$. 
%\todo[author=ref3]{The last 'sentence' in the penultimate paragraph in Example 3.11 is not a sentence.}
This is equivalent to the fact that all the vectors $a_i -x$ and $a_i-y$ have a unique maximum and a unique minimum entries.

Since condition (b) does not depend on the sites and condition (a) is, by Theorem~\ref{thm:genpos}, open, any perturbation of the sites will still produce at least two circumcenters. In particular, by 
Corollary~\ref{coro:genericity}, there are sites in general position for which this happens.
\end{example}

\subsection{Halfspheres, sectors, and the bisector of two points}
\label{sec:sectors}

The topology of a bisector is closely related to the following partition of $\partial K$.
Let $S\subseteq\R^d$ be a finite set of sites, and pick $a,b\in S$ distinct.
The \emph{open halfsphere in the direction of $b-a$}, denoted as $H(b-a)$, is the set of points in $\partial K$ whose exterior normal cone is contained in $(b-a)^\vee:= \smallSetOf{\lineardist}{ \lineardist(b-a) >0}$.
Informally, $H(b-a)$ are the faces of $K$ that \enquote{can be seen} from the direction $(b-a)$.
For a fixed site $a\in S$, the \emph{sector of $a$} is the set
\[
  H_S(a) \ = \ \bigcap_{b\in S\setminus\{a\}} H(b-a) \enspace.
\]
We denote $\mathcal{H}_S:= \smallSetOf{H_S(a)}{a \in S}$.
Observe that $H(b-a)$ and, hence, $H_S(a)$, are open in $\partial K$.

\begin{lemma}
\label{lem:halfsphere_facets}
Let $F_1,\dots, F_m$ be the facets of $K$ and let $\lineardist_{F_i}(x) \le 1$ be the valid linear inequality defining $F_i$. Then, for each $a\in S$,
\[
H_S(a) \ = \
\relint\left(
\bigcup \SetOf{F_i}{\lineardist_{F_i}(a) < \lineardist_{F_i}(b) \text{ for all } b\in S\backslash a}
\right) \enspace.
\]

In particular, $H_S(a)\cap H_S(b)=\emptyset$ for every $a,b\in S$ and, if $S$ is in weak general position,
\[
  \bigcup_{a\in S} \overline{H_S(a)} \ = \ \partial K \enspace.
\]
where $\overline{H_S(a)}$ denotes the topological closure of $H_S(a)$.
\end{lemma}

\begin{proof}
  It is clear from the definition that $H_S(a)$ contains the relative interior of every $F_i$ with $\lineardist_{F_i}(a) <\lineardist_{F_i}(b)$ for $ b\in S\setminus a$.
  By convexity of the cones $\smallSetOf{\lineardist}{\lineardist(b-a) >0}$ for each $a,b$, $H(b-a)$ (hence $H_S(a)$) also contains the relative interior of every lower dimensional face contained only in such facets.
  This proves the first formula.
  The second part follows from the first and the fact that in weak general position the minimum of each $\lineardist_{F_i}$ is attained at a single point of $S$.
%\qed
\end{proof}

\begin{remark}
\label{rem:halfsphere_not_full}
Assuming weak general position,
Lemma~\ref{lem:halfsphere_facets} allows us to think of  $\mathcal{H}_S$ as a labeling of the facets of $K$ by the elements of $S$ or, equivalently, as a map $\Fan(K) \to S$.
If $K$ is centrally symmetric, then each pair of opposite facets $F$ and $-F$ belong one to $H(b-a)$ and the other to $H(a-b)$. If $K$ is not, we can still guarantee that  $H(a-b)$ is never empty, and always disjoint from $H(b-a)$. As a consequence, $H(b-a)$ (and hence $H_S(a)$) cannot contain all the (relative interiors of) facets of~$K$.

For the case $K=\balld$ of the tropical ball this partition of the facets translates into something more meaningful.
Recall (see Proposition~\ref{prop:digraph} and the paragraph before it) that facets of $\balld$ can be represented as the arcs in the complete digraph on $d+1$ nodes. In particular, $\mathcal{H}_S$ colors the arcs by the points of $S$. Then:

\begin{enumerate}
\item Each coloring is a partial ordering of vertices, i.e. there is no mono\-chromatic cycle.
\item For the case of two points in general position, the two colors are opposite acyclic tournaments. In particular, there is a bijection between the possible halfspheres $H(b-a)$ and the total orderings of $d+1$ elements.
\end{enumerate}
\end{remark}

\begin{theorem}[\cite{IKLM95} for $d=3$]
  \label{thm:central_projection}
  Let $a,b\in\R^d$ be in weak general position.
  Then the central projection from~$a$ induces a homeomorphism between $\bisector(a,b)$ and $a + H(b-a)$.
  Hence, $\bisector(a,b)$ is homeomorphic to $\R^{d-1}$.
\end{theorem}

\begin{proof}
Let us first show that $\bisector(a,b)$ is contained in $a + \cone(H(b-a))$. To seek a contradiction, let $c\in\bisector(a,b)$ such that $c-a\notin \cone(H(b-a))$.
This implies that  the smallest ball centered at $a$ that contains $c$ touches it at a facet $F$ with functional $\lineardist_F$ such that $\lineardist_F(b-a) \leq 0$. Now, $c$ is equidistant to $a$ and $b$, and $a$ and $b$ cannot be in the same facet of the ball centered at $c$ (because they are in weak general position). Therefore, $\dist(c,\epsilon a+(1-\epsilon)b)< \dist(c,a)$, by convexity of the ball. This contradicts the fact that $\lineardist_{F_i}(b-a) \leq 0$.

Hence, we have a well-defined map $\phi: \bisector(a,b) \to a + H(b-a)$ given by central projection. The map $\phi$ is continuous since it is the restriction of central projection. It is also proper (that is, the inverse image of a compact set is compact) by a following argument: Let $C$ be a compact subset of $H(b-a)$. By continuity, $\phi^{-1}(C)$ is closed in $\bisector (a,b)$, hence in $\R^d$ since $\bisector (a,b)$ itself is closed (it is the zero set of the continuous function $d(x,a)-  d(x,b)$). Thus, we only need to prove that $\phi^{-1}(C)$ is bounded.
This follows from the fact that
\[
\phi^{-1}(C) \subset (a + \cone(K)) \ \cap \ (b+ \cone(\overline{H(a-b)} )) \enspace,
\]
and that $\cone(C)$ and $\cone(\overline{H(a-b)} )$ are two closed linear cones meeting only at the origin, since ${H(a-b)}$ and ${H(b-a)}$ are open and disjoint in $\partial K$.

Once we know $\phi$ is proper and continuous, we only need to check that this map is bijective in order for it to be a homeomorphism. To show this, we construct its inverse. For each $v\in H(b-a)$ we consider
the ray $r_v= \{a+\alpha v: \alpha \geq 0\}$. Along $r_v$, the distance to $a$ is linear in $\alpha$, the distance to $b$ is convex in $\alpha$ and both functions are continuous. Observe also that
\begin{align*}
  &\dist(a+0 v, a) =0 \,,  \quad \dist(a+0v, b)>0 \,, \\ 
  \text{ and } &\lim_{\alpha\rightarrow\infty}\dist(a+\alpha v,b)<\lim_{\alpha\rightarrow\infty}\dist(a+\alpha v,a) \enspace.
\end{align*}
  The last inequality comes from the fact that as the we move farther away from $a$ along $r_v$, eventually $(a+\alpha v)-b$ will be in the same cone of $\Fan(K)$ as $(a+\alpha v)-a=\alpha v$ (by weak general position), and $\langle b-a, v\rangle >0$ since $v\in H(b-a)$.

  Hence, the function $\alpha \mapsto \dist(a+\alpha v,b)-\dist(a+\alpha v, a)$ is negative at zero, positive at infinity, continuous, and convex. Therefore, it has exactly one root, which means $r_v$  intersects the bisector exactly once. We define $\psi(a+v)$ as this unique intersection point.

The maps $\phi$ and $\psi$ are clearly inverses of one another.
%\qed
\end{proof}

Looking at the proof, the reader can check that central projection gives a proper and continuous map from $\bisector(a,b)$ to $a + H(b-a)$ even without assuming weak general position.
We only need weak general position to construct the inverse.

\begin{corollary}
\label{coro:empty_sector}
If $S$ is in weak general position and there is an empty sector $H_S(a)\in\mathcal{H}_S$ then the bisector of $S$ is empty.
\end{corollary}

\begin{proof}
  Assume that there is a point $c\in\bisector(S)$.
  For a site $a\in S$ let us show that $H_S(a)\neq \emptyset$.
  By definition, $c\in\bisector(a,b)$ for $b\in S\setminus \{a\}$. By Proposition \ref{thm:central_projection}, each  $\bisector(a,b)$ can be mapped  to $H(b-a)$ by central projection.
  Since $c$ is in all of these bisectors, the central projection of $c$ into the ball $a+ K$ lies in $H(b-a)$ for all $b$, and hence in $H_S(a)$.
%\qed
\end{proof}

The converse of Proposition~\ref{coro:empty_sector} is true for \emph{three} points in arbitrary dimension (Theorem~\ref{thm:3points}) but not for more, even in general position, as the following example shows:

\begin{example}[Empty bisector, with non-empty sectors]
\label{exm:empty_bisector}
Consider $a= (1,-1,$ $0,0)$, $b=(-1,1,0,0)$, $c=(0,0,2,-2)$ and $d=(0,0,-2,2)$.
Then we have
\[
  \begin{split}
    \bisector(a,b) \ = \ & \SetOf{x}{x_3+1 \leq x_1, x_2, \leq x_4-1} \\
    &\quad \cup  \SetOf{x}{x_4+1 \leq x_1, x_2, \leq x_3-1} \cup  \SetOf{x}{x_1 = x_2} \enspace.
  \end{split}
\]
By symmetry, we also have
\[
  \begin{split}
    \bisector(c,d) \ = \ & \SetOf{x}{x_1+2 \leq x_3, x_4, \leq x_2-2} \\
    &\quad \cup \SetOf{x}{x_2+2 \leq x_3, x_4, \leq x_1-2} \cup \SetOf{x}{x_3 = x_4} \enspace.
  \end{split}
\]
Since $\bisector(a,b,c,d)$ lies in the intersection of the two, we have
\[
\bisector(a,b,c,d) \ \subseteq \ \SetOf{x}{x_1=x_2, x_3=x_4} \enspace.
\]
So for $x\in\bisector(a,b,c,d)$, we may assume $x_3=0$, which entails:
\[
\dist(a,x) \ = \ \max\{x_1 + 1, 0 \} - \min\{ x_1-1,0\} \ = \ \max\{|x_1| + 1, 2 \}\leq  |x_1| + 2 \enspace,
\]
with equality only when $x_1=0$ and
\[
\dist(c,x) \ = \ \max\{x_1, 2 \} - \min\{ x_1, -2\} \ = \ \max\{|x_1| + 2, 4 \} \geq |x_1|+2 \enspace,
\]
with equality only if $|x_1|\ge 2$. This shows that $\bisector(a,b,c,d)=\emptyset$.

However, $\mathcal H_S$ has no empty sector since the sectors of $a,b,c,d$ contain the facets with outer normals $(1,-1,0,0)$, $(-1,1,0,0)$, $(0,0,1,-1)$ and $(0, 0,$ $-1,1)$, respectively.
Both this property and the emptiness of $\bisector(a,b,c,d)$ are preserved under perturbation, so this behavior happens also in general position.
\end{example}

\subsection{Bisectors of three points}
\label{sec:3points}

The goal of this section is to prove our first main result.

\begin{theorem}
  \label{thm:3points}
  Let $S=\{a_1,a_2,a_3\}$ be a set of three distinct points in $\R^d$ which lie in weak general position with respect to a polytope $K$.
  If $H_S(a_i)\neq \emptyset$ for $i=1,2,3$ then $\bisector(a_1,a_2,a_3)$ is homeomorphic to a non-empty open subset of $\R^{d-2}$.
\end{theorem}

\begin{corollary}
  \label{coro:3points}
  For any three points in weak general position $\bisector(a_1,a_2,a_3)$ is either empty or pure of dimension $d-2$.
  If $d=2$ then $\bisector(a_1,a_2,a_3)$ is either empty or a single point.
\end{corollary}

We begin with the case $d=2$ of Theorem~\ref{thm:3points}; this occurs in \cite{IKLM95}:
\begin{lemma}
  \label{lem:dim2}
  Let $S=\{a_1,a_2,a_3\}\subset \R^2$ be in weak general position with respect to a polytope $K$.
  If $H_S(a_i)\neq \emptyset$ for the three of them, then the bisector $\bisector(S)$ is a single point.
\end{lemma}

\begin{proof}
  We first show that $\bisector(S)$ cannot contain two distinct points.
  To seek a contradiction, suppose that there are two points $x \neq y$ in the bisector.
  This means that there exist $\alpha,\beta>0$ such that $B_x := x-\alpha K$ and $B_y := y-\beta K$ satisfy $S\subset \partial B_x \cap \partial B_y$.
  For brevity we call $B_x$ and $B_y$ \emph{negative balls} centered at $x$ and $y$, respectively.
  The two balls are related to one another by a homothety $\rho: B_x \to B_y$.
  Let $a'_i := \rho(a_i)$ and $S'=\{a'_1,a'_2,a'_3\} = \rho(S)$.
  Then, $S\cup S' \subset \partial B_y$, but this is impossible: the vertex sets of two homothetic triangles can only lie in the boundary of a polytope in the plane if three of the points are collinear.
  Three collinear points of $\partial B_y$ necessarily lie on a single edge of $B_y$, and (at least) two of them would come from the same triangle $S$ or $S'$, violating the weak general position of $S$.

  It remains to exclude the case where $\bisector(S)=\emptyset$.
  Suppose that this would hold.
%  show that if $\bisector(S)$ is empty then one of the sectors is empty, too.
  That is, the two-point bisectors are pairwise disjoint.
  Now the two-point bisectors are homeomorphic to lines by Theorem~\ref{thm:central_projection}.
  So the fact that they do not meet implies that each of them appears either in full or not at all in the Voronoi diagram of $S$.
  But the three of them cannot appear, since then the diagram would have four regions, not three.
  Thus, one of them, say $\bisector(a_1,a_3)$, does not appear at all in $\Vor_S$.
  Consequently, $\bisector(a_1,a_3)$ is contained in the Voronoi region of the third point $a_2$.
  We will show that $H_S(a_2)=\emptyset$, and this yields the desired contradiction.

  To simplify the exposition, we call the line $a_1 a_3$ \enquote{horizontal}.
  Let $u$ and $v$ be the points where the ball $K$ has a horizontal tangent, i.e., parallel to $a_1 a_3$.
  Without loss of generality, $u$ is at the top of $K$, and $v$ is at the bottom; see Figure~\ref{tikz:two_boundary_rays}.
  The points $u$ and $v$ are unique and thus vertices of $K$, by weak general position.
  Consider the negative tangent cones $u+\cone(u-K)$ and $v+\cone(v-K)$ of $K$.
  Call $K_u$ and $K_v$ their respective translations that have $a_1$ and $a_3$ on the boundary, and let $u'$ and $v'$ be the apices of these translated cones.
  Such translations exist and are unique since the two boundary rays of $\cone(u-K)$ point upward, and the boundary rays of $\cone(v-K)$ point downward.

\begin{figure}[th]
\definecolor{uuuuuu}{rgb}{0.26666666666666666,0.26666666666666666,0.26666666666666666}
\begin{tikzpicture}[line cap=round,line join=round,>=triangle 45,x=.7cm,y=0.7cm,scale=0.6]
\clip(-4,-8) rectangle (4.5,4);
%\draw [line width=1.2pt] (-3,0) -- (0,1) -- (1.5,0.5) -- (1.5,-0.5) -- (0,-1) --cycle;
\draw [line width=1.2pt] (-3,0) -- (0,-2) -- (3,-1) -- (3,1) -- (-1,3) --cycle;

\draw[->] (-1.5,-1) -- ([xshift=-0.5cm,yshift=-0.75cm] -1.5,-1) ;
\draw[->] (1.5,-1.5) -- ([xshift=0.27cm,yshift=-0.81cm]1.5,-1.5) ;
\draw[->] (3,0) -- ([xshift=0.9cm,yshift=0cm]3,0) ;
\draw[->] (1,2) -- ([xshift=0.4cm,yshift=0.8cm]1,2) ;
\draw[->] (-2,1.5) -- ([xshift=-.75cm,yshift=0.5cm]-2,1.5) ;

\draw [line width=.5pt] (-1.5,3)--(-.5,3);
\draw (-1,3) node[anchor=south,fill=white] {$u$};
\draw [line width=.5pt] (-.5,-2)--(.5,-2);
\draw (0,-2) node[anchor=north,fill=white] {$v$};

\draw (0,0) node[fill=white] {$K$};

\end{tikzpicture}
\qquad
\begin{tikzpicture}[line cap=round,line join=round,>=triangle 45,x=0.5cm,y=0.5cm,scale=0.8]
\clip(-9,-5.5) rectangle (7,6.5);
%\draw[line width=1.2pt] (0,-.666)--(6,1) --(0,3)--(-3,1.333)--(-3,0.333)--cycle;
\draw [line width=1.2pt] (4.5,4) -- (-1.5,8) -- (-7.5,6) -- (-7.5,2) -- (.5,-2) --cycle;
\draw [line width=1.2pt] (6,-2) -- (0,2) -- (-6,0) -- (-6,-4) -- (2,-8) --cycle;

\draw [line width=0.6pt] (-9.5,3) -- (.5,-2) -- (8.5,10);
\draw [line width=0.6pt] (9,-4) -- (0,2) -- (-9,-1);

\draw[->] (1,1.33) -- ([xshift=-0.3cm,yshift=-0.45cm] 1,1.33) ;
\draw[->] (-1.8,1.4) -- ([xshift=0.166cm,yshift=-0.5cm]-1.8,1.4) ;
\draw[->] (-4.5,0.47) -- ([xshift=0.65cm,yshift=0cm]-4.5,0.47) ;
\draw[->] (-2,-0.75) -- ([xshift=0.25cm,yshift=0.5cm]-2,-0.75) ;
\draw[->] (1.2,-1) -- ([xshift=-.45cm,yshift=0.3cm]1.2,-1) ;

\coordinate[label = above:$a_1$] (a1) at (-4.5,0.47);
\coordinate[label = above:$a_2$] (a2) at (-.5,-.5);
\coordinate[label = right:$a_3$] (a3) at (2.2,0.55);
\node at (a1)[circle,fill,inner sep=1.5pt]{};
\node at (a2)[circle,fill,inner sep=1.5pt]{};
\node at (a3)[circle,fill,inner sep=1.5pt]{};

\draw (0,2) node[anchor=south] {$v'$};
\draw (.5,-2) node[anchor=north] {$u'$};

\end{tikzpicture}
\caption{%
  Illustrating the proof of Lemma~\ref{lem:dim2}.
  Left: polytope $K$ with top and bottom vertices, $u$ and $v$, horizontal tangents and facet normals.
  Right: intersection $K_u\cap K_v$ with $a_1$ and $a_3$ on the boundary.
  If $a_2\in K_u\cap K_v$ then all linear functions defined by exterior facet normals of $K$ or, equivalently, interior facet normals of $K_u$ and $K_v$, attain smaller values at $a_1$ or $a_3$ than at $a_2$.}
\label{tikz:two_boundary_rays}
\end{figure}
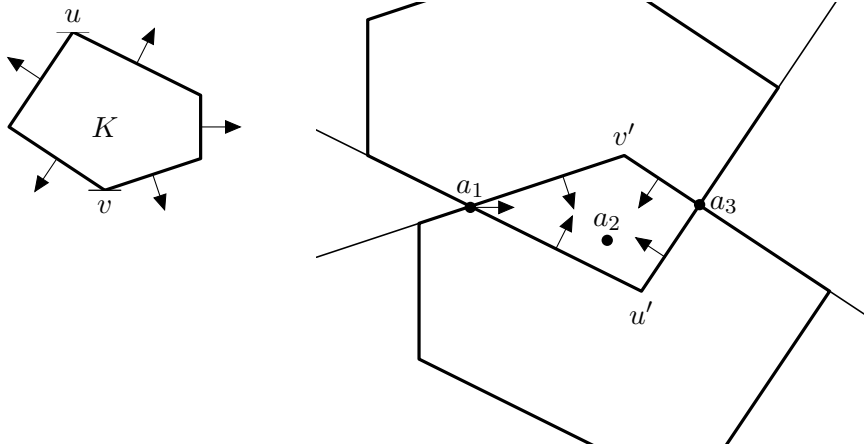

By construction the boundaries of $K_u$ and $K_v$ intersect precisely in $a_1$ and $a_3$.
Now, every negative ball of the form $B_{u',r}:=u'-r(K-u)$ has $K_u$  as its tangent cone at $u'$, and balls like $B_{v',r}:=v'-r(K-v)$ are similar.
Taking $r$ sufficiently large, we can force
\[
K_u \cap K_v \ = \ B_{u',r} \cap B_{v',r} \enspace.
\]
Since $B_{u',r}$ and $B_{v',r}$ are negative balls with $a_1$ and $a_3$ in their boundaries, their centers lie on the bisector $\bisector(a_1,a_3)$.
Our initial assumption that the bisector is contained in the Voronoi region of $a_2$ implies that $a_2$ is in the interior of both balls, i.e., in the interior of both $K_u$ and $K_v$. 
This implies that $H_S(a_2)$ is empty: every facet-defining functional of $K$ takes its minimum on $K_u \cap K_v$ at either $a_1$ or $a_3$.

We conclude that, if $H_S(a_i)\neq\emptyset$ for $i=1,2,3$ then $\bisector(S)\neq\emptyset$, and this finishes the proof.
%\qed
\end{proof}

For the rest of the proof of Theorem~\ref{thm:3points} let $S=\{a_1,a_2,a_3\}\subset \R^d$ be three points in general position with respect to a convex polytope $K$.
The idea is to reduce the general problem to two dimensions via the following construction.
Let $\proj:\R^d\rightarrow\R^{d-2}$ be the affine projection that quotients out the $2$-plane $\Pi$ containing $S$.
Next we exhibit relevant properties of that map.

\begin{lemma}
\label{lem:map_pi}
With the above notation, let $x\in \interior(\proj(K)) \subset \R^{d-2}$.
Further let $\Pi_x := \proj^{-1}(x)$, which is a $2$-plane parallel to $\Pi$, and  $K_x := K \cap \Pi_x$.
Then $K_x$ is a convex polygon.
\end{lemma}

\begin{proof}
This is a general property of polytope projections: if $Q$ is the image of a polytope $P$ under an affine map then the fiber of every point $x\in\interior(Q)$ is a polytope of dimension $\dim(P)-\dim(Q)$.
%\qed
\end{proof}

\begin{lemma}
  With the above notation, again let $x\in\interior(\proj(K))$.
  Further, let  $H_S^{(x)}(a_i)$ denote the sector of $a_i\in \Pi$ with respect to the polygon $K_x$.
  Then
  \[
    H_S^{(x)}(a_i) \ = \ H_S(a_i) \cap K_x \enspace.
  \]
\end{lemma}

\begin{proof}
  Note that no facet in $K_x$ is parallel to any $a_i-a_j$ because if it were, the facet of $K$ containing it would be parallel, too.

  Let $F'$ be a facet of $K_x$, and let $F$ be the corresponding facet in $K$. Then, $n(F')\in \Pi_x$, the normal vector to $F'$, is the projection into $\Pi$ of $n(F)$, and,
  \[
    \begin{split}
      F'\in H_S^{(x)}(a_i) \quad  & \iff \quad  \langle n(F'), a_i \rangle > \langle n(F'), a_j \rangle \text{ for } j=1,2,3 \\ & \iff \quad \langle n(F),a_i\rangle > \langle n(F), a_j\rangle \text{ for } j=1,2,3 \\ & \iff \quad F\in H_S(a_i)\enspace.
    \end{split}
  \]
%\qed
\end{proof}

\begin{lemma} \label{lem:domain}
  Let $S=\{a_1,a_2,a_3\}$ as before. If $H_S(a_i)\neq\emptyset$ for all $i$, then the intersection $\bigcap_{a_i\in S} \proj(H_S(a_i))\subset \R^{d-2}$ is open and not empty.
\end{lemma}

\begin{proof}
  First, observe that an $x\in \partial K$ with $\proj(x)\in \partial(\proj(K))$ cannot be in any of the $H_S(a_i)$: indeed, $x\in \partial K$ implies that there is a normal vector of $K$ at $x$ orthogonal to $\Pi$, hence orthogonal to $a_i-a_j$ for every $a_i,a_j$. As a consequence,
  \[
    \bigcap_{a_i\in S} \proj(H_S(a_i)) \ \subset \ \interior \proj(K) \enspace,
  \]
  which implies it is open.

  For any point $x\in \interior(\proj(K))$, the preimage $\proj^{-1}(x)$ is a polygon, a slice of $K$. This slice has to intersect at least two of the classes of $\mathcal{H}$, because $\mathcal{H}$ is a partition (so at the slice intersects at least one class), but no class can contain a set of facets whose vectors are positively dependent (because each class is an intersection of half-spheres). Then, any point $x\in\interior(K)$ lies in at least two sets $\proj(H_S(a_i))$.

  Thus, the three open sets $\proj(H_S(a_i))$ cover each point of $\interior(\proj(K))$ at least twice.
  At least two of these sets must intersect, say $\proj(H_S(a_1))$ and $\proj(H_S(a_2))$.
  Suppose that $(\proj(H_S(a_1))\cap\proj(H_S(a_2))) \cap \proj(H_S(a_3)) = \emptyset$. Then, $\interior(\proj(K))$ would be disconnected, because it is covered by two disjoint open sets. Since this is not possible, there must be a point in the common intersection of the three $H_S(a_i)$.
%\qed
\end{proof}

Our next lemma finishes the proof of Theorem~\ref{thm:3points}, and it actually gives more information.

\begin{lemma} \label{lem:homeo}
  The set $\bisector(a_1,a_2,a_3)$ is homeomorphic to $\bigcap_{i=1,2,3} \proj(H_S(a_i))$.
\end{lemma}

\begin{proof}
Consider the  map
\[
\phi: \bisector(a_1,a_2,a_3)\longrightarrow \bigcap_{i=1,2,3} \proj(H_S(a_i))
\]
defined as follows:
Let $p\in \bisector(S)$, and let $v_i\in H_S(a_i)$ be the central projection from $p$ to $a_i +\partial  K$, for each $i=1,2,3$.
Note that each $v_i$ lies in the corresponding $H_S(a_i)$, by Theorem~\ref{thm:central_projection}.
Further, the three points $v_1$, $v_2$ and $v_3$ lie in a plane parallel to $\Pi$. In particular, $\proj(v_1)=\proj(v_2)=\proj(v_3)$ lies in $\bigcap_{i=1,2,3} \proj(H_S(a_i))$ and we define
\[
\phi(p) \ := \ \proj(v_i) \enspace.
\]

To show that $\phi$ is a homeomorphism, let us construct its inverse $\psi$. Let $\gamma: \proj(\interior (K)) \to \interior(K)$ be a continuous section of $\proj$ in $K$. For example, but not necessarily, for each $2$-plane $\Pi'$ parallel to $\Pi$ and intersecting $K$ let $\gamma(\proj(\Pi'))$ be the centroid of $\Pi'\cap K$.

Now, let $x\in \bigcap_{i=1,2,3} \proj(H_S(a_i))$. Let $\Pi_x= \proj^{-1}(x)$ and let $w_i= \gamma(x) +a_i$, for each $i=1,2,3$. In the $2$-plane $\Pi_x$ we have a set $S_x=\{w_1,w_2,w_3\}$ and a unit ball $K\cap \Pi_x$.
Lemma~\ref{lem:map_pi} gives that $H_{S_x}(w_i) = H_S(a_i) \cap \Pi_x$. By choice of $x$ we have
 $\bigcap _i H_{S_x}(w_i) \neq \emptyset$ and Lemma~\ref{lem:dim2} guarantees that the bisector of $S_x$ is a unique point $r\in \Pi_x$; see Figure~\ref{fig:central_projection}.

 Let $v_i$ be the central projection of $r$ to $w_i+\partial K_x$. Observe that $|w_ir|/|w_i v_i|$ is independent of $i$ since it equals $\dist_{K_x}(w_i,r)$, where $\dist_{K_x}$ denotes the distance induced by $K_x$ in $\R^2$, and $r$ is in $\bisector_{K_X}(w_1,w_2,w_3)$. Since $w_i-a_i = \gamma(x)$ is also independent of $i$, the three rays $a_iv_i$ meet at the point
 \[
 p \:= \ r + \dist_{K_x}(w_i,r) \gamma(x) \enspace,
 \]
 and $\dist_K(a_i,p) = \dist_{K_X}(w_i,r)$. Thus, $p\in\bisector(a_1,a_2,a_3)$ and we define $p$ to be $\psi(x)$.

 This gives us a well-defined map
 \[
   \psi: \bigcap_{i=1,2,3} \proj(H_S(a_i)) \longrightarrow \bisector(a_1,a_2,a_3) \enspace ,
 \]
and by construction $\psi$ is the inverse of $\phi$ both ways.
The map $\gamma$ is continuous, as the bisector of three points in the plane depends continuously on a continuous deformation of the unit ball.
So $\psi$ is continuous.
Thus, $\phi$ and $\psi$ are homeomorphisms between $\bisector(a_1,a_2,a_3)$ and $\bigcap_{i=1,2,3} \proj(H_S(a_i))$.
Since the latter is not empty and open by Lemma~\ref{lem:domain}, the former is homeomorphic to a non-empty open subset of $\R^{d-2}$.
%\qed
\end{proof}

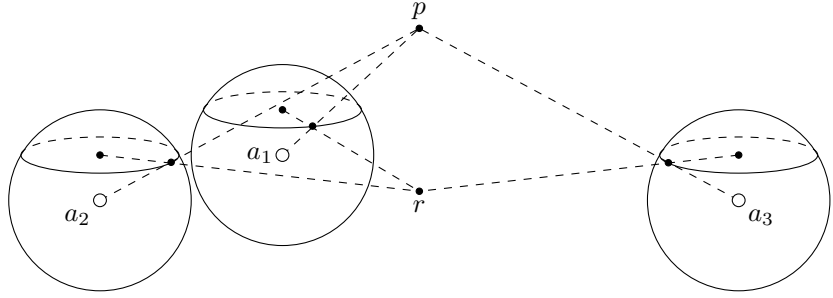
\begin{figure}\centering
  \begin{tikzpicture} [scale=1.2]
  \small

  \draw (0,0) circle (1cm);
  \draw (-0.866,0.5) arc (180:360:0.866 and 0.2);
  \draw[dashed] (0.866,0.5) arc (0:180:0.866 and 0.2);

  \draw (-2,-.5) circle (1cm);
  \draw (-2.866,0) arc (180:360:0.866 and 0.2);
  \draw[dashed] (0.866-2,0) arc (0:180:0.866 and 0.2);

  \draw (5,-.5) circle (1cm);
  \draw (-0.866+5,0) arc (180:360:0.866 and 0.2);
  \draw[dashed] (0.866+5,0) arc (0:180:0.866 and 0.2);

  \draw[dashed] (0,0) -- (1.5, 1.4);
  \draw[dashed](-2,-.5) -- (1.5, 1.4);
  \draw[dashed] (5,-.5) -- (1.5, 1.4);

  \draw[dashed] (0,0.5) -- (1.5, -.4);
  \draw[dashed](-2,0) -- (1.5, -.4);
  \draw[dashed] (5,0) -- (1.5, -.4);

  %points on the spheres
  \draw[fill=black] (-1.22, -0.08) circle(1pt);
  \draw[fill=black] (0.33, 0.32) circle(1pt);
  \draw[fill=black] (4.23,-0.085) circle(1pt);

  %circumcentres
  \draw[fill=black] (1.5,1.4) circle(1pt) node [anchor=south] {$p$};
  \draw[fill=black] (1.5, -.4) circle(1pt) node [anchor=north] {$r$};

  %three centers of spheres and circles
  \draw[fill=white, draw=black] (0,0) circle (2pt);
  \draw[fill=black] (0,0.5) circle (1pt);
  \draw(0,0) node[anchor=east] {$a_1$};

  \draw[fill=white, draw=black] (-2,-.5) circle (2pt);
  \draw[fill=black] (-2,0) circle (1pt);
  \draw(-2,-.5) node[anchor=north east] {$a_2$};

  \draw[fill=white, draw=black] (5, -.5) circle (2pt);
  \draw[fill=black] (5,0) circle (1pt);
  \draw(5,-.5) node[anchor=north west] {$a_3$};
  %\draw[dashed] (0,0 ) -- node[above]{$r$} (2,0);

  \end{tikzpicture}
  \caption{Construction of $\psi$ in the proof of Theorem~\ref{thm:3points} }
  \label{fig:central_projection}
\end{figure}

% Local Variables: 
% mode: latex
% mode: TeX-PDF
% TeX-master: "bisectors"
% mode: reftex
% mode: font-lock
% buffer-file-coding-system:utf-8-unix
% End: 

Our proof of Theorem~\ref{thm:3points} closely follows the proof of the $3$-dimensional case  in \cite{trisectors}. There, it is additionally shown that the number of connected components of $\bisector(a_1,a_2,a_3)$ equals the total number of connected components of the three sectors $H_S(a_1)$ minus two. We can extend this formula to higher dimension and to higher Betti numbers (over any field):

\begin{theorem}
\label{thm:3points_topo}
Let $a_1,a_2,a_3\in \R^d$ be three points in weak general position with respect to a polytope $K$ and assume that $H_S(a_i)\neq \emptyset$ for all three.
Then, for $j\in\{ 0,\dots, d-3 \}$, we have
\[
\tilde \beta_j(\bisector_K(a_1,a_2,a_3)) \ = \ \sum_{i=1}^3 \tilde \beta_j(H_S(a_i)) \enspace.
\]
\end{theorem}

\begin{proof}
Consider the same projection $\proj:\R^d \to \R^{d-1}$ as before. Observe that, since $\Pi_x \cap H_S(a_i)$ is empty or contractible for every plane $\Pi_x$ parallel to $\proj$, we have
\[
H_S(a_i) \ \simeq \ \pi(H_S(a_i)) \qquad \text{for } i=1,2,3 \enspace.
\]

We now apply Alexander duality in the one-point compactification $\mathcal S$ of $\interior(\proj(K))$, which is a sphere of dimension $d-2$. Alexander duality says that if $U$ is an open and locally contractible subset of $\mathcal S$ then
\[
\tilde \beta_j(U) \ = \ \tilde \beta_{d-3-j} (\mathcal S \setminus U) \enspace.
\]
In particular, if we let $C_i = \mathcal S \setminus \proj(H_S(a_i))$ we have
\[
 \sum_{i=1}^3 \tilde \beta_j(\proj( H_S(a_i))) \ = \ \sum_{i=1}^3 \tilde \beta_{d-3-j}(C_i) \enspace,
 \]
 and
\[
 \tilde \beta_j\left(\bigcap_{i=1}^3 \proj (H_S(a_i))\right) \ = \ \tilde \beta_{d-3-j}\left(\bigcup_{i=1}^3 C_i\right) \enspace.
\]
Yet $C_1$, $C_2$ and $C_3$ are pairwise disjoint except for the \enquote{point at infinity} of $\mathcal S$, because each point of $\interior(\proj(K)) $ lies in at least two of the sets $\proj(H_S(a_i))$.
Thus, $\bigcup_{i=1}^3 C_i$ is the topological wedge (or one-point sum) of $C_1$, $C_2$ and $C_3$, which makes the right-hand sides of the two last equations coincide.
%\qed
\end{proof}

\begin{remark}
One may ask how complicated the Betti numbers $\tilde \beta_j(H_S(a_i))$ in Theorem~\ref{thm:3points_topo} can be.
Equivalently, how complicated the topology of three point bisectors can be.
Such a bisector is $(d-2)$-dimensional, so the relevant Betti numbers are $\tilde \beta_0, \dots, \tilde \beta_{d-2}$.
The last one, $\tilde \beta_{d-2}$, must vanish as
$\tilde \beta_j(H_S(a_i)) = \tilde \beta_j(\proj(H_S(a_i)))$, and the latter is an open subset of $\R^{d-2}$.
But $\tilde \beta_{d-3}$ can be non-zero, as the following example shows.
\end{remark}

\begin{example}
  \label{exm:not_connected}
  Let $K=\ball{3}$ be the tropical unit ball, i.e., $d=3$.
  The three sites $a=(0,0,4,4)$, $b=(-3,0,2,0)$ and $c=(0,-3,0,2)$ lie in weak general position in $\torus4$.
  We describe each sector in terms of the facets whose relative interiors lie in that sector.
  The facet $F_{ij}$ is the one at which coordinate $i$ is minimized and $j$ is maximized.
  We obtain
  \[
    \begin{aligned}
      & H(a)=(F_{14},F_{23}) \,,\quad
      H(b)=(F_{12},F_{13},F_{32},F_{42},F_{43}) \,, \\
      &H(c)=(F_{21},F_{24},F_{31},F_{34},F_{41}) \enspace .
    \end{aligned}
  \]
  The sector $H(a)$ is disconnected; i.e., $\tilde \beta_{0}(H(a))$ is non-zero. 
  By Theorem~\ref{thm:3points_topo} $\bisector(a,b,c)$ is disconnected, too.
\end{example}

\begin{remark}\label{rem:convex-body}
The results in this section and in Section~\ref{sec:sectors} involve only \emph{weak} general position.
While we keep assuming that $K$ is a polytope for simplicity, these results generalize to an arbitrary convex body $K$.
In that case $S$ is in weak general position if the boundary $\partial K$ contains no segment parallel to a difference $a-b$ with $a,b\in S$.
\end{remark}

\section{Classification of tropical bisectors of two points}
\label{sec:2points}

\subsection{Tropical bisectors and tropical hypersurfaces}

In the classical case the bisector of two points is a degenerate quadric, namely the affine hyperplane perpendicular to the connecting line segment and which runs through the midpoint.
The tropical analog is more interesting.

\begin{proposition}\label{prop:hypersurface}
  Let $a,b\in\torus{d+1}$ be in weak general position.
  Then the homogeneous $\max$-tropical Laurent polynomial
  \begin{equation}\label{eq:polynomial}
    \bisectorpoly(a,b) \ = \ \max\left(\max_{i,j\in[d+1]}(x_i-a_i-x_j+a_j) , \max_{k,\ell\in[d+1]}(x_k-b_k-x_\ell+b_\ell)\right)
  \end{equation}
  vanishes on the bisector $\bisector(a,b)$.
  That is, the set $\bisector(a,b)$ is contained in a $\max$-tropical hypersurface of degree $d+1$.
\end{proposition}

\begin{figure}[th]
\definecolor{uuuuuu}{rgb}{0.26666666666666666,0.26666666666666666,0.26666666666666666}
\begin{tikzpicture}[line cap=round,line join=round,>=triangle 45,x=0.6cm,y=0.6cm,scale=0.7]
\clip(-2.5,-4.5) rectangle (5.5,5.5);
\draw [line width=1.2pt,dotted] (3.,-4.5) -- (3.,5.5);
\draw [line width=1.2pt,dotted,domain=-2.5:5.5] plot(\x,{(-0.--1.5*\x)/-2.598076211353316});
\draw [line width=1.2pt,dotted] (0.,-4.5) -- (0.,5.5);
\draw [line width=1.2pt,dotted,domain=-2.5:5.5] plot(\x,{(-4.494191659338056--1.5*\x)/2.598076211353316});
\draw [line width=1.2pt,dotted,domain=-2.5:5.5] plot(\x,{(-4.505808340661944--1.5*\x)/-2.598076211353316});
\draw [line width=1.2pt,dotted,domain=-2.5:5.5] plot(\x,{(-0.-0.5*\x)/-0.8660254037844385});
\draw [line width=2.pt] (1.5,-0.863789772421293)-- (1.5,0.8660254037844386);
\draw [line width=2.pt] (1.5,0.8660254037844386)-- (1.4980638864460185,0.8693788508291571);
\draw [line width=2.pt] (1.5,-0.863789772421293)-- (1.5019361135539813,-0.8671432194660114);
\draw [line width=2.pt] (1.4980638864460185,0.8693788508291571) -- (1.4980638864460185,5.5);
\draw [line width=2.pt] (1.5019361135539813,-0.8671432194660114) -- (1.5019361135539813,-4.5);
\draw [line width=2.pt,color=aqaqaq,domain=1.5:5.5] plot(\x,{(-0.--1.5*\x)/-2.598076211353316});
\draw [line width=2.pt,color=aqaqaq,domain=-2.5:1.5] plot(\x,{(-4.494191659338056--1.5*\x)/2.598076211353316});
\draw [line width=2.pt,color=aqaqaq,domain=-2.5:1.5] plot(\x,{(-4.505808340661944--1.5*\x)/-2.598076211353316});
\draw [line width=2.pt,color=aqaqaq,domain=1.5:5.5] plot(\x,{(-0.-0.5*\x)/-0.8660254037844385});
\begin{scriptsize}
\draw (-0.6869338062497061,0.0971017099450179) node[fill=white] {$a$};
\draw [fill=white,draw=black] (0.,0.) circle (2.5pt);
\draw (3.6136617561285123,0.07602035914904627) node[fill=white] {$b$};
\draw [fill=white,draw=black] (3.,0.002235631363145546) circle (2.5pt);
\draw [fill=uuuuuu] (1.5,0.8660254037844386) circle (1.0pt);
\draw [fill=uuuuuu] (1.5,-0.863789772421293) circle (1.0pt);
\draw [fill=uuuuuu] (1.4980638864460185,0.8693788508291571) circle (1.0pt);
\draw [fill=uuuuuu] (1.5019361135539813,-0.8671432194660114) circle (1.0pt);
\end{scriptsize}
\end{tikzpicture}
\qquad
\begin{tikzpicture}[line cap=round,line join=round,>=triangle 45,x=0.6cm,y=0.6cm,scale=0.7]
\clip(-2.5,-4.5) rectangle (5.5,5.5);
\draw [line width=1.2pt,dotted] (3.,-4.5) -- (3.,5.5);
\draw [line width=1.2pt,dotted,domain=-2.5:5.5] plot(\x,{(-0.--1.5*\x)/-2.598076211353316});
\draw [line width=1.2pt,dotted] (0.,-4.5) -- (0.,5.5);
\draw [line width=1.2pt,dotted,domain=-2.5:5.5] plot(\x,{(-1.8651857710400623--1.5*\x)/2.598076211353316});
\draw [line width=1.2pt,dotted,domain=-2.5:5.5] plot(\x,{(-7.134814228959938--1.5*\x)/-2.598076211353316});
\draw [line width=1.2pt,dotted,domain=-2.5:5.5] plot(\x,{(-0.-0.5*\x)/-0.8660254037844385});
\draw [line width=2.pt] (1.5,0.14811506578534553)-- (1.5,0.8660254037844386);
\draw [line width=2.pt] (1.5,0.8660254037844386)-- (0.6217285903466876,2.3872361081391147);
\draw [line width=2.pt] (1.5,0.14811506578534553)-- (2.3782714096533124,-1.3730956385693307);
\draw [line width=2.pt] (0.6217285903466876,2.3872361081391147) -- (0.6217285903466876,5.5);
\draw [line width=2.pt] (2.3782714096533124,-1.3730956385693307) -- (2.3782714096533124,-4.5);

\draw [line width=2.pt,color=aqaqaq,domain=2.3782714096533124:5.5] plot(\x,{(-0.--1.5*\x)/-2.598076211353316});
\draw [line width=2.pt,color=aqaqaq,domain=-2.5:1.5] plot(\x,{(-1.8651857710400623--1.5*\x)/2.598076211353316});
\draw [line width=2.pt,color=aqaqaq,domain=-2.5:0.6217285903466876] plot(\x,{(-7.134814228959938--1.5*\x)/-2.598076211353316});
\draw [line width=2.pt,color=aqaqaq,domain=1.5:5.5] plot(\x,{(-0.-0.5*\x)/-0.8660254037844385});
\begin{scriptsize}
\draw (-0.6869338062497061,0.0971017099450179) node[fill=white] {$a$};
\draw [fill=white,draw=black] (0.,0.) circle (2.5pt);
\draw (3.6136617561285123,1.087925197355685) node[fill=white] {$b'$};
\draw [fill=white,draw=black] (3.,1.0141404695697842) circle (2.5pt);
\draw [fill=uuuuuu] (1.5,0.8660254037844386) circle (1.0pt);
\draw [fill=uuuuuu] (1.5,0.14811506578534553) circle (1.0pt);
\draw [fill=uuuuuu] (0.6217285903466876,2.3872361081391147) circle (1.0pt);
\draw [fill=uuuuuu] (2.3782714096533124,-1.3730956385693307) circle (1.0pt);
\end{scriptsize}
\end{tikzpicture}
\qquad
\begin{tikzpicture}[line cap=round,line join=round,>=triangle 45,x=0.6cm,y=0.6cm,scale=0.7]
\clip(-2.5,-4.5) rectangle (5.5,5.5);
\fill[line width=0.pt,fill=black,fill]%opacity=0.10000000149011612
(-5.775621243975188,6.797513684148096) -- (9.910889429950332E-4,3.4623849987336675) -- (9.910889429950332E-4,9.889389154673843) -- (-4.987529368627924,9.910470505469815) -- cycle;
\fill[line width=0.pt,fill=black,fill]%opacity=0.10000000149011612
(2.9990089110570053,-1.7314786021008486) -- (2.9990089110570053,-8.198409828269822) -- (7.028840585075921,-9.400046823640205) -- (8.660691645410237,-5.000252652845942) -- cycle;
\draw [line width=1.2pt,dotted] (3.,-4.5) -- (3.,5.5);
\draw [line width=1.2pt,dotted,domain=-2.5:5.5] plot(\x,{(-0.--1.5*\x)/-2.598076211353316});
\draw [line width=1.2pt,dotted] (0.,-4.5) -- (0.,5.5);
\draw [line width=2.pt,color=aqaqaq,domain=-2.5:5.5] plot(\x,{(-0.0029732668289836894--1.5*\x)/2.598076211353316});
\draw [line width=1.2pt,dotted,domain=-2.5:5.5] plot(\x,{(-8.997026733171015--1.5*\x)/-2.598076211353316});
%\draw [line width=1.2pt,dotted,domain=-2.5:5.5] plot(\x,{(-0.-0.5*\x)/-0.8660254037844385});
\draw [line width=2.pt] (1.5,0.864880992848381)-- (1.5,0.8660254037844386);
\draw [line width=2.pt] (1.5,0.8660254037844386)-- (9.910889429950332E-4,3.4623849987336675);
\draw [line width=2.pt] (1.5,0.864880992848381)-- (2.9990089110570053,-1.7314786021008486);
\draw [line width=2.pt] (9.910889429950332E-4,3.4623849987336675) -- (9.910889429950332E-4,5.5);
\draw [line width=2.pt] (2.9990089110570053,-1.7314786021008486) -- (2.9990089110570053,-4.5);
\draw [line width=2.pt,domain=-2.5:9.910889429950332E-4] plot(\x,{(-6.264961132960372--1.0445052546963396*\x)/-1.809136169906729});
\draw [line width=2.pt,domain=2.9990089110570053:5.5] plot(\x,{(-0.-0.6852820682162759*\x)/1.1869433596664716});
\begin{scriptsize}
\draw (-0.686933806249707,0.0971017099450179) node [fill=white] {$a$};
\draw [fill=white,draw=black] (0.,0.) circle (2.5pt);
\draw (3.71366175612851,1.8046911244187207) node[fill=white] {$b''$};
\draw [fill=white,draw=black] (3.,1.7309063966328198) circle (2.5pt);
\draw [fill=uuuuuu] (1.5,0.8660254037844386) circle (1.0pt);
\draw [fill=uuuuuu] (1.5,0.864880992848381) circle (1.0pt);
\draw [fill=uuuuuu] (9.910889429950332E-4,3.4623849987336675) circle (1.0pt);
\draw [fill=uuuuuu] (2.9990089110570053,-1.7314786021008486) circle (1.0pt);
\end{scriptsize}
\end{tikzpicture}
\caption{Tropical bisectors for $b-a=(-1,1,0)$, $(-1,1,\tfrac{1}{2})$ and $(-1,1,1)$, respectively.   The first two are in weak general position but the last one is not. Only the middle one is in general  position.
The bisectors are shown in black and the rest of the tropical hypersurface containing it in gray; see Proposition~\ref{prop:hypersurface}.}
\label{tikz:bisectors}
\end{figure}
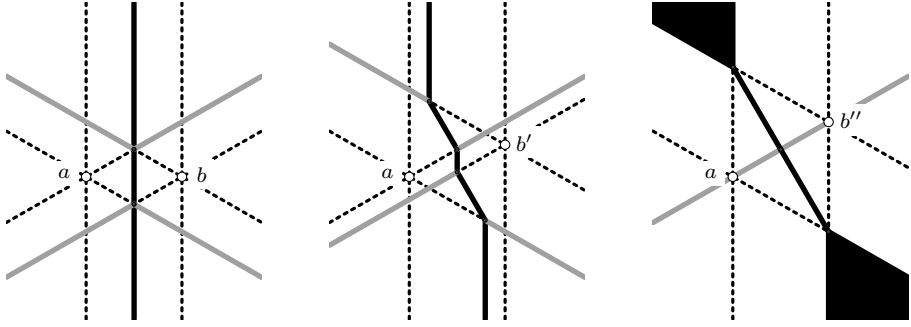

\begin{proof}
  Recall that a $\max$-tropical (Laurent) polynomial vanishes if the maximum is attained at least twice; see \cite[\S3.1]{Tropical+Book}.
  First, we check that there are no duplicates among the terms in the representation \eqref{eq:polynomial} of $\bisectorpoly(a,b)$.
  Assume the contrary, i.e., $x_i-a_i-x_j+a_j = x_i-b_i-x_j+b_j$ for some $i,j\in[d+1]$.
  Then $a_j-a_i=b_j-b_i$, which forces $\langle e_j-e_i, b-a \rangle =(b_j-a_j)-(b_i-a_i)=0$.
  Thus $b-a$ is parallel to the facet of $\ball{d}$ with normal vector $e_j-e_i$.
  This was explicitly excluded in our assumption, and we arrive at the desired contradiction.
  We infer that the $2d(d+1)$ terms are pairwise distinct.

  Let $x\in\bisector(a,b)$.
  This means that $\dist(a,x)=\dist(b,x)$, and thus
  \[ \max_{i,j\in [d+1]}(x_i-a_i-x_j+a_j) \ = \ \max_{k,\ell\in [d+1]} (x_k -b_k -x_\ell+b_\ell) \enspace. \]
  It follows that $\bisectorpoly(a,b)$ vanishes at $x$.

  The degree of the bisector tropical hypersurface can be read off any Laurent monomial like $x_i-x_j$ by adding $x_1+x_2+\dots+x_{d+1}$, which yields the true monomial $2x_i+ \sum_{k\in[d+1]-\{i,j\}} x_k$.
  The latter has degree $2+(d+1-2)=d+1$.
%\qed
\end{proof}

Proposition~\ref{prop:hypersurface} yields a trivial algorithm to compute tropical bisectors in weak general position: enumerate the maximal cells of the tropical hypersurface defined by \eqref{eq:polynomial} and select those maximal cells that attain maxima in one monomial of type $x_i-a_i-x_j+a_j$ and one monomial of type $x_k-b_k-x_\ell+b_\ell$.
This algorithm needs to go through the $\Theta(d^4)$ choices of one monomial from the left and one from the right.
This is worst case optimal, as we will prove in Corollary~\ref{cor:d^4} that tropical bisectors can have $\Omega(d^4)$ maximal cells.

\begin{example}\label{exm:labeling}
  The labeling of the faces of a tropical bisector does not need to be unique if $a$ and $b$ are not in weak general position.
  For instance, if $b-a=(-1,1,0)$ then
  \[
    \begin{split}
      &\bisector_{(-+*),(+-*)}(a,b) \ = \ \bisector_{(-++),(+-+)}(a,b) \\
      &\quad = \ \bisector_{(-++),(+--)}(a,b) \ = \ \bisector_{(-+-),(+-+)}(a,b) \\
      &\quad= \ \bisector_{(-+-),(+--)}(a,b)
    \end{split}
  \]
  is the only face; see Figure~\ref{tikz:bisectors} (left).
\end{example}

\subsection{The bisection fan}
\label{ssec:ccFan}

Normal equivalence of tropical bisectors is preserved  by translation and scaling. In particular the equivalence class of $\bisector(a,b)$ is uniquely determined by the direction of the vector $b-a$.
The \emph{bisection fan} $\ccF$ is the complete polyhedral fan in $\torus{d+1}$ whose relatively open cones are defined by \enquote{$x$ and $y$ lie in the same cone if and only if $\bisector(0,x)$ and $\bisector(0,y)$ are normally equivalent}. Put differently, two points $a,b\in\torus{d+1}$ are in \emph{general position} if and only if the difference $b-a$ lies in a maximal cone of $\ccF$.
In the rest of this section we show that $\ccF$ is indeed a polyhedral fan and give an explicit description of it.

Recall that an \emph{ordered partition} or \emph{total preorder} on a finite set $S$ is a partition of $S$ into non-empty parts together with a total order on the parts. If the parts are denoted $S_1,\dots,S_k$ (in this order), we can write $x \le y$ meaning \enquote{$x\in S_i$ and $y\in S_j$ for some $i\leq j$}. In particular, for all $x,y \in S$ we have $x\le y \le x$  if and only if $x$ and $y$ lie in the same part.

Any real vector $v=(v_1,\dots,v_{d+1})\in \R^{d+1}$ induces an ordered partition $S(v)$ of $[d+1]$ by putting together the coordinates that have the same value and ordering the groups according to their values.
For example, the vector $v=(3, 1, 6, 4, 6, 3, 1)$ of length seven induces the partition $(\{2,7\}, \{1,6\}, \{4\}$, $\{3,5\})$ of the set $\{1,2,\dots,7\}$ into four parts.
Note that the ordered partition $S(v)$ is constant on the class $v+\R\ones$.
Hence these ordered partitions are defined for points in the projective tropical torus $\torus{d+1}$.

For $v,w\in \R^{d+1}$ with $S(v)=S(w)$ we have $S(v+w)=S(v)$ and, moreover, $S(\alpha v)=S(v)$ for any positive real $\alpha$.
That is to say, the stratification of $\R^{d+1}$ by ordered partitions forms a complete polyhedral fan.
In what follows we seek to refine that fan by recording which part, or which gap between parts, contains the \emph{midvalue}
\[
  \midv(v) \ := \ \frac{1}{2}\bigl(\max_{i\in[d+1]} v_i + \min_{i\in[d+1]} v_i\bigr) \enspace .
\]

\begin{figure}[tb]
  % \input{three_fans.tikz}
  % polymake for gast
% Thu Mar 14 15:08:58 2019
% p11

\begin{tikzpicture}[x  = {(0.949970699338162cm,-0.0509944119960769cm)},
                    y  = {(-0.00366732069391268cm,0.984692961385677cm)},
                    z  = {(-0.31231750056295cm,-0.166671358495326cm)},
                    scale = 0.88,
                    color = {lightgray}]

  % POINTS STYLE
  \definecolor{pointcolor_p11}{rgb}{ 1,0,0 }
  \tikzstyle{pointstyle_p11} = [fill=pointcolor_p11]

  % DEF POINTS
  \coordinate (v0_p11) at (-3, 0, 0);
  \coordinate (v1_p11) at (-7, 0, 0);
  \coordinate (v2_p11) at (-5, 2, 0);
  \coordinate (v3_p11) at (-5, -2, 0);
  \coordinate (v4_p11) at (-5, 0, 2);
  \coordinate (v5_p11) at (-5, 0, -2);
  \coordinate (v6_p11) at (-4, 1, 1);
  \coordinate (v7_p11) at (-4, 1, -1);
  \coordinate (v8_p11) at (-4, -1, 1);
  \coordinate (v9_p11) at (-4, -1, -1);
  \coordinate (v10_p11) at (-6, 1, 1);
  \coordinate (v11_p11) at (-6, 1, -1);
  \coordinate (v12_p11) at (-6, -1, 1);
  \coordinate (v13_p11) at (-6, -1, -1);

  % EDGES STYLE
  \definecolor{edgecolor_p11}{rgb}{ 0,0,0 }

  % FACES STYLE
  \definecolor{facetcolor_p11}{rgb}{ 0.4667,0.9255,0.6196 }

  \tikzstyle{facestyle_p11} = [fill=facetcolor_p11, fill opacity=0.85, draw=edgecolor_p11, line width=1 pt, line cap=round, line join=round]

  % FACES and EDGES and POINTS in the right order
  \draw[facestyle_p11] (v2_p11) -- (v11_p11) -- (v1_p11) -- (v10_p11) -- (v2_p11) -- cycle;
  \draw[facestyle_p11] (v1_p11) -- (v13_p11) -- (v3_p11) -- (v12_p11) -- (v1_p11) -- cycle;
  \draw[facestyle_p11] (v5_p11) -- (v7_p11) -- (v0_p11) -- (v9_p11) -- (v5_p11) -- cycle;
  \draw[facestyle_p11] (v2_p11) -- (v7_p11) -- (v5_p11) -- (v11_p11) -- (v2_p11) -- cycle;
  \draw[facestyle_p11] (v5_p11) -- (v9_p11) -- (v3_p11) -- (v13_p11) -- (v5_p11) -- cycle;
  \draw[facestyle_p11] (v11_p11) -- (v5_p11) -- (v13_p11) -- (v1_p11) -- (v11_p11) -- cycle;

  %POINTS
  \fill[pointcolor_p11] (v11_p11) circle (1 pt);
  \node at (v11_p11) [text=black, inner sep=0.5pt, above right, draw=none, align=left] {};
  \fill[pointcolor_p11] (v5_p11) circle (1 pt);
  \node at (v5_p11) [text=black, inner sep=0.5pt, above right, draw=none, align=left] {};
  \fill[pointcolor_p11] (v13_p11) circle (1 pt);
  \node at (v13_p11) [text=black, inner sep=0.5pt, above right, draw=none, align=left] {};

  %FACETS
  \draw[facestyle_p11] (v0_p11) -- (v8_p11) -- (v3_p11) -- (v9_p11) -- (v0_p11) -- cycle;

  %POINTS
  \fill[pointcolor_p11] (v9_p11) circle (1 pt);
  \node at (v9_p11) [text=black, inner sep=0.5pt, above right, draw=none, align=left] {};

  %FACETS
  \draw[facestyle_p11] (v0_p11) -- (v7_p11) -- (v2_p11) -- (v6_p11) -- (v0_p11) -- cycle;

  %POINTS
  \fill[pointcolor_p11] (v7_p11) circle (1 pt);
  \node at (v7_p11) [text=black, inner sep=0.5pt, above right, draw=none, align=left] {};

  %FACETS
  \draw[facestyle_p11] (v4_p11) -- (v10_p11) -- (v1_p11) -- (v12_p11) -- (v4_p11) -- cycle;

  %POINTS
  \fill[pointcolor_p11] (v1_p11) circle (1 pt);
  \node at (v1_p11) [text=black, inner sep=0.5pt, above right, draw=none, align=left] {};

  %FACETS
  \draw[facestyle_p11] (v8_p11) -- (v4_p11) -- (v12_p11) -- (v3_p11) -- (v8_p11) -- cycle;

  %POINTS
  \fill[pointcolor_p11] (v12_p11) circle (1 pt);
  \node at (v12_p11) [text=black, inner sep=0.5pt, above right, draw=none, align=left] {};
  \fill[pointcolor_p11] (v3_p11) circle (1 pt);
  \node at (v3_p11) [text=black, inner sep=0.5pt, above right, draw=none, align=left] {};

  %FACETS
  \draw[facestyle_p11] (v6_p11) -- (v2_p11) -- (v10_p11) -- (v4_p11) -- (v6_p11) -- cycle;

  %POINTS
  \fill[pointcolor_p11] (v2_p11) circle (1 pt);
  \node at (v2_p11) [text=black, inner sep=0.5pt, above right, draw=none, align=left] {};
  \fill[pointcolor_p11] (v10_p11) circle (1 pt);
  \node at (v10_p11) [text=black, inner sep=0.5pt, above right, draw=none, align=left] {};

  %FACETS
  \draw[facestyle_p11] (v0_p11) -- (v6_p11) -- (v4_p11) -- (v8_p11) -- (v0_p11) -- cycle;

  %POINTS
  \fill[pointcolor_p11] (v0_p11) circle (1 pt);
  \node at (v0_p11) [text=black, inner sep=0.5pt, above right, draw=none, align=left] {};
  \fill[pointcolor_p11] (v6_p11) circle (1 pt);
  \node at (v6_p11) [text=black, inner sep=0.5pt, above right, draw=none, align=left] {};
  \fill[pointcolor_p11] (v4_p11) circle (1 pt);
  \node at (v4_p11) [text=black, inner sep=0.5pt, above right, draw=none, align=left] {};
  \fill[pointcolor_p11] (v8_p11) circle (1 pt);
  \node at (v8_p11) [text=black, inner sep=0.5pt, above right, draw=none, align=left] {};

  %FACETS

  % POINTS STYLE
  \definecolor{pointcolor_p2}{rgb}{ 1,0,0 }
  \tikzstyle{pointstyle_p2} = [fill=pointcolor_p2]

  % DEF POINTS
  \coordinate (v0_p2) at (1.9, 0, 0);
  \coordinate (v1_p2) at (-1.9, 0, 0);
  \coordinate (v2_p2) at (0, 1.9, 0);
  \coordinate (v3_p2) at (0, -1.9, 0);
  \coordinate (v4_p2) at (0, 0, 1.9);
  \coordinate (v5_p2) at (0, 0, -1.9);
  \coordinate (v6_p2) at (1, 1, 1);
  \coordinate (v7_p2) at (1, 1, -1);
  \coordinate (v8_p2) at (1, -1, 1);
  \coordinate (v9_p2) at (1, -1, -1);
  \coordinate (v10_p2) at (-1, 1, 1);
  \coordinate (v11_p2) at (-1, 1, -1);
  \coordinate (v12_p2) at (-1, -1, 1);
  \coordinate (v13_p2) at (-1, -1, -1);

  % EDGES STYLE
  \definecolor{edgecolor_p2}{rgb}{ 0,0,0 }

  % FACES STYLE
  \definecolor{facetcolor_p2}{rgb}{ 0.4667,0.9255,0.6196 }

  \tikzstyle{facestyle_p2} = [fill=facetcolor_p2, fill opacity=0.85, draw=edgecolor_p2, line width=1 pt, line cap=round, line join=round]

  % FACES and EDGES and POINTS in the right order
  \draw[facestyle_p2] (v3_p2) -- (v12_p2) -- (v13_p2) -- (v3_p2) -- cycle;
  \draw[facestyle_p2] (v12_p2) -- (v1_p2) -- (v13_p2) -- (v12_p2) -- cycle;
  \draw[facestyle_p2] (v0_p2) -- (v9_p2) -- (v7_p2) -- (v0_p2) -- cycle;
  \draw[facestyle_p2] (v7_p2) -- (v9_p2) -- (v5_p2) -- (v7_p2) -- cycle;
  \draw[facestyle_p2] (v2_p2) -- (v7_p2) -- (v11_p2) -- (v2_p2) -- cycle;
  \draw[facestyle_p2] (v11_p2) -- (v7_p2) -- (v5_p2) -- (v11_p2) -- cycle;
  \draw[facestyle_p2] (v3_p2) -- (v13_p2) -- (v9_p2) -- (v3_p2) -- cycle;
  \draw[facestyle_p2] (v9_p2) -- (v13_p2) -- (v5_p2) -- (v9_p2) -- cycle;
  \draw[facestyle_p2] (v1_p2) -- (v11_p2) -- (v13_p2) -- (v1_p2) -- cycle;
  \draw[facestyle_p2] (v11_p2) -- (v5_p2) -- (v13_p2) -- (v11_p2) -- cycle;

  %POINTS
  \fill[pointcolor_p2] (v5_p2) circle (1 pt);
  \node at (v5_p2) [text=black, inner sep=0.5pt, above right, draw=none, align=left] {};
  \fill[pointcolor_p2] (v13_p2) circle (1 pt);
  \node at (v13_p2) [text=black, inner sep=0.5pt, above right, draw=none, align=left] {};

  %FACETS
  \draw[facestyle_p2] (v10_p2) -- (v2_p2) -- (v11_p2) -- (v10_p2) -- cycle;
  \draw[facestyle_p2] (v10_p2) -- (v11_p2) -- (v1_p2) -- (v10_p2) -- cycle;

  %POINTS
  \fill[pointcolor_p2] (v11_p2) circle (1 pt);
  \node at (v11_p2) [text=black, inner sep=0.5pt, above right, draw=none, align=left] {};

  %FACETS
  \draw[facestyle_p2] (v2_p2) -- (v6_p2) -- (v7_p2) -- (v2_p2) -- cycle;
  \draw[facestyle_p2] (v6_p2) -- (v0_p2) -- (v7_p2) -- (v6_p2) -- cycle;

  %POINTS
  \fill[pointcolor_p2] (v7_p2) circle (1 pt);
  \node at (v7_p2) [text=black, inner sep=0.5pt, above right, draw=none, align=left] {};

  %FACETS
  \draw[facestyle_p2] (v12_p2) -- (v10_p2) -- (v1_p2) -- (v12_p2) -- cycle;

  %POINTS
  \fill[pointcolor_p2] (v1_p2) circle (1 pt);
  \node at (v1_p2) [text=black, inner sep=0.5pt, above right, draw=none, align=left] {};

  %FACETS
  \draw[facestyle_p2] (v12_p2) -- (v4_p2) -- (v10_p2) -- (v12_p2) -- cycle;
  \draw[facestyle_p2] (v8_p2) -- (v12_p2) -- (v3_p2) -- (v8_p2) -- cycle;
  \draw[facestyle_p2] (v8_p2) -- (v4_p2) -- (v12_p2) -- (v8_p2) -- cycle;

  %POINTS
  \fill[pointcolor_p2] (v12_p2) circle (1 pt);
  \node at (v12_p2) [text=black, inner sep=0.5pt, above right, draw=none, align=left] {};

  %FACETS
  \draw[facestyle_p2] (v10_p2) -- (v6_p2) -- (v2_p2) -- (v10_p2) -- cycle;

  %POINTS
  \fill[pointcolor_p2] (v2_p2) circle (1 pt);
  \node at (v2_p2) [text=black, inner sep=0.5pt, above right, draw=none, align=left] {};

  %FACETS
  \draw[facestyle_p2] (v10_p2) -- (v4_p2) -- (v6_p2) -- (v10_p2) -- cycle;

  %POINTS
  \fill[pointcolor_p2] (v10_p2) circle (1 pt);
  \node at (v10_p2) [text=black, inner sep=0.5pt, above right, draw=none, align=left] {};

  %FACETS
  \draw[facestyle_p2] (v6_p2) -- (v8_p2) -- (v0_p2) -- (v6_p2) -- cycle;
  \draw[facestyle_p2] (v6_p2) -- (v4_p2) -- (v8_p2) -- (v6_p2) -- cycle;

  %POINTS
  \fill[pointcolor_p2] (v6_p2) circle (1 pt);
  \node at (v6_p2) [text=black, inner sep=0.5pt, above right, draw=none, align=left] {};
  \fill[pointcolor_p2] (v4_p2) circle (1 pt);
  \node at (v4_p2) [text=black, inner sep=0.5pt, above right, draw=none, align=left] {};

  %FACETS
  \draw[facestyle_p2] (v8_p2) -- (v3_p2) -- (v9_p2) -- (v8_p2) -- cycle;

  %POINTS
  \fill[pointcolor_p2] (v3_p2) circle (1 pt);
  \node at (v3_p2) [text=black, inner sep=0.5pt, above right, draw=none, align=left] {};

  %FACETS
  \draw[facestyle_p2] (v0_p2) -- (v8_p2) -- (v9_p2) -- (v0_p2) -- cycle;

  %POINTS
  \fill[pointcolor_p2] (v0_p2) circle (1 pt);
  \node at (v0_p2) [text=black, inner sep=0.5pt, above right, draw=none, align=left] {};
  \fill[pointcolor_p2] (v8_p2) circle (1 pt);
  \node at (v8_p2) [text=black, inner sep=0.5pt, above right, draw=none, align=left] {};
  \fill[pointcolor_p2] (v9_p2) circle (1 pt);
  \node at (v9_p2) [text=black, inner sep=0.5pt, above right, draw=none, align=left] {};

  %FACETS

  % POINTS STYLE
  \definecolor{pointcolor_p33}{rgb}{ 1,0,0 }
  \tikzstyle{pointstyle_p33} = [fill=pointcolor_p33]

  % DEF POINTS
  \coordinate (v0_p33) at (5, 1, 1);
  \coordinate (v1_p33) at (5, 1, -1);
  \coordinate (v2_p33) at (5, -1, 1);
  \coordinate (v3_p33) at (5, -1, -1);
  \coordinate (v4_p33) at (6, 0, 1);
  \coordinate (v5_p33) at (6, 0, -1);
  \coordinate (v6_p33) at (4, 0, 1);
  \coordinate (v7_p33) at (4, 0, -1);
  \coordinate (v8_p33) at (6, 1, 0);
  \coordinate (v9_p33) at (6, -1, 0);
  \coordinate (v10_p33) at (4, 1, 0);
  \coordinate (v11_p33) at (4, -1, 0);
  \coordinate (v12_p33) at (5.99, 0.99, 0.99);
  \coordinate (v13_p33) at (5.99, 0.99, -0.99);
  \coordinate (v14_p33) at (5.99, -0.99, 0.99);
  \coordinate (v15_p33) at (5.99, -0.99, -0.99);
  \coordinate (v16_p33) at (4.01, 0.99, 0.99);
  \coordinate (v17_p33) at (4.01, 0.99, -0.99);
  \coordinate (v18_p33) at (4.01, -0.99, 0.99);
  \coordinate (v19_p33) at (4.01, -0.99, -0.99);
  \coordinate (v20_p33) at (6.45, 0.5, 0.5);
  \coordinate (v21_p33) at (6.45, 0.5, -0.5);
  \coordinate (v22_p33) at (6.45, -0.5, 0.5);
  \coordinate (v23_p33) at (6.45, -0.5, -0.5);
  \coordinate (v24_p33) at (5.5, 1.45, 0.5);
  \coordinate (v25_p33) at (5.5, 1.45, -0.5);
  \coordinate (v26_p33) at (4.5, 1.45, 0.5);
  \coordinate (v27_p33) at (4.5, 1.45, -0.5);
  \coordinate (v28_p33) at (5.5, 0.5, 1.45);
  \coordinate (v29_p33) at (5.5, -0.5, 1.45);
  \coordinate (v30_p33) at (4.5, 0.5, 1.45);
  \coordinate (v31_p33) at (4.5, -0.5, 1.45);
  \coordinate (v32_p33) at (3.55, 0.5, 0.5);
  \coordinate (v33_p33) at (3.55, 0.5, -0.5);
  \coordinate (v34_p33) at (3.55, -0.5, 0.5);
  \coordinate (v35_p33) at (3.55, -0.5, -0.5);
  \coordinate (v36_p33) at (5.5, -1.45, 0.5);
  \coordinate (v37_p33) at (5.5, -1.45, -0.5);
  \coordinate (v38_p33) at (4.5, -1.45, 0.5);
  \coordinate (v39_p33) at (4.5, -1.45, -0.5);
  \coordinate (v40_p33) at (5.5, 0.5, -1.45);
  \coordinate (v41_p33) at (5.5, -0.5, -1.45);
  \coordinate (v42_p33) at (4.5, 0.5, -1.45);
  \coordinate (v43_p33) at (4.5, -0.5, -1.45);
  \coordinate (v44_p33) at (6.9, 0, 0);
  \coordinate (v45_p33) at (3.1, 0, 0);
  \coordinate (v46_p33) at (5, 1.9, 0);
  \coordinate (v47_p33) at (5, -1.9, 0);
  \coordinate (v48_p33) at (5, 0, 1.9);
  \coordinate (v49_p33) at (5, 0, -1.9);

  % EDGES STYLE
  \definecolor{edgecolor_p33}{rgb}{ 0,0,0 }

  % FACES STYLE
  \definecolor{facetcolor_p33}{rgb}{ 0.4667,0.9255,0.6196 }

  \tikzstyle{facestyle_p33} = [fill=facetcolor_p33, fill opacity=0.85, draw=edgecolor_p33, line width=1 pt, line cap=round, line join=round]

  % FACES and EDGES and POINTS in the right order
  \draw[facestyle_p33] (v47_p33) -- (v38_p33) -- (v11_p33) -- (v39_p33) -- (v47_p33) -- cycle;
  \draw[facestyle_p33] (v34_p33) -- (v45_p33) -- (v35_p33) -- (v11_p33) -- (v34_p33) -- cycle;
  \draw[facestyle_p33] (v44_p33) -- (v23_p33) -- (v5_p33) -- (v21_p33) -- (v44_p33) -- cycle;
  \draw[facestyle_p33] (v38_p33) -- (v18_p33) -- (v11_p33) -- (v38_p33) -- cycle;
  \draw[facestyle_p33] (v5_p33) -- (v41_p33) -- (v49_p33) -- (v40_p33) -- (v5_p33) -- cycle;
  \draw[facestyle_p33] (v18_p33) -- (v34_p33) -- (v11_p33) -- (v18_p33) -- cycle;
  \draw[facestyle_p33] (v46_p33) -- (v25_p33) -- (v1_p33) -- (v27_p33) -- (v46_p33) -- cycle;
  \draw[facestyle_p33] (v39_p33) -- (v11_p33) -- (v19_p33) -- (v39_p33) -- cycle;
  \draw[facestyle_p33] (v19_p33) -- (v11_p33) -- (v35_p33) -- (v19_p33) -- cycle;

  %POINTS
  \fill[pointcolor_p33] (v11_p33) circle (1 pt);
  \node at (v11_p33) [text=black, inner sep=0.5pt, above right, draw=none, align=left] {};

  %FACETS
  \draw[facestyle_p33] (v40_p33) -- (v49_p33) -- (v42_p33) -- (v1_p33) -- (v40_p33) -- cycle;
  \draw[facestyle_p33] (v21_p33) -- (v5_p33) -- (v13_p33) -- (v21_p33) -- cycle;
  \draw[facestyle_p33] (v23_p33) -- (v15_p33) -- (v5_p33) -- (v23_p33) -- cycle;
  \draw[facestyle_p33] (v13_p33) -- (v5_p33) -- (v40_p33) -- (v13_p33) -- cycle;
  \draw[facestyle_p33] (v5_p33) -- (v15_p33) -- (v41_p33) -- (v5_p33) -- cycle;

  %POINTS
  \fill[pointcolor_p33] (v5_p33) circle (1 pt);
  \node at (v5_p33) [text=black, inner sep=0.5pt, above right, draw=none, align=left] {};

  %FACETS
  \draw[facestyle_p33] (v25_p33) -- (v13_p33) -- (v1_p33) -- (v25_p33) -- cycle;
  \draw[facestyle_p33] (v37_p33) -- (v47_p33) -- (v39_p33) -- (v3_p33) -- (v37_p33) -- cycle;
  \draw[facestyle_p33] (v27_p33) -- (v1_p33) -- (v17_p33) -- (v27_p33) -- cycle;
  \draw[facestyle_p33] (v41_p33) -- (v3_p33) -- (v43_p33) -- (v49_p33) -- (v41_p33) -- cycle;
  \draw[facestyle_p33] (v13_p33) -- (v40_p33) -- (v1_p33) -- (v13_p33) -- cycle;

  %POINTS
  \fill[pointcolor_p33] (v40_p33) circle (1 pt);
  \node at (v40_p33) [text=black, inner sep=0.5pt, above right, draw=none, align=left] {};

  %FACETS
  \draw[facestyle_p33] (v17_p33) -- (v1_p33) -- (v42_p33) -- (v17_p33) -- cycle;

  %POINTS
  \fill[pointcolor_p33] (v1_p33) circle (1 pt);
  \node at (v1_p33) [text=black, inner sep=0.5pt, above right, draw=none, align=left] {};

  %FACETS
  \draw[facestyle_p33] (v33_p33) -- (v7_p33) -- (v35_p33) -- (v45_p33) -- (v33_p33) -- cycle;
  \draw[facestyle_p33] (v49_p33) -- (v43_p33) -- (v7_p33) -- (v42_p33) -- (v49_p33) -- cycle;

  %POINTS
  \fill[pointcolor_p33] (v49_p33) circle (1 pt);
  \node at (v49_p33) [text=black, inner sep=0.5pt, above right, draw=none, align=left] {};

  %FACETS
  \draw[facestyle_p33] (v46_p33) -- (v27_p33) -- (v10_p33) -- (v26_p33) -- (v46_p33) -- cycle;
  \draw[facestyle_p33] (v26_p33) -- (v10_p33) -- (v16_p33) -- (v26_p33) -- cycle;
  \draw[facestyle_p33] (v15_p33) -- (v37_p33) -- (v3_p33) -- (v15_p33) -- cycle;
  \draw[facestyle_p33] (v3_p33) -- (v39_p33) -- (v19_p33) -- (v3_p33) -- cycle;

  %POINTS
  \fill[pointcolor_p33] (v39_p33) circle (1 pt);
  \node at (v39_p33) [text=black, inner sep=0.5pt, above right, draw=none, align=left] {};

  %FACETS
  \draw[facestyle_p33] (v15_p33) -- (v3_p33) -- (v41_p33) -- (v15_p33) -- cycle;

  %POINTS
  \fill[pointcolor_p33] (v41_p33) circle (1 pt);
  \node at (v41_p33) [text=black, inner sep=0.5pt, above right, draw=none, align=left] {};

  %FACETS
  \draw[facestyle_p33] (v3_p33) -- (v19_p33) -- (v43_p33) -- (v3_p33) -- cycle;

  %POINTS
  \fill[pointcolor_p33] (v3_p33) circle (1 pt);
  \node at (v3_p33) [text=black, inner sep=0.5pt, above right, draw=none, align=left] {};

  %FACETS
  \draw[facestyle_p33] (v17_p33) -- (v7_p33) -- (v33_p33) -- (v17_p33) -- cycle;
  \draw[facestyle_p33] (v19_p33) -- (v35_p33) -- (v7_p33) -- (v19_p33) -- cycle;

  %POINTS
  \fill[pointcolor_p33] (v35_p33) circle (1 pt);
  \node at (v35_p33) [text=black, inner sep=0.5pt, above right, draw=none, align=left] {};

  %FACETS
  \draw[facestyle_p33] (v10_p33) -- (v33_p33) -- (v45_p33) -- (v32_p33) -- (v10_p33) -- cycle;
  \draw[facestyle_p33] (v17_p33) -- (v42_p33) -- (v7_p33) -- (v17_p33) -- cycle;

  %POINTS
  \fill[pointcolor_p33] (v42_p33) circle (1 pt);
  \node at (v42_p33) [text=black, inner sep=0.5pt, above right, draw=none, align=left] {};

  %FACETS
  \draw[facestyle_p33] (v43_p33) -- (v19_p33) -- (v7_p33) -- (v43_p33) -- cycle;

  %POINTS
  \fill[pointcolor_p33] (v43_p33) circle (1 pt);
  \node at (v43_p33) [text=black, inner sep=0.5pt, above right, draw=none, align=left] {};
  \fill[pointcolor_p33] (v19_p33) circle (1 pt);
  \node at (v19_p33) [text=black, inner sep=0.5pt, above right, draw=none, align=left] {};
  \fill[pointcolor_p33] (v7_p33) circle (1 pt);
  \node at (v7_p33) [text=black, inner sep=0.5pt, above right, draw=none, align=left] {};

  %FACETS
  \draw[facestyle_p33] (v27_p33) -- (v17_p33) -- (v10_p33) -- (v27_p33) -- cycle;

  %POINTS
  \fill[pointcolor_p33] (v27_p33) circle (1 pt);
  \node at (v27_p33) [text=black, inner sep=0.5pt, above right, draw=none, align=left] {};

  %FACETS
  \draw[facestyle_p33] (v16_p33) -- (v10_p33) -- (v32_p33) -- (v16_p33) -- cycle;
  \draw[facestyle_p33] (v10_p33) -- (v17_p33) -- (v33_p33) -- (v10_p33) -- cycle;

  %POINTS
  \fill[pointcolor_p33] (v10_p33) circle (1 pt);
  \node at (v10_p33) [text=black, inner sep=0.5pt, above right, draw=none, align=left] {};
  \fill[pointcolor_p33] (v17_p33) circle (1 pt);
  \node at (v17_p33) [text=black, inner sep=0.5pt, above right, draw=none, align=left] {};
  \fill[pointcolor_p33] (v33_p33) circle (1 pt);
  \node at (v33_p33) [text=black, inner sep=0.5pt, above right, draw=none, align=left] {};

  %FACETS
  \draw[facestyle_p33] (v24_p33) -- (v8_p33) -- (v25_p33) -- (v46_p33) -- (v24_p33) -- cycle;
  \draw[facestyle_p33] (v8_p33) -- (v20_p33) -- (v44_p33) -- (v21_p33) -- (v8_p33) -- cycle;
  \draw[facestyle_p33] (v6_p33) -- (v32_p33) -- (v45_p33) -- (v34_p33) -- (v6_p33) -- cycle;

  %POINTS
  \fill[pointcolor_p33] (v45_p33) circle (1 pt);
  \node at (v45_p33) [text=black, inner sep=0.5pt, above right, draw=none, align=left] {};

  %FACETS
  \draw[facestyle_p33] (v8_p33) -- (v13_p33) -- (v25_p33) -- (v8_p33) -- cycle;

  %POINTS
  \fill[pointcolor_p33] (v25_p33) circle (1 pt);
  \node at (v25_p33) [text=black, inner sep=0.5pt, above right, draw=none, align=left] {};

  %FACETS
  \draw[facestyle_p33] (v48_p33) -- (v30_p33) -- (v6_p33) -- (v31_p33) -- (v48_p33) -- cycle;
  \draw[facestyle_p33] (v8_p33) -- (v21_p33) -- (v13_p33) -- (v8_p33) -- cycle;

  %POINTS
  \fill[pointcolor_p33] (v21_p33) circle (1 pt);
  \node at (v21_p33) [text=black, inner sep=0.5pt, above right, draw=none, align=left] {};
  \fill[pointcolor_p33] (v13_p33) circle (1 pt);
  \node at (v13_p33) [text=black, inner sep=0.5pt, above right, draw=none, align=left] {};

  %FACETS
  \draw[facestyle_p33] (v36_p33) -- (v2_p33) -- (v38_p33) -- (v47_p33) -- (v36_p33) -- cycle;
  \draw[facestyle_p33] (v12_p33) -- (v8_p33) -- (v24_p33) -- (v12_p33) -- cycle;
  \draw[facestyle_p33] (v12_p33) -- (v20_p33) -- (v8_p33) -- (v12_p33) -- cycle;

  %POINTS
  \fill[pointcolor_p33] (v8_p33) circle (1 pt);
  \node at (v8_p33) [text=black, inner sep=0.5pt, above right, draw=none, align=left] {};

  %FACETS
  \draw[facestyle_p33] (v29_p33) -- (v48_p33) -- (v31_p33) -- (v2_p33) -- (v29_p33) -- cycle;
  \draw[facestyle_p33] (v18_p33) -- (v6_p33) -- (v34_p33) -- (v18_p33) -- cycle;

  %POINTS
  \fill[pointcolor_p33] (v34_p33) circle (1 pt);
  \node at (v34_p33) [text=black, inner sep=0.5pt, above right, draw=none, align=left] {};

  %FACETS
  \draw[facestyle_p33] (v6_p33) -- (v16_p33) -- (v32_p33) -- (v6_p33) -- cycle;

  %POINTS
  \fill[pointcolor_p33] (v32_p33) circle (1 pt);
  \node at (v32_p33) [text=black, inner sep=0.5pt, above right, draw=none, align=left] {};

  %FACETS
  \draw[facestyle_p33] (v31_p33) -- (v6_p33) -- (v18_p33) -- (v31_p33) -- cycle;
  \draw[facestyle_p33] (v30_p33) -- (v16_p33) -- (v6_p33) -- (v30_p33) -- cycle;

  %POINTS
  \fill[pointcolor_p33] (v6_p33) circle (1 pt);
  \node at (v6_p33) [text=black, inner sep=0.5pt, above right, draw=none, align=left] {};

  %FACETS
  \draw[facestyle_p33] (v2_p33) -- (v18_p33) -- (v38_p33) -- (v2_p33) -- cycle;

  %POINTS
  \fill[pointcolor_p33] (v38_p33) circle (1 pt);
  \node at (v38_p33) [text=black, inner sep=0.5pt, above right, draw=none, align=left] {};

  %FACETS
  \draw[facestyle_p33] (v24_p33) -- (v46_p33) -- (v26_p33) -- (v0_p33) -- (v24_p33) -- cycle;

  %POINTS
  \fill[pointcolor_p33] (v46_p33) circle (1 pt);
  \node at (v46_p33) [text=black, inner sep=0.5pt, above right, draw=none, align=left] {};

  %FACETS
  \draw[facestyle_p33] (v14_p33) -- (v2_p33) -- (v36_p33) -- (v14_p33) -- cycle;
  \draw[facestyle_p33] (v28_p33) -- (v0_p33) -- (v30_p33) -- (v48_p33) -- (v28_p33) -- cycle;
  \draw[facestyle_p33] (v2_p33) -- (v31_p33) -- (v18_p33) -- (v2_p33) -- cycle;

  %POINTS
  \fill[pointcolor_p33] (v31_p33) circle (1 pt);
  \node at (v31_p33) [text=black, inner sep=0.5pt, above right, draw=none, align=left] {};
  \fill[pointcolor_p33] (v18_p33) circle (1 pt);
  \node at (v18_p33) [text=black, inner sep=0.5pt, above right, draw=none, align=left] {};

  %FACETS
  \draw[facestyle_p33] (v14_p33) -- (v29_p33) -- (v2_p33) -- (v14_p33) -- cycle;

  %POINTS
  \fill[pointcolor_p33] (v2_p33) circle (1 pt);
  \node at (v2_p33) [text=black, inner sep=0.5pt, above right, draw=none, align=left] {};

  %FACETS
  \draw[facestyle_p33] (v44_p33) -- (v20_p33) -- (v4_p33) -- (v22_p33) -- (v44_p33) -- cycle;
  \draw[facestyle_p33] (v4_p33) -- (v28_p33) -- (v48_p33) -- (v29_p33) -- (v4_p33) -- cycle;

  %POINTS
  \fill[pointcolor_p33] (v48_p33) circle (1 pt);
  \node at (v48_p33) [text=black, inner sep=0.5pt, above right, draw=none, align=left] {};

  %FACETS
  \draw[facestyle_p33] (v9_p33) -- (v36_p33) -- (v47_p33) -- (v37_p33) -- (v9_p33) -- cycle;

  %POINTS
  \fill[pointcolor_p33] (v47_p33) circle (1 pt);
  \node at (v47_p33) [text=black, inner sep=0.5pt, above right, draw=none, align=left] {};

  %FACETS
  \draw[facestyle_p33] (v15_p33) -- (v9_p33) -- (v37_p33) -- (v15_p33) -- cycle;

  %POINTS
  \fill[pointcolor_p33] (v37_p33) circle (1 pt);
  \node at (v37_p33) [text=black, inner sep=0.5pt, above right, draw=none, align=left] {};

  %FACETS
  \draw[facestyle_p33] (v0_p33) -- (v26_p33) -- (v16_p33) -- (v0_p33) -- cycle;

  %POINTS
  \fill[pointcolor_p33] (v26_p33) circle (1 pt);
  \node at (v26_p33) [text=black, inner sep=0.5pt, above right, draw=none, align=left] {};

  %FACETS
  \draw[facestyle_p33] (v12_p33) -- (v24_p33) -- (v0_p33) -- (v12_p33) -- cycle;

  %POINTS
  \fill[pointcolor_p33] (v24_p33) circle (1 pt);
  \node at (v24_p33) [text=black, inner sep=0.5pt, above right, draw=none, align=left] {};

  %FACETS
  \draw[facestyle_p33] (v30_p33) -- (v0_p33) -- (v16_p33) -- (v30_p33) -- cycle;

  %POINTS
  \fill[pointcolor_p33] (v30_p33) circle (1 pt);
  \node at (v30_p33) [text=black, inner sep=0.5pt, above right, draw=none, align=left] {};
  \fill[pointcolor_p33] (v16_p33) circle (1 pt);
  \node at (v16_p33) [text=black, inner sep=0.5pt, above right, draw=none, align=left] {};

  %FACETS
  \draw[facestyle_p33] (v28_p33) -- (v12_p33) -- (v0_p33) -- (v28_p33) -- cycle;

  %POINTS
  \fill[pointcolor_p33] (v0_p33) circle (1 pt);
  \node at (v0_p33) [text=black, inner sep=0.5pt, above right, draw=none, align=left] {};

  %FACETS
  \draw[facestyle_p33] (v22_p33) -- (v4_p33) -- (v14_p33) -- (v22_p33) -- cycle;
  \draw[facestyle_p33] (v4_p33) -- (v20_p33) -- (v12_p33) -- (v4_p33) -- cycle;

  %POINTS
  \fill[pointcolor_p33] (v20_p33) circle (1 pt);
  \node at (v20_p33) [text=black, inner sep=0.5pt, above right, draw=none, align=left] {};

  %FACETS
  \draw[facestyle_p33] (v44_p33) -- (v22_p33) -- (v9_p33) -- (v23_p33) -- (v44_p33) -- cycle;

  %POINTS
  \fill[pointcolor_p33] (v44_p33) circle (1 pt);
  \node at (v44_p33) [text=black, inner sep=0.5pt, above right, draw=none, align=left] {};

  %FACETS
  \draw[facestyle_p33] (v14_p33) -- (v4_p33) -- (v29_p33) -- (v14_p33) -- cycle;

  %POINTS
  \fill[pointcolor_p33] (v29_p33) circle (1 pt);
  \node at (v29_p33) [text=black, inner sep=0.5pt, above right, draw=none, align=left] {};

  %FACETS
  \draw[facestyle_p33] (v4_p33) -- (v12_p33) -- (v28_p33) -- (v4_p33) -- cycle;

  %POINTS
  \fill[pointcolor_p33] (v4_p33) circle (1 pt);
  \node at (v4_p33) [text=black, inner sep=0.5pt, above right, draw=none, align=left] {};
  \fill[pointcolor_p33] (v12_p33) circle (1 pt);
  \node at (v12_p33) [text=black, inner sep=0.5pt, above right, draw=none, align=left] {};
  \fill[pointcolor_p33] (v28_p33) circle (1 pt);
  \node at (v28_p33) [text=black, inner sep=0.5pt, above right, draw=none, align=left] {};

  %FACETS
  \draw[facestyle_p33] (v9_p33) -- (v14_p33) -- (v36_p33) -- (v9_p33) -- cycle;

  %POINTS
  \fill[pointcolor_p33] (v36_p33) circle (1 pt);
  \node at (v36_p33) [text=black, inner sep=0.5pt, above right, draw=none, align=left] {};

  %FACETS
  \draw[facestyle_p33] (v23_p33) -- (v9_p33) -- (v15_p33) -- (v23_p33) -- cycle;

  %POINTS
  \fill[pointcolor_p33] (v23_p33) circle (1 pt);
  \node at (v23_p33) [text=black, inner sep=0.5pt, above right, draw=none, align=left] {};
  \fill[pointcolor_p33] (v15_p33) circle (1 pt);
  \node at (v15_p33) [text=black, inner sep=0.5pt, above right, draw=none, align=left] {};

  %FACETS
  \draw[facestyle_p33] (v22_p33) -- (v14_p33) -- (v9_p33) -- (v22_p33) -- cycle;

  %POINTS
  \fill[pointcolor_p33] (v22_p33) circle (1 pt);
  \node at (v22_p33) [text=black, inner sep=0.5pt, above right, draw=none, align=left] {};
  \fill[pointcolor_p33] (v14_p33) circle (1 pt);
  \node at (v14_p33) [text=black, inner sep=0.5pt, above right, draw=none, align=left] {};
  \fill[pointcolor_p33] (v9_p33) circle (1 pt);
  \node at (v9_p33) [text=black, inner sep=0.5pt, above right, draw=none, align=left] {};

  %FACETS

\end{tikzpicture}

% Local Variables: 
% mode: latex
% mode: TeX-PDF
% TeX-master: "bisectors"
% mode: reftex
% mode: font-lock
% buffer-file-coding-system:utf-8-unix
% End: 

  \caption{The fans $\Fan(\ball{3})$, $\Fan(A_3)$, and the bisection fan $\bop^3$.}
  \label{fig:three_fans}
\end{figure}
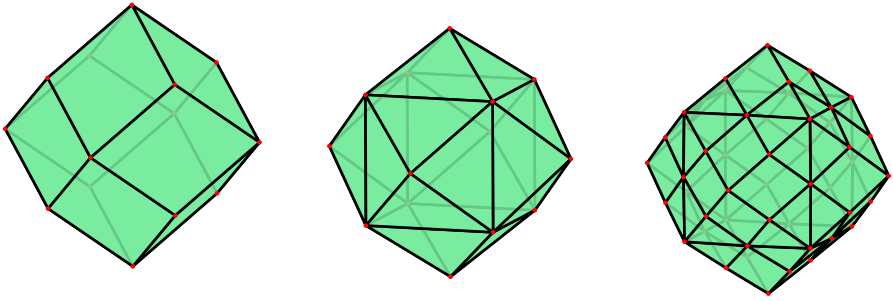

The \emph{bisected ordered partition} of $[d+1]$ induced by $v\in \R^{d+1}$ is the ordered partition $S(v)$ as defined above, together with the information of which (values of) parts are smaller, equal or greater than the midvalue $\midv(v)$; see also \cite{Develin:2005}.
Equivalently, this is the ordered partition associated with the \emph{extended vector} $(v,\midv(v)) \in \R^{d+2}$.
We denote by $\bop^d$ the fan of bisected ordered partitions of dimension $d$.

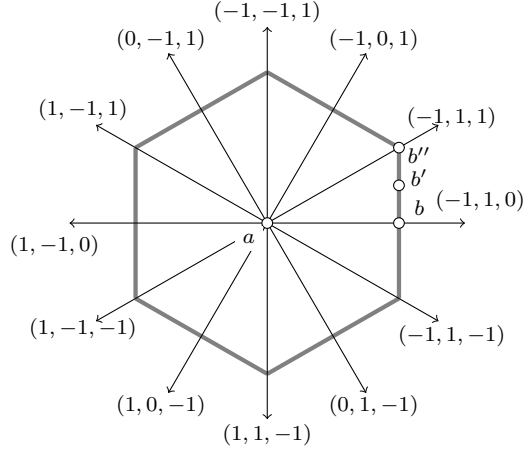
\begin{figure}[tb]
  %\sidecaption
  \begin{tikzpicture}[scale=2]\scriptsize
    \foreach \x in {0,60,...,300} {
      \draw [gray, ultra thick] (\x-30:1 cm) -- (\x + 30:1 cm);
    }
    \foreach \x in {30,60,...,360} {
      \node (n\x) at (\x+60:1.4cm) {};
      \draw[->] (0,0) -- (\x:1.3 cm);
    }
    \node at (n30) {$(-1,-1,1)$};
    \node at (n60) {$(0,-1,1)$};
    \node [yshift=0.1cm] at (n90) {$(1,-1,1)$};
    \node [yshift=-0.3cm] at (n120) {$(1,-1,0)$};
    \node at (n150) {$(1,-1,-1)$};
    \node at (n180) {$(1,0,-1)$};
    \node at (n210) {$(1,1,-1)$};
    \node at (n240) {$(0,1,-1)$};
    \node [yshift=-0.1cm] at (n270) {$(-1,1,-1)$};
    \node[yshift=0.3cm] at (n300) {$(-1,1,0)$};
    \node at (n330) {$(-1,1,1)$};
    \node at (n360) {$(-1,0,1)$};
    \draw [fill=white,draw=black] (0.866,0.) circle (1.0pt);
    \draw[color=black] (1.,0.1) node {$b$};
    \draw [fill=white,draw=black] (0.866,0.25) circle (1.0pt);
    \draw[color=black] (1.,0.30) node {$b'$};
    \draw [fill=white,draw=black] (0.866,0.5) circle (1.0pt);
    \draw[color=black] (1.,0.45) node {$b''$};
    \draw [fill=white,draw=black] (0,0) circle (1.0pt);
    \draw[color=black] (-0.12,-0.1) node[fill=white] {$a$};
  \end{tikzpicture}
  \caption{The bisection fan, $\ccF$, for $d=2$. The three vectors $b-a$ for Figure~\ref{tikz:bisectors} have been marked.}
  \label{tikz:fan}
\end{figure}

\begin{remark}\label{rem:fans}
  The \enquote{finest} or \enquote{most generic} ordered partitions are the permutations, in which each part is a singleton.
  Hence, the fan of ordered partitions equals the normal fan of the \emph{permutahedron}.
  This, in turn, coincides with the fan of regions in the \emph{braid arrangement} or Coxeter arrangement of type $A_d$, which consists of the hyperplanes $\SetOf{x}{x_i = x_j}$ for $1\leq i<j\leq d+1$.
  We denote this fan $\Fan(A_d)$.
  It is intermediate between the central fan of the tropical ball (which is coarser) and the fan of bisected ordered partitions (which is finer):
  \[
    \Fan(\ball{d}) \ \leq \ \Fan(A_{d}) \ \leq \ \bop^{d}  \enspace;
  \]
  see Figure~\ref{fig:three_fans} for a visualization of the case $d=3$.
  Note that $\Fan(A_{d})$ is also the fan of weak general position: $a$ and $b$ are in weak general position if and only if $b-a$ lies in a full-dimensional cell.
\end{remark}

\begin{example}\label{ex:2dim}
  In $d=2$, the fans $\Fan(A_{d})$ and $\Fan(\ball{d})$ coincide, they form the face fan of the regular hexagon $\ball{2}$.
  The bisection fan  $\bop^{d}$ is the barycentric subdivision of it; see Figure~\ref{tikz:fan}.
  Excluding permutations of the coordinates, and sign inversion, we infer that there are three types of tropical bisectors in the plane, and these are shown in Figure \ref{tikz:bisectors}.
  The type to the left is in weak general position but not in general position, the type to the left is in general position, and the type to the right is not even in weak general position.
\end{example}

Recall that the $\max$-tropical \emph{line segment} spanned by two points, $a$ and $b$, is the set
\[
  \tropseg{a}{b} \ := \ \SetOf{\max(\alpha\ones + a,\beta\ones + b)}{\alpha,\beta\in\R} \ \subset \ \torus{d+1}\enspace .
\]
It is worth noting that the combinatorial types of tropical segments are classified by the braid fan $\Fan(A_d)$; see \cite[Prop.~5.11]{Tropical+Book}.
In this sense the fan of bisected ordered partitions $\bop^{d}$ classifies the combinatorial types of \enquote{bisected tropical segments}:

\begin{proposition}
  The bisected ordered partition of $b-a$ contains the same information as the combinatorial type of the tropical line segment $[a,b]$ together with the information of which part contains the midpoint.
\end{proposition}

\begin{proof}
  Suppose for simplicity that $a,b\in\R^{d+1}$ satisfy
  \begin{equation}\label{eq:sorting}
    b_1-a_1 \ \leq \ b_2-a_2 \ \leq \ \dots \ \leq \ b_{d+1}-a_{d+1} \enspace .
  \end{equation}
  Then $\tropseg{a}{b}$ is the union of at most $d$ ordinary line segments, one for each subset of coordinates between a strict inequality in \eqref{eq:sorting}.
  That is, the combinatorics of the tropical segment is the same as the ordered partition of $b-a$.
  The midvalue $\midv(b-a)$ selects one of the ordinary segments.
%\qed
\end{proof}

The goal of the rest of this section is to prove the following:

\begin{theorem}
  \label{thm:ccfan}
  We have $\ccF= \bop^d$.
  That is, given $a,b,a',b' \in \torus{d+1}$ the polyhedra $\bisector(a,b)$ and $\bisector(a',b')$ are normally equivalent if and only if $b-a$ and $b'-a'$ induce the same bisected ordered partition of $[d+1]$.
\end{theorem}

\subsection{Proof of Theorem~\ref{thm:ccfan}}

In \eqref{eq:type} we defined the type $(F_-,F_*,F_+)$ of a face $F$ of $\ball{d}$ or the corresponding face cone in $\Fan(\ball{d})$.
For a pair of faces, $F$ and $G$, this gives rise to the following \emph{labeling partition} of $[d+1]$:
\begin{equation}\label{eq:labels}
  \begin{aligned}
    L_0 \ &:= \ (F_-\cap G_-)\cup (F_+\cap G_+)  \enspace,\\
    L_+ \ &:= \ \bigl(F_+ \cap G_*\bigr) \cup \bigl(F_*\cap G_-\bigr)  \enspace,\\
    L_- \ &:= \ \bigl(F_-\cap G_* \bigr) \cup \bigl(F_* \cap G_+ \bigr)  \enspace,\\
    L_{+1} \ &:= \ F_+\cap G_-  \enspace,\\
    L_{-1} \ &:= \ F_-\cap G_+  \enspace,\\
    L_{*} \ &:= \  F_* \cap G_* \enspace.
  \end{aligned}
\end{equation}
As a first step in the proof of Theorem~\ref{thm:ccfan}, the following lemma characterizes when $\bisector_{(F,G)}(a,b)$ is non-empty.
Recall that this is the case if and only if there is a tropical ball touching $a$ and $b$ at faces $F$ and $G$, respectively.

\begin{lemma}\label{lem:F+G+faces}
  Let $F$ and $G$ be a fixed pair of faces of $\ball{d}$ with the labeling partition defined as in \eqref{eq:labels}.
  Further let $a,b\in \torus{d+1}$.
  Then the set $\bisector_{(F,G)}(a,b)$ is not empty if and only if there exist $\gamma\in \R$ and $\delta \in [0,\infty)$ such that the following conditions are satisfied:
  \begin{equation}\label{eq:system} \left \{ \begin{array}{ll}
        (b-a)_i = \gamma & \text{ if  } i\in L_0 \enspace,\\
        (b-a)_i \in [\gamma,\gamma+\delta] & \text{ if  } i\in L_+ \enspace,\\
        (b-a)_i \in [\gamma-\delta,\gamma] & \text{ if  } i\in L_- \enspace,\\
        (b-a)_i =  \gamma-\delta & \text{ if  } i\in L_{-1} \enspace,\\
        (b-a)_i = \gamma+\delta & \text{ if  } i\in L_{+1} \enspace,\\
        (b-a)_i \in [\gamma-\delta,\gamma+\delta] & \text{ if  } i\in L_* \enspace.
  \end{array}\right. \end{equation}
\end{lemma}

\begin{proof}
  Let us assume that the face of the tropical bisector $\bisector(a,b)$ defined by $(F,G)$ is non-empty.
  Then there is a point $x$ such that $\dist(a,x)=\dist(x,b)=\delta$ and $a-x\in F$ as well as $b-x\in G$.
  We set $\gamma_a=\min_{i\in[d+1]}(a_i-x_i)$ and $\gamma_b=\min_{i\in[d+1]}(b_i-x_i)$.
  The possible values for the coordinates of $a-x$ and $b-x$ are
  \begin{align}
   \label{eq:system_a}
    \left\{ \begin{array}{ll} (a-x)_i= \gamma_a+\delta & \text{if  } i\in F_- \enspace, \\ (a-x)_i=\gamma_a & \text{if  } i \in F_+ \enspace, \\ (a-x)_i \in [\gamma_a,\gamma_a+\delta] & \text {if  } i\in[d+1]\setminus (F_-\cup F_+) \end{array}\right.
    \intertext{and}
    \label{eq:system_b}
    \left\{ \begin{array}{ll} (b-x)_i=\gamma_b+\delta & \text{if  } i\in G_- \enspace, \\ (b-x)_i=\gamma_b & \text{if  } i \in G_+ \enspace, \\ (b-x)_i \in [\gamma_b,\gamma_b+\delta] & \text {if  } i\in[d+1]\setminus (G_-\cup G_+) \enspace, \end{array}\right.
  \end{align}
  for some $\delta\geq 0$.
  Setting $\gamma=\gamma_b-\gamma_a$ the above translates into \eqref{eq:system}.

  For the converse, note that going from \eqref{eq:system_a} and \eqref{eq:system_b} to  \eqref{eq:system} is the Fourier-Motzkin elimination of the variables $x_i$.
  Therefore, any $\gamma$ and $\delta\geq0$ which are feasible for \eqref{eq:system} can be lifted to a solution of \eqref{eq:system_a} and \eqref{eq:system_b}.
  That is to say, we can set $\gamma_a=0$ and $\gamma_b=\gamma$, and the conditions in \eqref{eq:system_a} and \eqref{eq:system_b} yield a point $x\in\bisector_{(F,G)}(a,b)$.
%\qed
\end{proof}

\begin{proposition}
  \label{prop:F+G+faces}
  Let $F$ and $G$ be a fixed pair of faces of $\ball{d}$.
  Then the set
  \begin{equation}\label{eq:F+G+faces}
    C \ := \ \SetOf{b-a}{a,b\in\torus{d+1} \text{ with } \bisector_{(F,G)}(a,b) \neq \emptyset}
  \end{equation}
  is both a polyhedral cone and a tropical cone, although perhaps not a tropical polyhedral cone.
\end{proposition}

\begin{proof}
  Let $a,a',b,b'\in\torus{d+1}$ with $\bisector_{(F,G)}(a,b)$ and $\bisector_{(F,G)}(a',b')$ both non-empty.
  Since $\bisector_{(F,G)}(a,b)\neq\emptyset$, by Lemma~\ref{lem:F+G+faces}, there are scalars $\gamma$ and $\delta$ satisfying the conditions \eqref{eq:system}.
  Likewise there are certificates $\gamma'$ and $\delta'$ for $\bisector_{(F,G)}(a',b')\neq\emptyset$.
  By linearity of the conditions \eqref{eq:system} it follows that $\gamma+\gamma'$ and $\delta+\delta'$ certify that $\bisector_{(F,G)}(a+a',b+b')\neq\emptyset$:
  for instance, we have $(b+b'-a-a')_i = \gamma + \gamma'$ for $i \in L_0 = (F_-\cap G_-)\cup (F_+\cap G_+)$.
  Since clearly $\alpha c \in C$ for all $c\in C$ and $\alpha\geq 0$ we conclude that $C$ is an ordinary cone.
  This cone is polyhedral as it is defined in terms of the finitely many linear conditions \eqref{eq:system}.

  A similar argument shows that $C$ is also closed with respect to taking arbitrary $(\max,+)$-linear combinations:
  for instance, with the above notation we have $\max((b-a)_i,(b'-a')_i) = \max(\gamma,\gamma')$ for $i \in L_0$.
  This shows that $C$ is a tropical cone.
%\qed
\end{proof}

\begin{corollary}
  The bisection fan of tropical bisectors is a classical polyhedral fan, and a tropical (perhaps not tropical polyhedral) fan.
\end{corollary}

\begin{proof}
  We know that the feasibility region of a face $(F,G)$ is a tropical and classical polyhedral cone. Finite intersections of these cones are again tropical cones, classical cones, and polyhedral cones. Therefore, the feasibility region of a normal equivalence class, which is the intersection of the cones of its non-empty faces, is again a tropical and classical polyhedral cone.
  Hence, the whole fan has this structure.
%\qed
\end{proof}

The following shows one direction of Theorem~\ref{thm:ccfan}, namely, $\ccF$ is coarser than $\bop^{d}$.

\begin{lemma}%\label{lem:midpoint}
\label{lem:ccfan-part1}
   Let $F,G \in \Fan(\ball{d})$. Then, whether $\bisector_{(F,G)}(a,b)$ is empty or not, for each $a,b \in \torus{d+1}$ depends only on the bisected ordered partition of $b-a$.
\end{lemma}

\begin{proof}
  Consider the partition of $[d+1]$ into six sets $L_0,L_+,L_-,L_{+1},L_{-1},L_*$ defined in \eqref{eq:labels}.
  We want to show that feasibility of the system \eqref{eq:system} for a given $a$ and $b$ depends only on the bisected ordered partition of $b-a$.
  Without loss of generality we assume $a=0$ and $b_1 \leq b_2 \leq \dots \leq b_{d+1}$ as in \eqref{eq:sorting}.

  Let $\midv(b):= \tfrac{1}{2}(b_{d+1} + b_1)= \tfrac{1}{2}\dist(0,b) + b_1$ be the midvalue of $b-0$  and $m=\max(0+\midv(b)\ones,b)$ the midpoint of the segment $[0, b]$. In particular, $\dist(0,m)=\dist(m,b)=\tfrac{1}{2}\dist(0,b)$.

  We distinguish three cases, depending on whether both, none, or exactly one of $L_{+1}$ and $L_{-1}$ are empty.

\textbf{Claim I:} \emph{Suppose $L_{+1}\cup L_{-1}=\emptyset$. Then, \eqref{eq:system} is feasible if and only if there are $k\leq\ell$ with
  \begin{equation}
    \label{eq:claim1}
    \tag{I.1}
    L_-\subseteq\{1,\dots,k\} \,,\ L_0\subseteq\{k+1,\dots,\ell\} \,,\ L_+\subseteq\{\ell+1,\dots,d+1\} \enspace.
    % i_1\in L_-,\  i_2\in L_0,\  i_3\in L_+, \qquad\Rightarrow\qquad  i_1 < i_2 < i_3.
  \end{equation}}
  Indeed, in this case feasibility of \eqref{eq:system} is equivalent to feasibility of
  \begin{equation}\label{eq:midpoint:unbounded}
    \begin{cases}
      \gamma\geq b_i & \text{for } i\in L_-\enspace, \\
      \gamma = b_i & \text{for } i\in L_0\enspace,\\
      \gamma \leq b_i  &\text{for } i\in L_+\enspace,
    \end{cases}
  \end{equation}
  which implies the ordered partition to satisfy \eqref{eq:claim1}. Conversely, if the ordered partition satisfies  \eqref{eq:claim1}, then let $\gamma$ be chosen to satisfy \eqref{eq:midpoint:unbounded} and let
%  In particular, this is only possible if the order of $b$ respects the order given by the labels.
%  That is, there exist indices $k\leq\ell$ such $L_-\subseteq\{1,\dots,k\}$ and $L_0\subseteq\{k+1,\dots,\ell\}$ as well as $L_+\subseteq\{\ell+1,\dots,d+1\}$.
%  Taking
  $\delta \geq \min(\gamma-b_1,b_{d+1}-\gamma)$. This yields a feasible solution to \eqref{eq:system}.
  In other words, in this case we can tell if $C$ is empty or not by just looking at the the ordered partition of  $b$; the relative position of the midpoint is irrelevant.

  \textbf{Claim II:} \emph{Suppose $L_{+1}\neq \emptyset = L_{-1}$. Then, \eqref{eq:system} is feasible if and only if, in addition to \eqref{eq:claim1}, we have
    \begin{align}
        \label{eq:claim2.1} \tag{II.1}
        \SetOf{b_i}{i\in L_{+1}} =\{b_{d+1}\}\enspace,\\
        \label{eq:claim2.2} \tag{II.2}
        b_i \leq \midv(b) \text{ for } i\in L_{0}\cup L_- \enspace, \\
        \label{eq:claim2.3}\tag{II.3}
        | \SetOf{b_i}{i\in L_{0}}| \leq 1\enspace.
    \end{align}}
  If \eqref{eq:claim1}, \eqref{eq:claim2.1}, \eqref{eq:claim2.2}, and \eqref{eq:claim2.3} hold, take $\gamma= \max\SetOf{b_i}{i \in L_0\cup L_-\cup \{1\} }$ and $\delta= b_{d+1}-\gamma$.

  Conversely, if $(\gamma,\delta)$ is feasible for \eqref{eq:system} then $\gamma+\delta=b_i=b_{d+1}$ for all $i\in L_{+1}$.
  In particular, $\bisector_{(F,G)}(a,b)$ is empty unless $\smallSetOf{b_i}{i\in L_{+1}}=\{b_{d+1}\}$ is a singleton.
  Since the coefficients of $b$ are in ascending order, it follows that
  \begin{equation}\label{eq:midpoint:right}
    \gamma+\delta \ = \ b_{d+1}
  \end{equation}
  and $b_i=b_{d+1}$ for all $i \in L_{+1}$. This shows that \eqref{eq:claim2.1} holds.
  Now the constraints of \eqref{eq:system} translate into \eqref{eq:midpoint:unbounded} as in the previous case, which implies \eqref{eq:claim1}. Additionally
  \begin{equation}\label{eq:midpoint:left}
    \gamma-\delta \ \leq \ b_1 \enspace .
  \end{equation}
  Adding \eqref{eq:midpoint:right} and \eqref{eq:midpoint:left} now yields
  \[
    \gamma \ \leq \ \tfrac{1}{2} (b_1+b_{d+1}) \ = \ \midv(b) \enspace ,
  \]
  which, by \eqref{eq:midpoint:unbounded}, gives \eqref{eq:claim2.2}.
  Finally, \eqref{eq:claim2.3} follows from the fact that the only possible value in $\SetOf{b_i}{i\in L_{0}}$ is $\gamma$.

  The case where $L_{-1}\neq \emptyset$ and $L_{+1}=\emptyset$ is analogous.

  \textbf{Claim III:} \emph{Suppose $L_{+1}\neq \emptyset \neq L_{-1}$. Then, \eqref{eq:system} is feasible if and only if
  \begin{align}
    \label{eq:claim3.1} \tag{III.1}
&\SetOf{b_i}{i\in L_{-1}} =\{b_{1}\} \,,\
%  \label{eq:claim3.2}
\SetOf{b_i}{i\in L_{+1}} =\{b_{d+1}\}\enspace,\\
    \label{eq:claim3.2} \tag{III.2}
&\  b_i \leq \mu(b) \text{ for } i\in  L_- \,,\
 b_i = \mu(b) \text{ for } i\in L_{0} \,,\
 b_i \geq \mu(b) \text{ for } i\in  L_+\enspace.
  \end{align}
  }
  Indeed, in this case the only candidate solution for \eqref{eq:system} is $\gamma= \mu(b) =(b_{d+1}+b_1)/2$ and $\delta=(b_{d+1}-b_1)/2$.
  This is a solution or not depending only on whether Equations~\eqref{eq:claim3.1} and~\eqref{eq:claim3.2} are satisfied.
%\qed
\end{proof}

The following gives the other direction of Theorem~\ref{thm:ccfan}: $\ccF$ refines $\bop^{d}$.

\begin{lemma}
\label{lem:ccfan-part2}
Let $a,b,a',b'\in\torus{d+1}$ be two pairs of points.
If the bisected ordered partitions of $b-a$ and $b'-a'$ are not the same, then there is a pair of faces $F,G\in\Fan(\balld)$ such that exactly one of $\bisector_{F,G}(a,b)$ or $\bisector_{F,G}(a',b')$ is empty.
\end{lemma}

\begin{proof}
  As before, we assume without loss of generality that $a=a'=0$.
  We know that the bisected ordered partitions of $b$ and $b'$ are different.
  Our goal is to find  a pair of faces $(F,G)$ that lies in one and only one of the bisectors.
  We do this in three cases, depending on the difference between the bisected ordered partitions

  \textbf{Case I:} \emph{Suppose that
  \begin{align*}
  \SetOf{i\in [d+1]}{b_i = \max_j b_j} \ &\neq \ \SetOf{i\in [d+1]}{b'_i = \max_j b'_j}
  \text{ or } \\
   \SetOf{i\in [d+1]}{b_i = \min_j b_j} \ &\neq \  \SetOf{i\in [d+1]}{b'_i = \min_j b'_j} \enspace.
  \end{align*}
  }
Without loss of generality, there is an index $i$ such that $b_i$ is maximum and $b'_i$ is not,
or $b_i$ is minimum and $b'_i$ is not.
Let $F$ be the face with
\[
F_+ \ = \ \SetOf{i\in [d+1]}{b_i = \max_j b_j} \ , \qquad
F_- \ = \ \SetOf{i\in [d+1]}{b_i = \min_j b_j} \ ,
\]
and let $G=-F$. This choice makes
\[
L_{+1}= F_+\enspace, \qquad
L_{-1}=F_-\enspace, \qquad
L_{*}=F_*\enspace, \qquad
L_{-}=L_{+}=L_0=\emptyset \enspace.
\]
The cell $\bisector_{F,G}(a,b)$ is not empty by lemma~\ref{lem:F+G+faces}, since the following is a solution for  \eqref{eq:system}:
  \[
    \begin{aligned}
      \gamma \ & = \ \frac{1}{2}(\max_{i\in [d+1]} b_i + \min_{i\in [d+1]}b_i) \enspace, \\
      \delta \ & = \ \frac{1}{2}(\max_{i\in [d+1]} a_i - \min_{i\in [d+1]} b_i) \ = \ \frac{1}{2}\dist(a,b) \enspace.
    \end{aligned}
  \]
 However, $\bisector_{(F,G)}(a',b')$ is empty: in order for it not to be empty we would need
  \[
  \SetOf{i\in [d+1]}{b'_i = \min_j b'_j} \ \subset \ F_- \ ,
  \qquad
  \SetOf{i\in [d+1]}{b'_i = \max_j b'_j} \ \subset \ F_+ \ .
  \]

    \textbf{Case II:} \emph{Suppose that
  $b$ and $b'$ have exactly the same maxima and minima but the ordered partitions of $b$ and $b'$ do not coincide.} That is,
 there is a pair of indices, $i,j\in[d+1]\setminus\{1,d+1\}$ such that $b_i\geq b_j$ but $b'_i< b'_j$.

 We may assume that $1$ and $d+1$ are a minimum and a maximum, respectively, of both $b$ and $b'$.
 Let $F$ be the face with $F_+=\{d+1\}$ and $F_-=\{1\}$. Let $G$ be the face with
 $G_+=\{i\}$, and $G_-=\{j\}$.
 Then, \eqref{eq:labels} gives
 \[
   \begin{array}{ll}
     L_+=\{d+1,j\} \,, & L_-=\{1, i\} \,, \\
     L_*=[d+1]\setminus (L_+\cup L_-) \,, & L_{-1}=L_{+1}=L_0=\emptyset \enspace.
   \end{array}
 \]
 Then, $\bisector_{F,G}(a,b)$ is not empty since $\gamma=(b_i+b_j)/2$, and $\delta= \dist(a,b)$ is a solution of \eqref{eq:system}.
 However, the system for $b'$ is infeasible, since $b'_i<b'_j$. Thus, $\bisector_{F,G}(a',b')$ is empty.

\textbf{Case III:} \emph{Suppose that
  $b$ and $b'$ have exactly the same maxima and minima and the same ordered partitions but the midvalue does not coincide.}

   As before, we assume without loss of generality that $1$ and $d+1$ are a minimum and a maximum, respectively, of both $b$ and $b'$.
  Then, there is an index $i\in [d+1]\setminus \{1,d+1\}$ such that $\midv(b)  \leq b_i$ but $\midv(b') > b'_i$ (or vice-versa, but that would give an equivalent case).

  In this case, we let $F$ and $G$ be the faces with $F_+=\{d+1\}$, $F_-=\{1\}$, $G_+=\{1\}$ and $G_-=\{i\}$.
  These faces produce
  \[
   \begin{array}{ll}
     L_+=\{d+1, i\} \,, & L_{-1}=\{1\} \,, \\
     L_*=[d+1]\setminus (L_+\cup L_{-1}) \,, & L_{-}=L_{+1}=L_0=\emptyset \enspace.
   \end{array}
 \]
  Then, $\bisector_{F,G}(a,b)$ is not empty since $\gamma=\midv(b)$, $\delta= \dist(a,b)/2$ is a solution of  \eqref{eq:system} for $b$

  However, the system for $b'$ is infeasible. This is because \eqref{eq:system} specifies that $\gamma+\delta \geq b'_{d+1}$, and $\gamma-\delta = b'_1$.
  Adding them together and dividing by two we get $\gamma \geq (b'_1+b'_{d+1})/2= \midv(b')$. We also need by \eqref{eq:system} and $i\in L_+$  that $\gamma\leq b'_i$.
  Then, $\midv(b)\leq b'_i$, which contradicts our assumption.
%\qed
\end{proof}

\begin{corollary}
  \label{cor:d^4}
  The tropical bisector of two points in general position has $\Theta(d^4)$ maximal cells.
\end{corollary}

\begin{proof}
  The upper bound is trivial, each maximal cell corresponds to a choice of a pair of facets $(F,G)$ from the tropical ball.
  For the lower bound, assume without loss of generality that $a=0$ and $b_1< b_2 < \dots < b_{d+1}$.
  Then, for each choice of $i,j,k,\ell \in \{1,\dots,d+1\}$, all different and with $\max\{j,\ell\}<\min\{i,k\}$, let $F_+=\{i\}$,  $F_-=\{j\}$,  $G_+=\{k\}$,  $G_-=\{\ell\}$.
  By Claim 1 in the proof of Lemma~\ref{lem:ccfan-part1} the set $\bisector_{(F,G)}(a,b)$ is not empty.
  Since $a$ and $b$ are in general position and $F$ and $G$ are facets of $\ball{d}$, we have $\dim\bisector_{(F,G)}(a,b)=d-1$.
  There are $4\binom{d+1}{4}$ ways of choosing such $\{i,j,k,\ell\}$.
%\qed
\end{proof}

\subsection{The structure of tropical Voronoi regions}
\label{sec:regions}
A \emph{polytrope} is an ordinary polytope which is also convex in the tropical sense (with respect to $\min$ and $\max$ simultaneously); see \cite{JoswigKulas:2008}.
These are precisely the ordinary polytopes whose facets normals are roots of type $A_d$, i.e., $e_i-e_j$ for $i\neq j$; they generalize the \enquote{alcoved polytopes} of Lam and Postnikov \cite{LamPostnikov05}.
Here we relax this notion by also calling a not necessarily bounded ordinary polyhedron a \emph{polytrope} if its facets normals are roots of type $A_d$; this was called a \enquote{weighted digraph polyhedron} in \cite{JoswigLoho:2016}.

The tropical unit ball $\ball{d}$ is a polytrope. But a more important example for us are the polytropes $Q= \bigcap_{a\in S}(a_+F_a)$, where $F_a \in \balld$ for each $a\in S$. Recall from Section~\ref{sec:polyhedral_norms} that in such a $Q$ bisectors of subsets of $S$ agree with affine subspaces. Thus:

\begin{lemma}
\label{lemma:star}
For each polytrope $Q$ as above and $a\in S$, the set $Q\cap \Vor_S(a)$ is the intersection of $Q$ with ordinary affine halfspaces with facet normal $e_i - e_j - e_k + e_\ell$, where $i$ and $j$ are fixed.
\end{lemma}

\begin{proof}
Let $i$ and $j$ be the coordinates maximized and minimized in $F_a$, respectively. For each $b\in S\setminus a$, the condition for $x$ to be closer to $a$ than to $b$ is that $x_i - a_i - x_j + a_j \leq x_k - b_k - x_\ell + b_\ell$, where $k$ and $\ell$ are the coordinates corresponding to $F_b$; see Proposition~\ref{prop:hypersurface}.
%\qed
\end{proof}

We call the intersection of a (possibly unbounded) polytrope with ordinary affine halfspaces with facet normal $e_i - e_j - e_k + e_\ell$, where $i$ and $j$ are fixed, a \emph{semi-polytrope} of type $(i,j)$.
A semi-polytrope in $\torus{d+1}\cong\R^d$ has at most $2\tbinom{d+1}{2}$ facets, since there are at most $(d+1)d$ vectors $e_i - e_j - e_k + e_\ell$ for $k\neq\ell$ and fixed $(i,j)$, plus the (at most) $(d+1)d$ facets of a polytrope.

% This $k$ and $l$ label the cone $F$ of the fan $a + \balld$ that contains the half-polytope we are looking at, where $a$ is the site whose Voroni region we are decomposing.

A set $X\subset\torus{d+1}\cong\R^d$ is \emph{star convex} with center $c$ if for any point $x\in X$ the ordinary line segment $[c,x]$ is contained in~$X$.
Clearly any convex set is star convex, but the converse does not hold.
Star convex sets are contractible.
Despite the many differences to Euclidean Voronoi diagrams, the following result expresses a key similarity.

\begin{theorem}\label{thm:star}
  Let $S\subset\torus{d+1}$ be a finite set in weak general position.
  Then each tropical Voronoi region of $S$ is the star convex union of finitely many (possibly unbounded) semi-polytropes.
%  \michael{This should hold without the assumption on weak general position; but this does not directly follow from Theorem~\ref{thm:central_projection}}
\end{theorem}

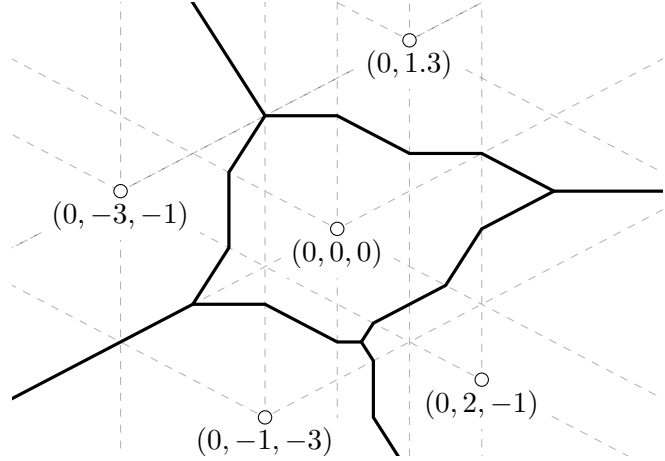
\begin{figure}\centering
  \begin{tikzpicture}[scale=1, x={(-0.95cm,-0.5cm)},y={(0.95cm,-0.5cm)}, z={(0cm,1cm)}]
    \tikzset{tvd/.style = { very thick }}
    \tikzset{partition/.style = { dashed, opacity=0.3 }}
    \clip (0, -4.5, -5.25) -- (0,-4.5,0.75) -- (0, 4.5, 5.25) -- (0,4.5, -0.75) --cycle;

    \draw[tvd] (0,0,-1.5) -- (0,-1,-1.5) -- (0, -2, -2) -- (0,-1.5, -1) -- (0,-1.5, 0) -- (0,-1,1) -- (0,0,1.5) -- (0,1,1.5) -- (0,2,2) -- (0,3,2) -- (0,2,1) -- (0,1.5,0) -- (0,0.5, -1) -- (0, 0.3333, -1.3333) -- cycle;
    \draw[tvd] (0,-1,1) -- (0,-51, 51);
    \draw[tvd] (0,-2,-2) -- (100, -2,-2);
    \draw[tvd] (0,3,2) -- (0,103, 52);
    \draw[tvd] (0, 0.3333, -1.3333) -- (0, 0.5, -1.5) -- (0, 0.5, -2.25) -- (0, 1, -2.75) -- (0,1,-102.75);

    \draw[partition] (0,0,-100) -- (0,0,+100);
    \draw[partition] (0,-100,0) -- (0,+100, 0);
    \draw[partition] (-100,0,0) -- (100,0,0);

    \draw[partition] (0,1,3-100) -- (0,1,3+100);
    \draw[partition] (0,1-100,3) -- (0,1+100, 3);
    \draw[partition] (-100,1,3) -- (100,1,3);

    \draw[dashed, opacity=0.3] (0,-3,-100-1) -- (0,-3,+100-1);
    \draw[dashed, opacity=0.3] (-100,-3,-1) -- (100,-3,-1);

    \draw[dashed, opacity=0.3] (0,-1,-100-3) -- (0,-1,+100-3);
    \draw[dashed, opacity=0.3] (0,-100-1,-3) -- (0,+100-1, -3);
    \draw[dashed, opacity=0.3] (-100,0-1,-3) -- (100,-1,-3);

    \draw[dashed, opacity=0.3] (0,2,-100-1) -- (0,2,+100-1);
    \draw[dashed, opacity=0.3] (0,2-100,-1) -- (0,+100+2, -1);

    \draw (0,0,0) node[anchor=north, fill=white] {$(0,0,0)$};
    \draw[fill=white, draw=black] (0,0,0) circle (2.5pt);
    \draw (0,1,3) node[anchor=north, fill=white] {$(0,1,3)$};
    \draw[fill=white, draw=black] (0,1,3) circle (2.5pt);
    \draw (0,-3,-1) node[anchor=north, fill=white] {$(0,-3,-1)$};
    \draw[fill=white, draw=black] (0,-3,-1) circle (2.5pt);
    \draw (0,-1,-3) node[anchor=north, fill=white] {$(0,-1,-3)$};
    \draw[fill=white, draw=black] (0,-1,-3) circle (2.5pt);
    \draw (0,2,-1) node[anchor=north, fill=white] {$(0,2,-1)$};
    \draw[fill=white, draw=black] (0,2,-1) circle (2.5pt);
  \end{tikzpicture}
  \caption{Tropical Voronoi diagram of five points in $\torus{3}$. The decomposition of Voronoi regions into semi-polytropes is shown by dashed lines}
  \label{fig:tvd}
\end{figure}

% Local Variables: 
% mode: latex
% mode: TeX-PDF
% TeX-master: "bisectors"
% mode: reftex
% mode: font-lock
% buffer-file-coding-system:utf-8-unix
% End: 

\begin{proof}
  That Voronoi regions for polyhedral norms are star-convex is a well-known fact (see \cite[p.~133]{AurenhammerKleinLee:2013} or \cite[p.~127]{MartiniSwanepoel:2004}), which follows for example from Theorem~\ref{thm:central_projection}.
  By Lemma \ref{lemma:star},  $\Vor_S(a)$ decomposes as finitely many semi-polytropes, by intersecting it with the individual polyhedra $Q= \bigcap_{a\in S}(a_+F_a)$, for all choices of $\{F_a\}_{a\in S}$.
%
%  In Proposition~\ref{prop:hypersurface} we saw that the tropical bisector of two sites $a,b$ is contained in the vanishing locus of the homogeneous tropical polynomial $\phi(a,b)$.
%  This entails that, e.g, the Voronoi region of $a$ with respect to $\{a,b\}$ is the union of finitely many ordinary polyhedra.
%  More explicitly, for each choice of facets $F,G\in \Fan(\balld)$, the intersection of  the polytrope $(a+F) \cap (b+G)$ with the Voronoi region of $a$ is given by the affine inequality $x_i - a_i - x_j + a_j \leq x_k - b_k - x_\ell + b_\ell$, where $(i,j)$ are the coordinates maximized and minimized by $F$, and $(k,\ell)$ are those of $G$; see Equation~\eqref{eq:polynomial}.
%  Hence, the Voronoi region of $a$ decomposes, within each cone $a+F$, as a union of semi-polytropes of the type $(i,j)$ given by $F$.
%  As $a$ and $b$ lie in weak general position, it follows from Theorem~\ref{thm:central_projection} that the Voronoi region of $a$ with respect to $\{a,b\}$ is star convex with center $a$.
%  Now the Voronoi region of $a$ with respect to $S$ is the intersection of the Voronoi regions of $a$ with respect to all sets $\{a,b\}$ for $b\in S\setminus\{a\}$.
%  The intersection of finitely many semi-polytropes of a fixed type is a semi-polytrope of the same type, and the intersection of star convex sets with a fixed center is star convex with the same center.
%\qed
\end{proof}

Semi-polytropes are not necessarily tropically convex, and this entails that the regions of a tropical Voronoi diagram are not necessarily tropically convex either; see Figure~\ref{fig:tvd} for an example.
The tropical torus $\torus{d+1}$ is compactified by the tropical projective space $\TP{d}$; the latter is the $\max$-tropical convex hull of the $d+1$ $\max$-tropical unit vectors 
\[ (0,-\infty,-\infty,\dots,-\infty), \ (-\infty,0,-\infty,\dots,-\infty), \,\ldots\,, \ (-\infty,-\infty,\dots,-\infty,0). \]
In this way, $\TP{d}$ may be seen as an infinitely scaled tropical unit ball, which is a polytrope; see \cite[\S3.5]{JoswigLoho:2016}.
Similarly for arbitrary (semi-)polytropes the line between bounded and unbounded is blurred in the compactification.

\section{Computing tropical Voronoi diagrams}
\label{sec:computing}
We will discuss several algorithms.
Some of these methods are similar to their classical Euclidean counterparts, others rely on tailored data structures, which are based on Theorem~\ref{thm:star}.
For the complexity analysis of our algorithms we will consider the dimension as constant.

\subsection{The planar case}
\label{sec:beachline}
There are several methods for computing Euclidean Voronoi diagrams in $\R^2$ with the optimal time complexity $O(n\log n)$ and linear space; see \cite[\S7.2]{Four+Marks}.
This agrees with the situation for planar tropical convex hull computations; see \cite[\S5]{Tropical+halfspaces}.
Chew and Drysdale~\cite{Chew1985voronoi} gave a divide-and-conquer algorithm with the same complexity for planar Voronoi diagrams with respect to arbitrary norms.
Here we sketch a tropical analog of Fortune's beach line algorithm \cite{Fortune:1987}; see also \cite{Widmayer1987some}.

Suppose that we are given a set $S$ of $n$ sites in $\torus{3}$.
In view of Theorem~\ref{thm:star} the tropical Voronoi diagram of $S$ gives rise to a planar graph where vertices are circumcenters of triples of points in $S$, edges are two point bisectors, and faces are Voronoi regions.
We can make this planar embedding piecewise linear by subdividing each bisector into at most five segments; see Figure~\ref{tikz:bisectors}.
The relevant data structure, as in the classical setting, is a doubly-connected edge list which requires $O(n)$ space; see \cite[\S2.2]{Four+Marks}.

The beach line algorithm is based on a line sweep.
The \emph{tropical sweep line} at \emph{time} $t$ in $\torus{3}$ is the set $L(t)=(0,t,0)+\R(0,0,1)+\R\ones$.
Note that $L(t)$ is an ordinary line which is also tropically convex (with respect to $\min$ and $\max$).
For an arbitrary point $x$ we call the set
\[
  P(x,t) \ = \ \SetOf{a\in\torus{3}}{\dist(x,a)=\dist(x,L(t))}
\] 
the \emph{parabola spanned by $x$ and $L(t)$}; here $\dist(x,L(t))=\min\smallSetOf{\dist(x,y)}{y-(0,t,0)\in\R(0,0,1)+\R\ones}$.
This is a $1$-dimensional polyhedral complex, which is homeomorphic with $L(t)$ via orthogonal projection, consisting of five segments.

We will assume that our set $S$ of sites is in general position and hence, in particular, each sweep line contains at most one site.
A point $a=(a_1,a_2,a_3)$ is said to have been \emph{visited} by the sweep line $L(t)$ if $a_2-a_1\leq t$.

The \emph{beach line} $B(t)$ of $S$ at time $t$ is formed by the points $(b_1,b_2,b_3)$ which lie on a parabola $P(s,t)$ for a visited point $s$ such that $b_2-b_1$ is maximal among all such points for a fixed value $b_3-b_1$.
That is, the beach line is formed by the right-most points on the parabolas spanned by the visited points and the sweep line; see Figure~\ref{fig:beachline}.
So $B(t)$ is a union of parabolic arcs; it is easy to see that each parabola contributes at most two arcs to the beach line at any time.
Like a single parabola also the beach line $B(t)$ is homeomorphic to $L(t)$ via orthogonal projection.
In the portion of $\torus{3}$ left to $B(t)$ the tropical Voronoi diagram of $S$ is known at time $t$.

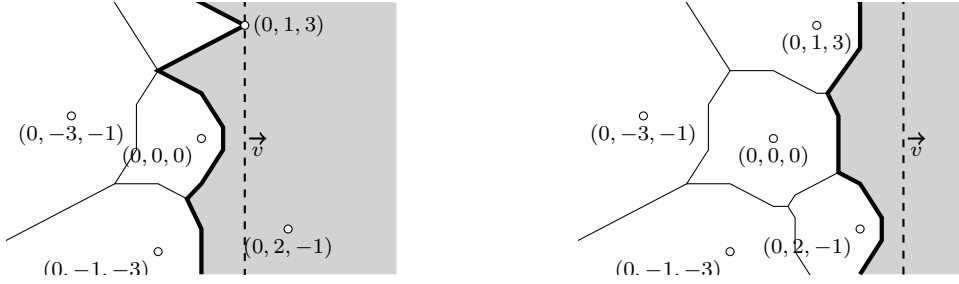
\begin{figure}
  \begin{tikzpicture}[scale=0.6, x={(-0.95cm,-0.5cm)},y={(0.95cm,-0.5cm)}, z={(0cm,1cm)}]
    \scriptsize
    \clip (0, -4.5, -5.25) -- (0,-4.5,0.75) -- (0, 4.5, 5.25) -- (0,4.5, -0.75) --cycle;

    \draw (0,-0.3333, -1.5) -- (0,-1,-1.5) -- (0, -2, -2) -- (0,-1.5, -1) -- (0,-1.5, 0) -- (0,-1,1);

    \draw (0,-1,1) -- (0,-51, 51);
    \draw (0,-2,-2) -- (100, -2,-2);

    % sweepline
    \draw[dashed, thick] (0,1,102) -- (0,1,-98);

    % beachline
    \draw[ultra thick] (0, -1, 3) -- (0,1,3) -- (0,-1,1) -- (0,0,1) -- (0,0.5, 0.5) -- (0,0.5,0)-- (0,0,-1) -- (0,-0.3333, -1.5) -- (0,0,-2) -- (0,0,-3);
    \fill[color=sqsqsq,fill=sqsqsq,fill opacity=0.2] (0, -1, 3) -- (0,1,3) -- (0,-1,1) -- (0,0,1) -- (0,0.5, 0.5) -- (0,0.5,0)-- (0,0,-1) -- (0,-0.3333, -1.5) -- (0,0,-2) -- (0,0,-3) -- (0,4.5,-0.75) -- (0,4.5,5.25) -- cycle;

    \draw [thick, ->] (0,1.1,0.55) -- (0,1.5,0.75);
    \draw (0,1.3,0.65) node[anchor=north] {$v$};

    \draw[fill=white, draw=black] (0,0,0) circle (2.5pt);
    \draw (0,0,0) node[anchor=north east] {$(0,0,0)$};
    \draw[fill=white, draw=black] (0,1,3) circle (2.5pt);
    \draw (0,1,3) node[anchor=west] {$(0,1,3)$};
    \draw[fill=white, draw=black] (0,-3,-1) circle (2.5pt);
    \draw (0,-3,-1) node[anchor=north] {$(0,-3,-1)$};
    \draw[fill=white, draw=black] (0,-1,-3) circle (2.5pt);
    \draw (0,-1,-3) node[anchor=north east] {$(0,-1,-3)$};
    \draw[fill=white, draw=black] (0,2,-1) circle (2.5pt);
    \draw (0,2,-1) node[anchor=north] {$(0,2,-1)$};
  \end{tikzpicture}
  \hfill
  \begin{tikzpicture}[scale=0.6, x={(-0.95cm,-0.5cm)},y={(0.95cm,-0.5cm)}, z={(0cm,1cm)}]
    \scriptsize
    \clip (0, -4.5, -5.25) -- (0,-4.5,0.75) -- (0, 4.5, 5.25) -- (0,4.5, -0.75) --cycle;

    \draw (0, 0.3333,-1.3333) -- (0,0,-1.5) -- (0,-1,-1.5) -- (0, -2, -2) -- (0,-1.5, -1) -- (0,-1.5, 0) -- (0,-1,1) -- (0,0,1.5) -- (0,1,1.5) -- (0,1.25,1.625);

    \draw (0,1.5,0) -- (0,0.5, -1) -- (0, 0.3333, -1.3333);
    \draw (0,-1,1) -- (0,-51, 51);
    \draw (0,-2,-2) -- (100, -2,-2);
    \draw (0, 0.3333, -1.3333) -- (0, 0.5, -1.5) -- (0, 0.5, -2.25) -- (0, 1, -2.75) -- (0,1,-102.75);

    % sweepline
    \draw[dashed,thick] (0,3,102) -- (0,3,-98);

    % beachline
    \draw[ultra thick] (0,2,-2) -- (0,2.5,-1) -- (0,2.5,-0.5) -- (0,2,0) -- (0,1.5,0) -- (0,1.5,1.25) -- (0,1.25,1.625) -- (0,2,3) -- (0,2,4) -- (0,1,5) ;

    \fill[color=sqsqsq,fill=sqsqsq,fill opacity=0.2] (0,2,-2) -- (0,2.5,-1) -- (0,2.5,-0.5) -- (0,2,0) -- (0,1.5,0) -- (0,1.5,1.25) -- (0,1.25,1.625) -- (0,2,3) -- (0,2,4) -- (0,1,5) -- (0,4.5,5.25) -- (0,4.5,-0.75) -- cycle;

    \draw [thick, ->] (0,3.1,1.55) -- (0,3.5,1.75);
    \draw (0,3.3,1.65) node[anchor=north] {$v$};

    \draw[fill=white, draw=black] (0,0,0) circle (2.5pt);
    \draw (0,0,0) node[anchor=north] {$(0,0,0)$};
    \draw[fill=white, draw=black] (0,1,3) circle (2.5pt);
    \draw (0,1,3) node[anchor=north] {$(0,1,3)$};
    \draw[fill=white, draw=black] (0,-3,-1) circle (2.5pt);
    \draw (0,-3,-1) node[anchor=north] {$(0,-3,-1)$};
    \draw[fill=white, draw=black] (0,-1,-3) circle (2.5pt);
    \draw (0,-1,-3) node[anchor=north east] {$(0,-1,-3)$};
    \draw[fill=white, draw=black] (0,2,-1) circle (2.5pt);
    \draw (0,2,-1) node[anchor=north east] {$(0,2,-1)$};
  \end{tikzpicture}
  \caption{The beach line and the sweep line, for $t=1$ (left) and $t=3$ (right).}
  \label{fig:beachline}
\end{figure}

% Local Variables: 
% mode: latex
% mode: TeX-PDF
% TeX-master: "bisectors"
% mode: reftex
% mode: font-lock
% buffer-file-coding-system:utf-8-unix
% End: 

\begin{observation}
  The beach line is a polygonal line with $O(n)$ segments.
\end{observation}

The actual algorithm works as in the classical case.
We maintain a priority queue of \emph{site events} (when the sweep line visits a site) and \emph{circle events} (when there is a candidate for a new vertex of the tropical Voronoi diagram).
The total number of events is linear in $n$.
As in the classical case, it is possible to relax the condition on general position by means of symbolic perturbation.

\begin{theorem}
  The beach line algorithm computes a tropical Voronoi diagram of $n$ sites in $\torus{3}$ in $O(n\log n)$ time and $O(n)$ space.
\end{theorem}

For the output we can choose, with the same complexity, between an abstract planar graph (encoding $\Vor(S)$ topologically) and its piecewise linear embedding resulting from Theorem~\ref{thm:star}.

\subsection{Polytrope partitions}
\label{sec:tmap}

Let $S \subseteq \torus{d+1}$ be a finite set of sites.
From Theorem~\ref{thm:star} we know that the tropical Voronoi diagram can be described in terms of (semi-)polytropes.
For the definition and basic facts on polytropes, see Section~\ref{sec:regions} and \cite{JoswigKulas:2008}.
The following takes inspiration from the \emph{trapezoid map} data structure; see \cite[\S6]{Four+Marks}.

\begin{definition}
  \label{def:polytrope_map}
  A \emph{polytrope partition} for $S$ is a finite collection $\tmap$ of (perhaps unbounded) non-degenerate polytropes with disjoint interiors, covering $\torus{d+1}$, such that:
  \begin{enumerate}
  \item each facet-defining hyperplane of any cell in $\tmap$ lies in the hyperplane arrangement $S+A_d$.
  \item\label{it:polytrope-decomp:max} for each cell $P$ in $\tmap$ and site $a\in S$ the restricted Voronoi region $\Vor_S(a) \cap P$ is contained in a maximal cone $a+F$ of $a+\Fan(\balld)$.
  \end{enumerate}
	A \emph{valid labeling} for $\tmap$ assigns to each cell $P\in \tmap$ a (partial) matching $\mathcal{L}_{\tmap}(P)$ of $S \times \Fan(\ball{d})$ containing
\[ \label{eq:matchings}
    \SetOf{(a,F)\in S \times \Fan(\ball{d})}{(a+F) \cap P\cap \Vor_S(a) \neq \emptyset}.
  \]
\end{definition}
% \pacoc{In the algorithm, we do not require minimality of the labeling, but we do require the matching property. Maybe we should redefine this part?}
%
%  \item each cell $C$ in $\tmap$ is \emph{labelled} by a subset of
%  \[
%  \{(p,F)\in S \times \Fan(\ball{d}) : C \subset p+F\}.
%  \]
%  We require the label of $C$ not to repeat the same $F\in \Fan(\ball{d}) $ twice (that is, it is a matching in  $S \times \Fan(\ball{d})$) and to contain all (but maybe not only) the pairs $(p,F)$ such that $C\cap \Vor_S(p) \ne \emptyset$.

The distance function $x\mapsto \dist(x,a)$ to a fixed site $a$ is piecewise linear, as it is linear in each translated cone $a+F$ for $F\in\Fan(\balld)$. Polytrope partitions are designed to exploit this fact, so that the distances to the relevant sites in each cell are linear. In particular:

\begin{observation}
  \label{obs:dome}
  Let $P$ be a cell in a polytrope partition $\tmap$ for $S$.
  Then for all $x\in P$ we have
  \begin{equation}\label{eq:dome:lineardist}
    \dist(x, S) \ = \ \min_{a\in S} \lineardist_{F_a}(x-a)
  \end{equation}
  where $\lineardist_{F_a}$ is the linear function defined by restricting the distance to $a$ on some maximal cone $a+F_a$ of $a+\Fan(\ball{d})$ which contains $\Vor_S(a)\cap P$.
  Thus computing the restriction of $\Vor(S)$ to the polytrope~$P$ amounts to finding the regions of linearity of the tropical polynomial $\min_{a\in S} \lineardist_{F_a}(x)$.
  The latter can be obtained via an ordinary dual convex hull computation.
\end{observation}

Note that the maximal cone $F_a$ in the above is irrelevant if $\Vor_S(a)\cap P= \emptyset$, and it is unique, by axiom \eqref{it:polytrope-decomp:max}, if $\Vor_S(a)\cap P \neq \emptyset$. If the polytrope partition is equipped with a valid labeling, then this tells us the choice of the right cone $F_a$ for each site, if it exists.

The following shows that polytrope partitions exist.

\begin{example}\label{exmp:standard}
  The braid arrangement $A_d$ consists of the $\tbinom{d+1}{2}$ ordinary hyperplanes $\smallSetOf{x}{x_i=x_j}$, where $i\neq j$.
  This gives rise to the \emph{standard polytrope partition} $S + A_{d}$, which is finer than any other polytrope partition for $S$; Figure~\ref{fig:tvd} shows an example for $d=2$.
  This construction occurs in planar tropical convex hull algorithms; see \cite[Figure~3]{Tropical+halfspaces}.
\end{example}

And finally, the following lemma shows that valid labelings exist for every polytrope partition, if the sites $S$ are in weak general position:

\begin{lemma}
  \label{lem:labeling}
  Let $\tmap$ be a polytrope partition for $S$.
  If $S$ is in weak general position then there is a valid labeling of $\tmap$.
  Moreover, for $d$ considered constant, a labeling of each polytropal cell has constant size, and it can be computed in $O(n)$ time.
\end{lemma}

\begin{proof}
	Suppose that a valid labeling does not exist for some cell $P$. Either the set of pairs in \eqref{eq:matchings} matches two cones with the same site, or matches two sites with the same cone of $\Fan(\ball{d})$. The former cannot happen since $P$ is full-dimensional and the cones in $\Fan(\balld)$ only intersect in lower-dimensional polyhedral cones.

  Then there are sites $a,b\in S$ and a maximal cone $F \in\Fan(\ball{d})$ such that the sets $(a+F) \cap P\cap \Vor_S(s)$ and $(b+F)\cap P\cap\Vor_S(t)$ both are non-empty.

  With the notation of \eqref{eq:dome:lineardist} we have $\lineardist_{F_a}=\lineardist_{F_b}$; and we shortly write $\lineardist$.
  Since the sites are in weak general position, we may assume that $\lineardist(b)>\lineardist(a)$.
  Picking $y\in(b+F)\cap P\cap\Vor_S(b)$ yields
  \[
    \dist(y,b) \ \geq \ \lineardist(b) \ > \ \lineardist(a) \ \geq \ \dist(y,S) \enspace,
  \]
  where the last inequality follows from \eqref{eq:dome:lineardist}.
  The resulting inequality $\dist(y,b)>\dist(y,S)$ implies that $y\not\in\Vor_S(b)$, which is a contradiction.
  Hence a valid labeling does exist.

  To compute such labeling, we iterate through all the sites.
  For each site $a$, the candidate facet of $F_a\in\Fan(\balld)$ is known by definition of the polytrope partition.
  To check if $(a,F_a)$ is a labeling candidate, we need to determine if $(a+F_a)\cap P \cap \Vor_S(a)$ is empty or not.
  This amounts to solving a linear program that has constant size (as $d$ is a constant).
  It follows that the entire labeling can be computed in $O(n)$ time.
%\qed
\end{proof}

We aim at a first algorithm for computing a tropical Voronoi diagram in arbitrary dimension.
This will employ the standard polytrope partition from Example~\ref{exmp:standard}.

\begin{lemma}\label{lem:number-of-cells}
  If $S$ is in weak general position and has size $n$ then the trivial polytrope partition has
  \[
    (d+1)^{d-1} n^d + O(n^{d-1})
  \]
  maximal cells, if we consider $d$ a fixed constant.
\end{lemma}

\begin{proof}
  Pick a generic direction $v\in \torus{d+1}$.
  The cells of the polytrope partition that are bounded in the direction of $v$ are in a one-to-one correspondence with the vertices of the arrangement, by associating each polytrope with the optimum of the linear program maximizing $v^Tx$.
  Since the number of vertices equals
  \[
    (d+1)^{d-1} n^d - n \left((d+1)^{d-1} -1\right),
  \]
  by Cayley's formula, it suffices to show that the number of unbounded cells is in $O(n^{d-1})$, as the bases of $A_d$ correspond to spanning trees in the complete graph with $d$ vertices.
  The unbounded cells intersect a hyperplane, $H$, normal to $v$ that is far enough in the $v$ direction.
  The cells intersecting $H$ are the same as the cells in the restricted hyperplane arrangement, which is a $(d-1)$-dimensional arrangement with $N={d+1 \choose 2}n$ hyperplanes.
  The number of such cells is known to be in $O\left(N^{d-1}\right)$, which agrees with $O\left(n^{d-1}\right)$ as $N$ depends linearly on $n$.
%\qed
\end{proof}

\begin{remark}[Standard polytrope partition algorithm]
  \label{rem:standard}
  This directly yields a first algorithm for computing a tropical Voronoi diagram of $n$ sites in $\torus{d+1}$ in $O(n^{d+1})$ time, as follows:
  First, we sort $S$ along each of the ${(d+1)\choose 2}$ directions $e_i+e_j$, in $O(n\log n)$ time.
  As in the proof of Lemma~\ref{lem:number-of-cells} we pick a generic direction $v\in\torus{d+1}$.
  We can compute the vertices of the hyperplane arrangement $A_d+S$ in time $O(n^d)$ by enumerating all $d$-sets of independent directions, which can be derived from the the oriented spanning tree of $K_{d+1}$, in constant time.
  For each of the $d$ directions we choose an index $i\in [n]$.

  Next we perturb each such vertex $p$ by a small multiple of $-v$, and we collect the intersection of bands of contiguous parallel hyperplanes of $A_d+S$ that contain the perturbed point.
  This can be done in time $O(\log n)$ for each direction.
  In this way, we find those cells which are bounded in the direction of this particular $v$ in linear output-dependent time.
  We repeat the same procedure for a set of directions $v_1,\dots, v_{d+1}$ which positively span the entire space $\torus{d+1}$.
  Each polytropal cell will be bounded in at least one of these directions, and thus their enumeration is still in $O(n^d)$ for $d$ fixed.

  Then, for each polytrope $P$, we compute a corresponding labeling in time $O(n)$, by Lemma~\ref{lem:labeling}.
  Therefore, we can compute the standard polytrope partition, including labels, in time $O(n^{d+1})$ for fixed $d$.
  The tropical Voronoi diagram in each cell is an ordinary dual convex hull problem of constant size.
  This computation splits each polytrope in the partition into semi-polytropes.
  The convex hull problem can be solved in constant time, and hence this algorithm takes $O(n^{d+1})$ time, if $d$ is considered a fixed constant.

  We implemented this procedure in \polymake \cite{DMV:polymake}, version 3.6.
\end{remark}

\begin{question}
  In the plane $\torus{3}$, we believe that ideas similar to the \enquote{trapezoidal maps} used in point location, see \cite[\S6.1]{Four+Marks}, should yield polytrope partitions of linear size but we did not work out the details.
  More generally: Is there a polytrope partition of complexity better than $\Theta(n^d)$ in arbitrary dimension?
  One could hope for something in $O(n^{d/2})$, which is the worst-case complexity of Euclidean Voronoi diagrams.
%  Note that the common refinement of the face fans of $\balld$ has $O(n^d)$ number of cells in fixed dimension, albeit with a much smaller constant in $d$; each fan has $O(d^2)$ cones instead of $O(d!)$.
\end{question}

\subsection{An $O(n^d\log n)$ randomized incremental algorithm in $\torus{d+1}$}
\label{sec:incremental}

We can improve the algorithm from Remark~\ref{rem:standard} by constructing a polytrope partition incrementally.
The idea is to update an existing polytrope partition by including a new point and to employ randomization to improve the efficiency.
Moreover, we will also produce a coarser polytrope partition than the standard one, but only by a constant in $d$.

% In this section, we give a randomized (Las Vegas) algorithm for computing the polytrope partition $S+\Fan(\balld)$, with an appropiate labeling. Note that this is coarser than the trivial polytrope partition, but asymptotically it has the same number of cells for fixed $d$, namely $\Theta(n^d)$.

A key ingredient is a new data structure that we call a \emph{polytrope tree}.
Throughout we assume that the set $S$ of $n$ sites forms a subset of $\torus{d+1}$ in general position.
We fix the polytrope partition $\tmap:=S+\Fan(\balld)$, which is coarser than the standard polytrope partition but only by a factor which is constant in $d$; see Example~\ref{exmp:standard}.

\begin{definition}
  A \emph{polytrope tree} for $S$ is a (rooted) tree $T$ such that
  \begin{enumerate}
  \item for each leaf $\ell$ there is a polytropal cell $P(\ell)$ of $\tmap$;
  \item for each interior node $i$ there is a site $a(i)\in S$ and a polytrope $P(i)$.
  \end{enumerate}
  These satisfy the following consistency conditions:
  \begin{itemize}
  \item for the root node $r$ of $T$ we have $P(r)=\torus{d+1}$, which may be seen as an unbounded polytrope;
  \item the map $\ell\mapsto P(\ell)$ is a bijection between the set of leaves of $T$ and the set of polytropes in $\tmap$;
  \item the map $i\mapsto a(i)$ is a surjection from the interior nodes onto the set $S$;
  \item if $i$ is an interior node with children $c_1,\dots,c_k$, then $P(c_1),\dots,P(c_k)$ form the maximal cells of $(a(i)+\Fan(\balld))\cap P(i)$.
  \end{itemize}
\end{definition}

It is easy to construct a polytrope tree for $S$, and its purpose is to speed up the computation of a valid labeling.
This will reduce the algorithmic complexity from $O(n^{d+1})$ to $O(n^d\log n)$.

For the incremental update to insert a new site $b\not\in S$ we maintain a stack $\Sigma$ of unvisited nodes in a given polytrope tree for $\Sigma$ and process it as follows:
\begin{itemize}
\item the stack $\Sigma$ is initialized with the root node $r$;
\item we remove the top node $q$ from the stack $\Sigma$ unless it is empty;
\item if $q$ is an interior node such that $P(q)$ intersects more than one maximal cone of $s+\Fan(\balld)$, then we push the children of $q$ onto the stack $\Sigma$;
\item if $q$ is a leaf such that $P(\ell)$ intersects more than one maximal cone of $p+\Fan(\balld)$, then we create the intersections of $P(\ell)$ with $s+\Fan(\balld)$ as new leaves, which now become children of $q$, and we set $a(q)\leftarrow b$.
\end{itemize}
Note that an interior node $q$ with $P(q)$ contained in a unique maximal cone of $a+\Fan(\balld)$ is kept unchanged, and its children will not be visited.
The following is the essential part of the complexity analysis.

\begin{proposition}
  \label{prop:ptree}
  Let $T$ be a polytrope tree created in the way explained above, where the $n$ sites in $S$ are processed in uniformly random order.
  Then the expected height of $T$ is of order $O(\log n)$, if $d$ is considered a fixed constant.
\end{proposition}

\begin{proof}
  Let $P$ be a polytrope in the polytrope partition $\tmap(S)$.
  For each ordering $\pi: [n] \to S$ of $S$ we have a polytrope tree $T(S,\pi)$ with $P$ as a leaf.
  By induction on $n$ we will show:
  \begin{equation}
  \label{eq:induction}
    E[h_{T(S,\pi)}(P)] \ \leq \ d(d+1) \sum_{i=1}^n \frac {1}{i} \quad \in \ O(\log n)\enspace,
  \end{equation}
  where the expectation $E[\cdot]$ is taken uniformly over all $n!$ orderings of $S$, and $h_{T(S,\pi)}(P)$ is the depth of the leaf of $P$ in $T(S,\pi)$.

  We proceed by backwards analysis.
  Let $S'\subset S$ be the subset of sites that lie in some facet-defining hyperplane of $P$. Since $P$ has at most $d(d+1)$ facets and (by general position) their corresponding hyperplanes contain each exactly one point of $S$, we have $|S'| \leq d(d+1)$. Thus, the probability that the height $h_T(P)$ increases in the last insertion is at most $d(d+1)/n$. Since the increase is by exactly one, we have
  \[
  E[h_{T(S,\pi)}(P)] \ \leq \ E[h_{T(S\backslash\pi(n),\pi_{[n-1]})}(P')] + \frac{d(d+1)}{n} \enspace,
  \]
  where $P'$ is the polytrope containing $P$ in the polytrope partition before the last insertion.
  By induction hypothesis
  \[
    E[h_{T(S\backslash\pi(n),\pi_{[n-1]})}(P')] \ \leq \ d(d+1) \sum_{i=1}^{n-1} \frac {1}{i} \enspace.
  \]
  The last two formulas give Eq.~\eqref{eq:induction}.
%\qed
\end{proof}

\begin{corollary}
  \label{coro:ptree}
  The above method constructs a polytrope tree for the polytrope partition $S+\Fan(\balld)$ in expected time $O(n^d\log n)$ and space $O(n^d)$, for $d$ constant.
\end{corollary}

\begin{proof}
  The algorithm that inserts a new site $a$ into the tree only visits nodes that are above some leaf requiring an update.
  For each such leaf $\ell$ the polytrope $P(\ell)$ intersects one of the $d(d+1)$ hyperplanes in $a+A_d$.
  This implies that there are $O(n^{d-1})$ of them.
  Since the expected depth of every leaf is $O(\log n)$ it requires expected time $O(n^{d-1} \log n)$ for inserting $a$.
  Hence the total complexity for $n$ sites amounts to $O(n^d \log n)$.
%\qed
\end{proof}

In order to compute the tropical Voronoi diagram, we also need to compute the labeling of this polytrope partition.
The naive way is to compute the labeling for each leaf as we did in Remark~\ref{rem:standard}.

A slight improvement is to compute the labeling during the depth-first-search (DFS) exploration of the tree at each insertion of a new site.
But in this way, even if an interior  node is completely contained in only one cone of the fan $a+\Fan(\ball d)$, we need to descend to its subtree in order to update the labels.
This would slow the algorithm down to $\Theta(n^{d+1})$ because each insertion will have to iterate through all the leaves.

A better way is to compute the labeling lazily.
To this end we equip each interior node $i$ with a partial labeling $\mathcal{L}_\tmap(i)$.
With each new insertion, we proceed as we just explained, but we do not cascade down the label updates.
Only once all sites in $S$ have been inserted we cascade the partial labels, updating them in DFS order.
This takes $O(n^d\log n)$ time to compute the polytropes and the lazy labelings, plus $O(n^d)$ time to cascade the lazy labelings down in the tree, for a total time complexity of $O(n^d\log n)$ time.
This gives our final result.

\begin{theorem}\label{thm:tree}
  There is a randomized incremental algorithm for computing tropical Voronoi diagrams of $n$ sites in $\torus{d+1}$ in general position with expected time complexity $O(n^d\log n)$ and space complexity $O(n^d)$, for $d$ constant.
\end{theorem}

\bibliographystyle{amsplain}
\bibliography{TropicalVoronoiDiagrams}

\providecommand{\bysame}{\leavevmode\hbox to3em{\hrulefill}\thinspace}
\providecommand{\MR}{\relax\ifhmode\unskip\space\fi MR }
% \MRhref is called by the amsart/book/proc definition of \MR.
\providecommand{\MRhref}[2]{%
  \href{http://www.ams.org/mathscinet-getitem?mr=#1}{#2}
}
\providecommand{\href}[2]{#2}
\begin{thebibliography}{10}

\bibitem{Alessandrini:2013}
Daniele Alessandrini, \emph{Logarithmic limit sets of real semi-algebraic
  sets}, Adv. Geom. \textbf{13} (2013), no.~1, 155--190. \MR{3011539}

\bibitem{ABGJ:2018}
Xavier Allamigeon, Pascal Benchimol, St\'ephane Gaubert, and Michael Joswig,
  \emph{Log-barrier interior point methods are not strongly polynomial}, {SIAM}
  J. Appl. Algebra Geom. \textbf{2} (2018), no.~1, 140--178.

\bibitem{AminiManjunath:2010}
Omid Amini and Madhusudan Manjunath, \emph{Riemann-{R}och for sub-lattices of
  the root lattice {$A_n$}}, Electron. J. Combin. \textbf{17} (2010), no.~1,
  Research Paper 124, 50. \MR{2729373}

\bibitem{AurenhammerKleinLee:2013}
Franz Aurenhammer, Rolf Klein, and Der-Tsai Lee, \emph{Voronoi diagrams and
  {D}elaunay triangulations}, World Scientific Publishing Co. Pte. Ltd.,
  Hackensack, NJ, 2013. \MR{3186045}

\bibitem{BakerNorine:2007}
Matthew Baker and Serguei Norine, \emph{Riemann-{R}och and {A}bel-{J}acobi
  theory on a finite graph}, Adv. Math. \textbf{215} (2007), no.~2, 766--788.
  \MR{2355607}

\bibitem{Chew1985voronoi}
L.~Paul Chew and Robert L.~Scot Dyrsdale~III, \emph{Voronoi diagrams based on
  convex distance functions}, Proceedings of the first annual symposium on
  Computational geometry, ACM, 1985, pp.~235--244.

\bibitem{Four+Marks}
Mark de~Berg, Otfried Cheong, Marc van Kreveld, and Mark Overmars,
  \emph{Computational geometry}, third ed., Springer-Verlag, Berlin, 2008,
  Algorithms and applications. \MR{2723879}

\bibitem{DepersinGaubertJoswig:2017}
Jules Depersin, St\'ephane Gaubert, and Michael Joswig, \emph{A tropical
  isoperimetric inequality}, S\'em. Lothar. Combin. -- FPSAC 2017 --
  \textbf{78B} (2017), Art. 27, 12. \MR{3678609}

\bibitem{Develin:2005}
Mike Develin, \emph{The moduli space of {$n$} tropically collinear points in
  {$\R^d$}}, Collect. Math. \textbf{56} (2005), no.~1, 1--19. \MR{2131129}

\bibitem{Fortune:1987}
Steven Fortune, \emph{A sweepline algorithm for {V}orono\u\i\ diagrams},
  Algorithmica \textbf{2} (1987), no.~2, 153--174. \MR{MR895442 (88e:68101)}

\bibitem{DMV:polymake}
Ewgenij Gawrilow and Michael Joswig, \emph{\polymake: a framework for analyzing
  convex polytopes}, Polytopes---combinatorics and computation (Oberwolfach,
  1997), DMV Sem., vol.~29, Birk\-h\"au\-ser, Basel, 2000, pp.~43--73.
  \MR{MR1785292 (2001f:52033)}

\bibitem{MR3135681}
Chan He, Horst Martini, and Senlin Wu, \emph{On bisectors for convex distance
  functions}, Extracta Math. \textbf{28} (2013), no.~1, 57--76. \MR{3135681}

\bibitem{IKLM95}
Christian Icking, Rolf Klein, Ng\d{o}c-Minh L\^{e}, and Lihong Ma, \emph{Convex
  distance functions in {$3$}-space are different}, Fund. Inform. \textbf{22}
  (1995), no.~4, 331--352, Computational geometry (San Diego, CA, 1993).
  \MR{1360951}

\bibitem{trisectors}
Christian Icking, Rolf Klein, Ngoc-Minh L{\^e}, Lihong Ma, and Francisco
  Santos, \emph{On bisectors for convex distance functions in 3-space}, CCCG,
  1999.

\bibitem{JellScheidererYu:1810.05132}
Philipp Jell, Claus Scheiderer, and Josephine Yu, \emph{Real tropicalization
  and analytification of semialgebraic sets}, 2018, Preprint
  \arXiv{1810.05132}.

\bibitem{Tropical+halfspaces}
Michael Joswig, \emph{Tropical halfspaces}, Combinatorial and computational
  geometry, Math. Sci. Res. Inst. Publ., vol.~52, Cambridge Univ. Press,
  Cambridge, 2005, pp.~409--431. \MR{MR2178330 (2006g:52012)}

\bibitem{JoswigKulas:2008}
Michael Joswig and Katja Kulas, \emph{Tropical and ordinary convexity
  combined}, Adv. Geometry \textbf{10} (2010), 333--352.

\bibitem{JoswigLoho:2016}
Michael Joswig and Georg Loho, \emph{Weighted digraphs and tropical cones},
  Linear Algebra Appl. \textbf{501} (2016), 304--343.

\bibitem{KosowskyYuille:1994}
J.~J. Kosowsky and Alan~L. Yuille, \emph{The invisible hand algorithm: Solving
  the assignment problem with statistical physics}, Neural Networks \textbf{7}
  (1994), no.~3, 477--490.

\bibitem{LamPostnikov05}
Thomas Lam and Alexander Postnikov, \emph{Alcoved polytopes. {I}}, Discrete
  Comput. Geom. \textbf{38} (2007), no.~3, 453--478. \MR{MR2352704}

\bibitem{LinMonodYoshida:1805.12400}
Bo~Lin, Anthea Monod, and Ruriko Yoshida, \emph{Tropical foundations for
  probability \& statistics on phylogenetic tree space}, 2018, Preprint
  \arXiv{1805.12400}.

\bibitem{Tropical+Book}
Diane Maclagan and Bernd Sturmfels, \emph{Introduction to tropical geometry},
  Graduate Studies in Mathematics, vol. 161, American Mathematical Society,
  Providence, RI, 2015. \MR{3287221}

\bibitem{MartiniSwanepoel:2004}
Horst Martini and Konrad~J Swanepoel, \emph{The geometry of minkowski spaces--a
  survey. part ii}, Expositiones mathematicae \textbf{22} (2004), no.~2,
  93--144.

\bibitem{Mikhalkin:2005}
Grigory Mikhalkin, \emph{Enumerative tropical algebraic geometry in {$\Bbb
  R^2$}}, J. Amer. Math. Soc. \textbf{18} (2005), no.~2, 313--377. \MR{2137980
  (2006b:14097)}

\bibitem{SpeyerSturmfels04}
David Speyer and Bernd Sturmfels, \emph{The tropical {G}rassmannian}, Adv.
  Geom. \textbf{4} (2004), no.~3, 389--411. \MR{MR2071813 (2005d:14089)}

\bibitem{HDCG:3:topological+methods}
Rade~T. \v{Z}ivaljevi\'{c}, \emph{Topological methods in discrete geometry},
  Handbook of Discrete and Computational Geometry (Csaba~D. T\'oth, Jacob~E.
  Goodmann, and Joseph O'Rourke, eds.), CRC Press, 2018, 3rd edition.

\bibitem{Widmayer1987some}
Peter Widmayer, Ying-Fung Wu, and Chak-Kuen Wong, \emph{On some distance
  problems in fixed orientations}, {SIAM} Journal on Computing \textbf{16}
  (1987), no.~4, 728--746.

\end{thebibliography}

\end{document}